\documentclass[a4paper,11pt]{article}

\usepackage[latin1]{inputenc}
\usepackage[T1]{fontenc}

\usepackage{array}
\usepackage{fancyhdr}
\usepackage{longtable}
\usepackage{tabularx}
\usepackage{rotating}

\usepackage{makeidx}
\usepackage{amsmath}
\usepackage{amsthm}
\usepackage{amssymb}
\usepackage[mathscr]{eucal}
\usepackage{enumerate}
\usepackage[dvips]{epsfig}
\usepackage{float}  
\usepackage{ifthen} 
\usepackage{extarrows}  
\usepackage{stmaryrd}   
\usepackage{arydshln}     
\usepackage{multirow}
\usepackage{mathtools}
\usepackage[dvips]{geometry}
\usepackage{listings}
\usepackage{subfig}
\usepackage{dsfont}
\usepackage{xspace}
\usepackage{color}
\usepackage{epstopdf}
\usepackage[bookmarksopen,bookmarksnumbered,colorlinks=true,linkcolor=red,citecolor=blue]{hyperref}

\theoremstyle{plain}
\newtheorem{thm}{Theorem}[section]
\newtheorem{prop}[thm]{Proposition}
\newtheorem{lma}[thm]{Lemma}

\theoremstyle{definition}

\theoremstyle{remark}
\newtheorem{remark}{Remark}[section]

\newcommand{\Prob}{\operatorname{P}}

\newcommand{\tr}{\operatorname{tr}}

\newcommand{\costs}{\operatorname{cost}}

\newcommand{\Oo}{\mathcal{O}}

\newcommand{\Y}{Y}

\newcommand{\YY}{Y^{N,K,M}}
\newcommand{\WW}{W^{K,M}}

\newcommand{\sI}{\sum_{l=0}^{m-1} \int_{t_l}^{t_{l+1}}}

\newcommand{\su}{\sum_{l=0}^{m-1}}
\newcommand{\I}{\int_{t_l}^{t_{l+1}}}

\newcommand{\MIL}{\text{MIL}}
\newcommand{\MILA}{\text{MIL-A1}}
\newcommand{\MILB}{\text{MIL-A2}}

\newcommand{\LIE}{\text{LIE}}
\newcommand{\EES}{\text{EXE}}

\newcommand{\DFM}{\text{DFM}}
\newcommand{\DFMA}{\text{DFM-A1}}
\newcommand{\DFMB}{\text{DFM-A2}}
\newcommand{\qd}{q_{\text{DFM}}}
\newcommand{\qee}{q_{\text{EES}}}
\newcommand{\qm}{q_{\text{MIL}}}
\newcommand{\CC}{\bar{c}}

\allowdisplaybreaks

\title{{A Derivative-Free Milstein Type Approximation Method for
SPDEs covering the Non-Commutative Noise case
}}

\author{Claudine von Hallern$^1$%
\thanks{
e-mail: vonhallern@math.uni-kiel.de}
\ \ and Andreas R\"o\ss ler$^2$\thanks{e-mail: roessler@math.uni-luebeck.de}
\bigskip
\\
\small{$^1$Department of Mathematics, Christian-Albrechts-Universit\"at zu Kiel,} \\
\small{Christian-Albrechts-Platz 4, 24118 Kiel, Germany} \\[0.2cm]
\small{$^2$Institute of Mathematics, Universit\"at zu L\"ubeck,} \\
\small{Ratzeburger Allee 160, 23562 L\"ubeck, Germany} 
}

\date{June 14, 2020}

\begin{document}

\maketitle

\begin{abstract}
Higher order schemes for stochastic partial differential 
equations that do not possess 
commutative noise require
the simulation of iterated stochastic integrals. 
In this work, we propose a derivative-free Milstein type scheme 
to approximate the mild solution of stochastic partial differential equations
that need not to fulfill a commutativity condition for the noise term and
which can flexibly be combined with some
approximation method for the involved iterated integrals.
Recently, the authors introduced two algorithms 
to simulate such iterated stochastic integrals; these clear the way for the 
implementation of the proposed higher order scheme.
We prove the mean-square convergence of the introduced derivative-free
Milstein type scheme which attains the same order as the original
Milstein scheme. The original scheme, however, is definitely outperformed
when the computational cost is taken into account
additionally, that is, in terms of the effective order of
convergence. We derive the effective order of convergence 
for the derivative-free Milstein type scheme analytically 
in the case that one of the recently proposed algorithms for the approximation of 
the iterated stochastic integrals is applied.
Compared to the exponential Euler scheme and the original Milstein scheme,
the proposed derivative-free Milstein type scheme
possesses at least the same and in most cases even a higher 
effective order of convergence
depending on the particular SPDE under consideration. 
These analytical results are
illustrated and confirmed with numerical simulations.
\end{abstract}
%
%
%
\section{Motivation}
For the approximation of stochastic partial differential equations (SPDEs)
with commutative noise, some higher order schemes such
as the Milstein schemes in~\cite{MR2996432,MR3027891,MR3320928,MR2677551}, 
the derivative-free versions~\cite{2015arXiv150908427L} and \cite{MR3011387}, or the
Wagner-Platen type scheme~\cite{MR3534472} were derived and implemented in the last years.
Concerning equations that do not need to possess commutative noise,
see \cite{MR2174871,MR3305472,MR2127642,MR2520127,schnoerr2016cox}
for some applications, it was, however, an open question
how to implement a higher order
scheme 
due to the iterated stochastic integrals that are involved and 
the numerical scheme of choice was so far some Euler scheme, for example, 
the exponential Euler or the linear implicit Euler, 
see~\cite{MR2471778,MR1825100,MR3047942}. 
Recently, the authors presented two algorithms
to obtain an approximation of such 
stochastic integrals, see~\cite{MR3949104}. 
In \cite{MCQMC2018}, the Milstein scheme proposed by 
A.~Jentzen and M.~R\"ockner \cite{MR3320928} has been analyzed for 
non-commutative equations in the case that it is 
combined with the algorithms proposed in \cite{MR3949104}. 
However, as the main drawback the Milstein scheme requires the 
evaluation of the derivative of an operator in each time step. This is the 
reason that its computational complexity increases quadratically w.r.t.\ the dimension of the state space
compared to the Euler scheme with linearly growing computational complexity. 
In the present paper, we propose a derivative-free 
numerical scheme to efficiently 
approximate the mild solution of SPDEs which do not need to have 
commutative noise, that is, the commutativity condition 
\begin{equation}\label{Comm}
    \big( B'(v) (B(v) u) \big) \tilde{u} =
    \big( B'(v) (B(v) \tilde{u}) \big) u
\end{equation}
for all $v\in H_{\beta}$, $u, \tilde{u} \in U_0$  has \emph{not} to 
be fulfilled. Our goal is to approximate the mild solution to SPDEs of type
\begin{equation}\label{SPDE}
  \mathrm{d} X_t = \big( AX_t+F(X_t)\big) \, \mathrm{d}t + B(X_t) \, \mathrm{d}W_t, 
  \quad t\in(0,T], \quad X_0 = \xi
\end{equation}
with a scheme that obtains the same temporal order of convergence 
as the Milstein scheme, however, without the need to evaluate any derivative
and with significantly reduced computational complexity which is 
of the same order of magnitude as for the Euler scheme,
 i.e., which depends only linearly on the dimension of the state space.
For details on the notation, we refer to Section~\ref{Sec:Setting}.
In general, the Milstein scheme proposed in \cite{MR3320928} applied to \eqref{SPDE}
reads as $\YY_0 = P_N\xi$ and
\begin{align} \label{Milstein}
  \YY_{m+1} &= P_N e^{Ah} \bigg(\YY_m + hF(\YY_m) 
  + B(\YY_m) \Delta \WW_m \nonumber \\
  &\quad + \int_{t_m}^{t_{m+1}} B'(\YY_m) \Big(\int_{t_m}^{s}
  B(\YY_m) \, \mathrm{d}W^K_r\Big)
  \, \mathrm{d}W^K_s \bigg)
\end{align}
for some $K,M,N \in\mathbb{N}$, $h=\frac{T}{M}$, and $m\in\{0,\ldots,M-1\}$.
Numerical schemes that attain higher orders of convergence involve
iterated stochastic integrals and it is not possible to rewrite these expressions
such as
\begin{equation}\label{IteratedSPDE}
  \int_t^{t+h} B'(X_t) \Big( \int_t^s B(X_t) \, \mathrm{d}W_r^K \Big) \, \mathrm{d}W_s^K
\end{equation}
for $h>0$, $t,t+h\in[0,T]$ and $K\in\mathbb{N}$ 
in terms of increments of the approximated $Q$-Wiener process $(W_t^K)_{t\in[0,T]}$
like in the commutative case, see~\cite{MR3320928}. 
Therefore, methods such as the derivative-free Milstein 
type scheme presented in~\cite{MR3842926}, which was
developed based on this assumption, are not applicable to approximate the
mild solution of these equations.
In~\cite{MR3949104}, we introduced two methods
to approximate iterated stochastic integrals 
\begin{equation}\label{DoubleIntSPDE}
 \int_t^{t+h} \Psi\left(\Phi \int_t^s  \mathrm{d}W_r\right) \mathrm{d}W_s
\end{equation}
with $t\geq 0$, $h>0$
for some operators $\Psi\in L(H,L(U,H)_{U_0})$, $\Phi\in L(U,H)_{U_0}$, and a $Q$-Wiener process
$(W_t)_{t\in[0,T]}$ of trace class. Therewith, it is possible to 
implement the Milstein scheme~\eqref{Milstein} from \cite{MR3320928},  
we refer to \cite{MCQMC2018} for details.
However, the evaluation of the derivative in the Milstein scheme is costly.
Precisely, the computational cost needed to evaluate this term is of
order $\mathcal{O}(N^2 K)$ in each time step, see~\cite{MR3842926,MCQMC2018}. 
This computational effort can be reduced by one order of magnitude
if the derivative is replaced by some customized approximation --
see also the detailed discussion of this issue in~\cite{MR3842926}. \\ \\
In this work, we design a derivative-free numerical scheme to approximate the
mild solution of equation~\eqref{SPDE} which can be combined with any method 
to simulate the iterated stochastic integrals involved in the scheme, see also \cite{Diss}.
First, we introduce the setting in which 
we work and state results on the convergence of the proposed 
scheme -- both, with and without an approximation of the iterated integrals. 
The same theoretical order of convergence
as for the Milstein scheme can be obtained. 
Moreover, we illustrate the advantages of such a higher order derivative-free 
scheme with a concrete example in Section~\ref{Sec:EffOrder}. 
We combine the scheme with Algorithm~1 presented in~\cite{MR3949104}, 
which is based on a truncated Fourier series expansion,
and derive the effective order of convergence for this scheme
-- a concept that combines the theoretical 
order of convergence with the computational effort
based on a cost model introduced in~\cite{MR3842926}.
In terms of this effective order of convergence, the original Milstein
scheme \eqref{Milstein}
is outperformed by the proposed derivative-free Milstein type scheme.
Compared to the exponential Euler scheme, the
proposed scheme obtains a higher effective 
order of convergence for a large set of parameter values
when combined with  
Algorithm 1 from~\cite{MR3949104}.
In Section~\ref{NumExampleNonComm}, we analyze the mean-square
error and the computational cost for the derivative-free 
Milstein type scheme numerically. 
The presented simulations confirm a higher effective order
of convergence in contrast to the original Milstein scheme 
and at least the same or even higher effective order of convergence 
in contrast to the Euler scheme for the examples
considered in Section~\ref{NumExampleNonComm}.
Finally, in Sections~\ref{Sec:Conclusion} and~\ref{Sec:Proofs}, we give some 
concluding remarks and the proofs for the convergence results.
\section{Approximation of Solutions for SPDEs}
In this section, we present a derivative-free Milstein type scheme 
for SPDE~\eqref{SPDE}
which does not need to have commutative noise. 
Precisely, we introduce a scheme which can be coupled
with an arbitrary method for the approximation 
of the involved iterated stochastic integrals. For example, 
when combined
with the algorithms introduced in \cite{MR3949104}
for the simulation of twice-iterated integrals, the
theoretical order of 
convergence of the original Milstein scheme can be maintained.
\subsection{Framework}\label{Sec:Setting}
Throughout this work, we assume the framework presented in the following.
Let $(H,\langle \cdot,\cdot \rangle_H)$ and $(U,\langle \cdot,\cdot\rangle_U)$ denote 
some separable real-valued Hilbert spaces  and let $T\in(0,\infty)$ be some fixed 
time point. Further, let the operator $Q \in L(U)$ be non-negative, symmetric and 
have finite trace. Then, the subspace
$U_0 \subset U$ is defined as $U_0 =  Q^{\frac{1}{2}}U$.
Moreover, we consider some complete probability space $(\Omega,\mathcal{F},\Prob)$ 
and a $U$-valued $Q$-Wiener process $(W_t)_{t\in[0,T]}$ with respect to the
filtration $(\mathcal{F}_t)_{t\in [0,T]}$ which fulfills the usual conditions. 
In terms of 
the eigenvalues of $Q$, denoted as $\eta_j$, with corresponding eigenvectors 
$\tilde{e}_j$ for $j\in\mathcal{J}$
with some countable index set $\mathcal{J}$ forming an orthonormal basis
$\{ \tilde{e}_j : j \in J\}$ of $U$ (see~\cite{MR2329435}),
we obtain the following series representation 
of the $Q$-Wiener process, see~\cite{MR2329435},
\begin{align}
   W_t = \sum_{\substack{j\in\mathcal{J} \\ \eta_j \neq 0}} 
   \sqrt{\eta_j} \, \tilde{e}_j \, \beta_t^j, \quad t\in[0,T].
\end{align}
In this representation, the stochastic processes $(\beta_t^j)_{t\in[0,T]}$ denote independent 
real-valued Brownian motions for all $j\in\mathcal{J}$
with $\eta_j\neq 0$.
Below, the following notation is used for different sets 
of linear operators. The space of linear and bounded operators 
mapping from $U$ to $H$ that are restricted to the subspace $U_0$ is called
$(L(U,H)_{U_0},\|\cdot\|_{L(U,H)})$  with
$L(U,H)_{U_0} := \{T \colon U_0 \to H \, | \, T\in L(U,H) \}$, by $L_{HS}(U,H)$, we 
denote the set
of Hilbert-Schmidt operators mapping from $U$ to $H$ 
and, finally, we denominate $L^{(2)}(U,H) = L(U,L(U,H))$ 
and $L_{HS}^{(2)}(U,H)= L_{HS}(U,L_{HS}(U,H))$. 
\\ \\
For the existence and uniqueness of a mild solution of SPDE~\eqref{SPDE} 
and the validity of the proofs of convergence in Section~\ref{Proof}, we 
assume the following conditions.
\begin{description}
  \item[(A1)]  The linear operator 
  $A \colon \mathcal{D}(A)\subset H \to H$ is the generator of an analytic $C_0$-semigroup
  $S(t) = e^{At}$ for all $t\geq 0$. 
  We denote the eigenvalues of $-A$ by $\lambda_i \in (0,\infty)$ and the corresponding 
  eigenvectors by $e_i$ for $i \in \mathcal{I}$ and some countable index set $\mathcal{I}$,
  that is, $-Ae_i =\lambda_i e_i$ for all $i\in\mathcal{I}$. Furthermore, let
  $\inf_{i\in\mathcal{I}}\lambda_i >0$ and let
  the eigenfunctions $\{ e_i : i\in\mathcal{I}\}$ of $-A$
  form an orthonormal basis of $H$, see~\cite{MR1873467}, and 
  \begin{equation*}
    Av = \sum_{i\in\mathcal{I}} -\lambda_i\langle v,e_i\rangle_H e_i
  \end{equation*}
  for all $v\in \mathcal{D}(A)$. We introduce the real Hilbert
  spaces
  $H_r := \mathcal{D}((-A)^r)$ for $r\in[0,\infty)$
  with norm $\|x\|_{H_r} =\|(-A)^rx\|_H$ for
  $x\in H_r$.
  \item[(A2)] Let $\beta\in[0,1)$ and assume that 
  $F \colon H_{\beta} \to H$ is twice continuously Fr\'{e}chet differentiable
  with $\sup_{v\in H_{\beta}} \|F'(v)\|_{L(H)} < \infty$ and
  $\sup_{v\in H_{\beta}} \|F''(v)\|_{L^{(2)}(H_{\beta},H)}<\infty$.
  \item[(A3)] The operator $B \colon H_{\beta} \to L(U,H)_{U_0}$ 
  is assumed to be twice continuously Fr\'{e}chet 
  differentiable such that 
  $\sup_{v\in H_{\beta}} \|B'(v)\|_{L(H,L(U,H))} < \infty$
  and $\sup_{v\in H_{\beta}} \|B''(v)\|_{L^{(2)}(H,L(U,H))}<\infty$. 
  Further, let $B(H_{\delta}) \subset L(U,H_{\delta})$ for some $\delta
  \in (0,\tfrac{1}{2})$ and assume that 
  \begin{align*}
      \| B(u) \|_{L(U,H_{\delta})} &\leq C ( 1 +
      \| u \|_{H_{\delta}} ) , \\
      \| B'(v) P B(v) - B'(w) P B(w) \|_{L_{HS}^{(2)}(U_0,H)} &\leq C
      \| v - w \|_{H} , \\
      \| (-A)^{-\vartheta} B(v) Q^{-\alpha} \|_{L_{HS}(U_0,H)} &\leq C
      (1 + \| v \|_{H_{\gamma}})
  \end{align*}
  for some constant $C>0$, all $u \in H_{\delta}$, $v, w \in H_{\gamma}$, where
  $\gamma \in \left[ \max( \beta, \delta),
  \delta + \frac{1}{2} \right)$, $\alpha \in (0,\infty)$, $\vartheta \in \left( 0,
  \frac{1}{2} \right)$, $\beta \in [0,\delta+\tfrac{1}{2})$,
  any projection operator $P \colon H \to \text{span}\{e_i: i\in\tilde{\mathcal{I}}\}\subset H$
  with finite index set $\tilde{\mathcal{I}}\subset \mathcal{I}$ and the 
  case that $P$ is the identity.
  \item[(A4)] The initial value $\xi \colon \Omega \to H_{\gamma}$ is
  $\mathcal{F}_0$-$\mathcal{B}(H_{\gamma})$-measurable and
  it holds $\mathrm{E}\big[\|\xi\|^4_{H_{\gamma}}\big] <\infty$.
  \item[(A5)]
  Assume that at least one of the following conditions is fulfilled:
  \begin{enumerate}[a)]
    \item \label{A5a} $Q^{\frac{1}{2}}$ is a trace class operator,
    \item \label{A5c}
    $\|B''(v)(P B(u), P B(u))\|_{L^{(2)}(U,L(U,H))}
    \leq C(1+\|v\|_{H} + \|u\|_{H})$ for all $u,v \in H$, some $C>0$ and
    any projection operator $P \colon H \to \text{span}\{e_i: i\in\tilde{\mathcal{I}}\}\subset H$
    with finite index set $\tilde{\mathcal{I}}\subset \mathcal{I}$.
  \end{enumerate}
\end{description}
In this work, we do not make a difference between the operator $B$ and its extension 
$\tilde{B} \colon H \to L(U,H)_{U_0}$. The operator $\tilde{B}$ is globally Lipschitz continuous as 
$H_{\beta}\subset H$ is dense. With $F$, we deal analogously. 
%
%
Note that assumptions (A1)--(A4) are the same 
as for the scheme for SPDEs with commutative noise introduced 
in~\cite{MR3842926} and similar to 
the conditions imposed in~\cite{MR3320928} and \cite{MCQMC2018}
for the original Milstein scheme. However, 
the commutativity condition~\eqref{Comm}, which is essential
in~\cite{MR3320928} and \cite{MR3842926}, needs not to
be fulfilled in our setting. On the other hand, assumption (A5) 
is required. 
%
%
%
Assumptions (A1)--(A4) assure the existence of a unique mild solution 
$X \colon [0,T] \times \Omega \to H_{\gamma}$ for SPDE
\eqref{SPDE}, see
~\cite{MR2852200,MR3320928}. 
Moreover, it holds   
$\big(\sup_{t \in [0,T]} \mathrm{E}\big[ \|X_t \|_{H_{\gamma}}^4 + \| B(X_t)
\|^4 _{L_{HS}(U_0,H_{\delta})} \big]\big) < \infty$ and 
\begin{equation*}
        \sup_{\substack{s,t \in [0,T] \\ s \neq t}}
        \frac{\left( \mathrm{E}\big[ \|X_t-X_s\|_{H_r}^p \big] \right)^{\frac{1}{p}}}
        {|t-s|^{\min(\gamma-r, \frac{1}{2})}}
        <\infty
\end{equation*}
for every $r \in [0,\gamma]$ and $p\in[2,4]$, see \cite{MR2852200}.
\subsection{The Derivative-Free Milstein Type Scheme}\label{Sec:DFM}
We derive a numerical scheme to approximate
the mild solution of SPDE~\eqref{SPDE} in this section.
At first, the infinite dimensional spaces have to be discretized. 
For the solution
space $H$, we introduce the projection operator
$P_N \colon H \to H_N$ that maps $H$ to the finite dimensional subspace 
$H_N := \text{span}\{e_i : i\in\mathcal{I}_N\}$
for some fixed $N\in\mathbb{N}$ with some index
set $\mathcal{I}_N \subset \mathcal{I}$ 
and $|\mathcal{I}_N|=N$. We define this operator as
\begin{equation*}
  P_N x = \sum_{i\in\mathcal{I}_N} \langle x,e_i\rangle_H e_i, \quad x\in H.
\end{equation*}
Analogously, we define the projection
operator $P_K \colon U \to U_K$
to approximate the $Q$-Wiener process for some fixed $K\in\mathbb{N}$ by
\begin{equation*}
  W_t^K := P_K W_t = \sum_{j\in\mathcal{J}_K} \sqrt{\eta_j} \tilde{e}_j \beta_t^j,
  \quad t\in[0,T],
\end{equation*}
with $U_K := \text{span}\{\tilde{e}_j : j \in \mathcal{J}_K\}$ for
some index set $\mathcal{J}_K \subset \mathcal{J}$, $|\mathcal{J}_K| = K$, and 
$\eta_j\neq 0$ for $j\in\mathcal{J}_K$.
In order to discretize the time interval, we work with 
an equidistant time step for legibility of the representation.
Let $h = \frac{T}{M}$ for some $M\in\mathbb{N}$ and define $t_m = m\cdot h$ 
for $m\in\{0,\ldots,M\}$.
The increments of the approximated $Q$-Wiener process are then denoted by
\begin{equation*}
    \Delta \WW_m := W_{t_{m+1}}^K - W_{t_m}^K
    = \sum_{j \in \mathcal{J}_K}
    \sqrt{\eta_j} \, \Delta \beta_m^j \, \tilde{e}_j \quad
\end{equation*}
where the increments of the real-valued Brownian
motions are given by  $\Delta \beta_m^j = 
\beta_{t_{m+1}}^j-\beta_{t_m}^j$
for $m \in \{0,\ldots, M-1\}$, $j\in\mathcal{J}_K$. 
The Milstein scheme \eqref{Milstein} is computationally expensive 
due to the derivative that has to be evaluated in
each step, see also~\cite{MR3842926} and \cite{MCQMC2018}. 
In order to compare numerical methods in this work, we
consider the so-called effective order of convergence
first introduced in \cite{MR2669396}.
This number combines the theoretical order of convergence 
with the computational cost involved in the calculation of an approximation 
by a particular scheme.
As in~\cite{MR3842926}, the goal is to raise the effective order of convergence 
by means of a customized approximation that is free of derivatives and, in addition, 
computationally less expensive. Here, however, we do not need to assume that the operator $B'B$
fulfills a commutativity condition, that is, condition~\eqref{Comm} is not required. 
In Theorem~\ref{Error:DFM}, we state that for the proposed 
derivative-free Milstein type scheme, the theoretical 
order of convergence 
is the same as for the Milstein scheme given in~\cite{MR3320928}. 
At the same time, compared to the Milstein scheme in~\cite{MR3320928},
the computational effort is significantly reduced for the derivative-free Milstein type scheme. 
That means that the effective order of convergence is a priori
larger for the proposed derivative-free method.
Moreover, compared to the Euler schemes like in~\cite{MR2471778,MR1825100,MR3047942}, 
the computational cost is of the same magnitude while the order of convergence 
w.r.t.\ step size $h$ is at least the same or even significantly higher.
Thus, the scheme that we derive in the
following is more efficient in terms of the effective order of convergence 
than the Euler type schemes
for most parameter sets determined by the SPDE under consideration
if we combine it, for example,
with the algorithms for the simulation of 
the iterated stochastic integrals introduced 
in~\cite{MR3949104},
see Table~\ref{Tab:CompareEffOrder}. 
Precisely, compared to the Euler schemes,
the increase in the computational cost that results from the approximation
of the iterated stochastic integrals can be neglected and we get, 
in many cases, a significantly higher effective order of convergence due to the
higher theoretical order of convergence in the time
step that the derivative-free Milstein type scheme features.
\\ \\
The main idea for the derivative-free Milstein type scheme is
alike to that in the commutative case, see~\cite{MR3842926}, 
which in turn is based on the work for the finite dimensional
setting in~\cite{MR2338537,MR2505871,MR2669396}. The operator $B'B$ 
is approximated by a customized difference operator in such a way that the overall computational cost is decreased
by one order of magnitude.  However, the stage values have to be chosen differently compared to 
the commutative case. Compared to \cite{MR3842926}, the main 
distinction is that we employ one difference term only.
The derivative-free Milstein type scheme yields a discrete process which we denote by 
$(\Y_m^{N,K,M})_{m \in \{0,\ldots, M\}}$ such that $Y_m^{N,K,M}$ is $\mathcal{F}_{t_m}$-$\mathcal{B}(H)$-measurable
for all $m \in \{0,\ldots,M\}$, $M\in\mathbb{N}$. 
We define the derivative-free Milstein type ($\DFM$) scheme as
$Y_{0}^{N,K,M} = P_N \xi$ and
\begin{equation}
  \begin{split} \label{DFM}
  Y_{m+1}^{N,K,M} &= P_N e^{Ah} \bigg( Y_{m}^{N,K,M} + h F(Y_{m}^{N,K,M}) + B(Y_m^{N,K,M}) \Delta W_m^{K,M} \\
  &\quad + \sum_{j\in\mathcal{J}_K} \Big( B \Big( Y_m^{N,K,M} + \sum_{i\in\mathcal{J}_K} P_N B(Y_m^{N,K,M}) 
  \tilde{e}_i \, I^Q_{(i,j),m} \Big) - B\big(Y_m^{N,K,M}) \Big) \tilde{e}_j \bigg)
  \end{split}
\end{equation}
for $m\in \{0,\ldots,M-1\}$, $N,M,K\in\mathbb{N}$. 
For $i,j\in\mathcal{J}_K$ and  $m\in\{0,\ldots,M-1\}$, the term $I^Q_{(i,j),m} = I^Q_{(i,j),t_m,t_{m+1}}$
denotes the iterated stochastic It{\^o} integral 
\begin{equation}\label{DoubleInt}
  I^Q_{(i,j),t_m,t_{m+1}} 
    = \int_{t_m}^{t_{m+1}}\int_{t_m}^s \, \langle \mathrm{d}W_r, \tilde{e}_i \rangle_U \, 
    \langle \mathrm{d}W_s, \tilde{e}_j \rangle_{U} .
\end{equation}
At this point, we assume that the iterated stochastic integrals 
are given exactly in order to consider the error estimate independent
of the approximation error for the iterated integrals. 
We consider the error resulting from the approximation of the mild solution of~\eqref{SPDE} 
without an approximation of the iterated stochastic integral
since this is interchangeable.
Then, in a second step, we conclude from Theorem~\ref{Thm:ErrorTotal} below
that if an approximation of the iterated stochastic integral 
fulfills some specified conditions, 
this estimate remains valid.
\begin{thm}[Convergence of $\DFM$ scheme] \label{Error:DFM}
   Assume that (A1)--(A4) and (A5) hold. Then, there exists a
    constant $C_{Q,T} \in (0,\infty)$, independent of $N$, $K$ and $M$, such
    that for $(Y_m^{N,K,M})_{0 \leq m \leq M}$, defined by the $\DFM$ scheme in
    \eqref{DFM}, it holds
    \begin{equation} \label{Thm:Error-DFM-Estimate:eqn}
      \max_{0 \leq m \leq M}
      \Big( \mathrm{E}\Big[ \big\| X_{t_m} - Y_m^{N,K,M} \big\|_H^2
      \Big] \Big)^{\frac{1}{2}}
      \leq C_{Q,T} \Big( \Big( \inf_{i \in \mathcal{I} \setminus
      \mathcal{I}_N} \lambda_i \Big)^{-\gamma}
      + \Big( \sup_{j \in \mathcal{J} \setminus \mathcal{J}_K}
      \eta_j \Big)^{\alpha} + M^{-\qd}\Big)
    \end{equation}
    for 
    all $N,K,M \in \mathbb{N}$
    and with $\qd = \min(2(\gamma-\beta),\gamma)$. The parameters are determined by 
    assumptions (A1)--(A4).
\end{thm}
\begin{proof}
  The proof of Theorem~\ref{Error:DFM} is stated in Section~\ref{Proof}.
\end{proof}
This is the same estimate (apart from the constant) as for the Milstein scheme 
\eqref{Milstein} proposed in~\cite{MR3320928} or the derivative-free Milstein 
type scheme for SPDEs with commutative noise in~\cite{MR3842926}. 
The computational effort, however, increases compared 
to the schemes for SPDEs with commutative noise
as the iterated stochastic integrals have to be simulated. We 
discuss this issue below.
\subsection{Approximation of Iterated Integrals}
\label{SubSec:Approx-Iterated-Integrals}
In Section~\ref{Sec:DFM}, we implicitly assumed that the iterated stochastic integrals
can be computed exactly. However, up to now there exists no algorithm for the 
exact simulation of the iterated stochastic integrals in a setting with
non-commutative noise. Therefore, the iterated integrals have
to be approximated appropriately.
We prove the following general result.
\begin{thm}\label{Thm:ErrorTotal}
  Let $\bar{I}^Q_{(i,j),m}$, $i,j\in\mathcal{J}_K$, $m \in \{0, \ldots, M-1\}$,
  denote some approximations of the iterated stochastic integrals
  in~\eqref{DoubleInt} and let $(\bar{\Y}_m)_{m \in \{0,\ldots,M\}}$ 
  with $\bar{\Y}_m = \bar{\Y}_m^{N,K,M}$ denote
  the discrete time process obtained by the
  $\DFM$ scheme~\eqref{DFM} 
  if the integrals
  $I^Q_{(i,j),m}$ are replaced by the approximations 
  $\bar{I}^Q_{(i,j),m}$, $i,j\in\mathcal{J}_K$, $m \in \{0,\ldots,M-1\}$. 
  Assume that conditions (A1)--(A5) are fulfilled and that
  \begin{align}\label{Cond1:DI}
    &\bigg( \mathrm{E}\bigg[
    \Big\|\int_{t_l}^{t_{l+1}}
    B'(\bar{\Y}_l)\Big(\int_{t_l}^s P_N B(\bar{\Y}_l) \, \mathrm{d}W_r^K\Big) \, \mathrm{d}W^K_s
    - \sum_{i,j\in\mathcal{J}_K} \bar{I}_{(i,j),l}^Q B'(\bar{\Y}_l) (P_N B(\bar{\Y}_l) \tilde{e}_i,
    \tilde{e}_j ) \Big\|_H^2 \bigg] \bigg)^{\frac{1}{2}}
    \leq \mathcal{E}(M,K)
  \end{align}
  for all $l \in \{0,\ldots,M-1\}$, $K,M \in \mathbb{N}$ and some function 
  $\mathcal{E} \colon \mathbb{N} \times \mathbb{N} \to \mathbb{R}_+$.
  Further, in case of assumption (A5\ref{A5a}) assume that 
  \begin{equation} \label{Cond2:DI-A5a}
    \sum_{j\in\mathcal{J}}
    \bigg( \mathrm{E} \bigg[ \Big( \sum_{i\in\mathcal{J}}  
    \big(\bar{I}_{(i,j),t,t+h}^Q \big)^2 \Big)^{2} \bigg] \bigg)^{\frac{1}{4}}
    \leq C_Q h
  \end{equation}
  and in case of assumption (A5\ref{A5c}) assume that
  \begin{equation} \label{Cond2:DI-A5c}
    \sum_{j \in \mathcal{J}} \bigg( 
    \mathrm{E} \bigg[ \Big( \sum_{i \in \mathcal{J}} \big( \bar{I}_{(i,j),t,t+h}^Q
    \big)^2 \Big)^q \bigg] \bigg)^{\frac{1}{2}}
    \leq C_Q h^q
  \end{equation}
  for $q\in\{2,3\}$, some $C_Q>0$, all $h >0$ and $t \in [0,T-h]$. 
  Then, there exists a constant $C_{Q,T} \in (0,\infty)$, independent of $N$, $K$ and $M$, 
  such that it holds
  \begin{equation*}
    \max_{0 \leq m \leq M}
    \Big( \mathrm{E}\Big[ \big\| X_{t_m} - \bar{\Y}_m \big\|_H^2
    \Big] \Big)^{\frac{1}{2}}
    \leq C_{Q,T} \Big( \Big( \inf_{i \in \mathcal{I} \setminus
    \mathcal{I}_N} \lambda_i \Big)^{-\gamma}
    + \Big( \sup_{j \in \mathcal{J} \setminus \mathcal{J}_K}
    \eta_j \Big)^{\alpha} + M^{-\qd}
    + M^{\frac{1}{2}} 
    \, \mathcal{E}(M,K) 
    \Big)
  \end{equation*}
  for 
  all $N,K,M \in \mathbb{N}$
  and with $\qd=\min(2(\gamma-\beta),\gamma)$.
\end{thm}
\begin{proof}
 The proof of this theorem is stated in Section~\ref{Proof}.
\end{proof}
Note that Theorem~\ref{Thm:ErrorTotal} applies to the Milstein
scheme \eqref{Milstein} as well, see also \cite{MCQMC2018}.
Now, we want to illustrate this statement with two exemplary choices -- 
Algorithm~1 and Algorithm~2 as introduced in~\cite{MR3949104}. 
First, we consider Algorithm~1 which is based on a series representation 
of the iterated stochastic integral. This representation is 
truncated after $D$ summands for some 
$D\in\mathbb{N}$, see~\cite{MR1178485,MR3949104}, which yields the approximation.
The numerical scheme~\eqref{DFM} is called $\DFMA$ if the iterated 
integrals are approximated by Algorithm~1 -- denoted 
as $\bar{I}^{Q}_{(i,j),m}=\bar{I}^{Q,(D),(1)}_{(i,j),m}$. 
For this method, there exists some constant $C_{Q,T}>0$ such that 
\eqref{Cond1:DI} is fulfilled with
\begin{align} \label{Iter-Int-error-Alg1}
  \mathcal{E}(M,K) = \mathcal{E}^{(D),(1)}(M,K) = C_{Q,T} \frac{1}{M \, \sqrt{D}}
\end{align}
for all 
$D,K,M\in\mathbb{N}$, see \cite[Corollary~1]{MR3949104}.
If we approximate the integrals with Algorithm~2 instead, 
we denote the scheme~\eqref{DFM} by $\DFMB$ and 
the approximation $\bar{I}^{Q}_{(i,j),m}$ of $I^Q_{(i,j),m}$ by
$\bar{I}^{Q,(D),(2)}_{(i,j),m}$. The series representation is not 
only truncated after $D$ summands, but the remainder is
approximated by a multivariate normally distributed random
vector, for details, we refer to~\cite{MR3949104,MR1843055}.
For this algorithm, \eqref{Cond1:DI} holds with
\begin{align}  \label{Iter-Int-error-Alg2}
  \mathcal{E}(M,K) = \mathcal{E}^{(D),(2)}(M,K) = C_{Q,T} 
  \frac{\min \Big(  K \sqrt{K-1}, (\min_{j\in\mathcal{J}_K} \eta_j)^{-1} \Big)}{M \, D}
\end{align}
for all 
$D, K, M \in\mathbb{N}$ and some constant $C_{Q,T}>0$, see \cite[Corollary~2, Theorem~4]{MR3949104}.
This estimate shows that the error converges in $D$ with a higher order, compared to the 
estimate for Algorithm~1. Note that the error estimate also depends on the number $K$, which
controls the accuracy of the approximation of the $Q$-Wiener process, and on the eigenvalues of the operator $Q$.
For a proof of the error estimates \eqref{Iter-Int-error-Alg1} and \eqref{Iter-Int-error-Alg2}, 
we refer to~\cite{MR3949104}.
Moreover, conditions~\eqref{Cond2:DI-A5a} and \eqref{Cond2:DI-A5c} are 
fulfilled for Algorithm~1 and~2,
which can be easily seen from the definition of the algorithms in \cite{MR3949104}. \\ \\
In order to determine which of the two algorithms obtains a higher
order of convergence, one has to analyze the computational costs that 
are involved, see also~\cite{MR3949104,MCQMC2018} for a comparison. The goal 
is that the $\DFM$ scheme combined with Algorithm~1 or Algorithm~2 preserves the 
error estimate stated in Theorem~\ref{Error:DFM}. 
This requires a choice of $D \geq D_{1} = \lceil M^{2\min(2(\gamma-\beta),\gamma)-1} \rceil$
for Algorithm~1, whereas for Algorithm~2 we need  
$D \geq D_2 = \lceil \min \big(  K \sqrt{K-1}, (\min_{j\in\mathcal{J}_K} \eta_j)^{-1} \big)
M^{\min(2(\gamma-\beta),\gamma)-\frac{1}{2}} \rceil$.
Alternatively, one can choose
$D \geq D_{1} = \lceil M^{-1} (\sup_{j \in \mathcal{J} \setminus \mathcal{J}_K}\eta_j )^{-2\alpha} \rceil$
for the first algorithm and 
$D \geq D_{2} = \lceil M^{-\frac{1}{2}}
\min \big(  K \sqrt{K-1}, (\min_{j\in\mathcal{J}_K} \eta_j)^{-1} \big)
( \sup_{j \in \mathcal{J} \setminus \mathcal{J}_K}\eta_j )^{-\alpha} \rceil$
for the second algorithm. 
However, if all summands of the error estimate
in Theorem~\ref{Error:DFM} are optimally balanced, then 
$( \sup_{j \in \mathcal{J} \setminus \mathcal{J}_K} \eta_j )^{\alpha} 
= \mathcal{O}( M^{-\min(2(\gamma-\beta),\gamma)} )$
which results in the same orders of magnitude for the choice of $D_1$ and $D_2$, respectively.
These considerations show that the computational effort for the two schemes $\DFMA$ and $\DFMB$ is 
determined by the parameters which in turn are specified by the equation. 
Therefore, the choice of the optimal scheme depends 
on the SPDE that has to be solved. 
%
%
From now on, we assume that $D\in\mathbb{N}$ is chosen such that 
the temporal order of convergence is not decreased, i.e., such that
$D=D_1$ for Algorithm~1 or $D=D_2$ for Algorithm~2, respectively.
\begin{remark}
  Note that Algorithm~1 and Algorithm~2 proposed in \cite{MR3949104}
  merely represent examples and that Theorem~\ref{Thm:ErrorTotal}
  is valid if the
  derivative-free Milstein type scheme $\DFM$ is combined 
  with any approximation for the iterated
  stochastic integrals such that conditions~\eqref{Cond1:DI} 
  together with \eqref{Cond2:DI-A5a} or \eqref{Cond2:DI-A5c} are fulfilled.
\end{remark}
\section{The Effective Order of Convergence -- A Comparison}\label{Sec:EffOrder}
In the following, we compare the performance of the derivative-free Milstein type ($\DFM$)
scheme to the performance of the original Milstein ($\MIL$) scheme \eqref{Milstein}, 
the exponential Euler ($\EES$) scheme and the linear implicit Euler ($\LIE$) scheme. 
For example, one can combine the $\DFM$ scheme and the $\MIL$ scheme 
with Algorithm~1 or Algorithm~2 in order to approximate the solution of SPDEs that need not fulfill the
commutativity condition \eqref{Comm}. However, the analysis can be done similarly
for any other approximation method for the iterated stochastic integrals
as specified in Theorem~\ref{Thm:ErrorTotal}. 
In the following, we restrict our analysis to Algorithm~1 as an example.
The $\LIE$ scheme is considered in \cite{MR1825100,MR2136207} and the
$\EES$ scheme is introduced in~\cite{MR3047942} which are combined with a 
Galerkin approximation. The proof of the following theorem is detailed in~\cite{Diss} 
and the main idea can be found in~\cite{MR3320928}.
\begin{prop}[Convergence of $\EES$ scheme] \label{Error:EES}
    Assume that (A1)--(A4) hold. Then, there exists a
    constant $C_{Q,T} \in (0,\infty)$, independent of $N$, $K$ and $M$, such
    that for the approximation process $(Y_m^{\EES})_{0 \leq m \leq M}$
    with $Y_m^{\EES}=Y_m^{\EES; N,K,M}$, 
    defined by the $\EES$ scheme, it holds
    \begin{equation} \label{Thm:Error-EES-Estimate:eqn}
      \max_{0 \leq m \leq M}
      \Big( \mathrm{E}\Big[ \big\| X_{t_m} - Y_m^{\EES} \big\|_H^2
      \Big] \Big)^{\frac{1}{2}} 
      \leq C_{Q,T} \Big( \Big( \inf_{i \in \mathcal{I} \setminus
      \mathcal{I}_N} \lambda_i \Big)^{-\gamma}
      + \Big( \sup_{j \in \mathcal{J} \setminus \mathcal{J}_K}
      \eta_j \Big)^{\alpha} + 
      M^{-q_{\EES}}
      \Big)
    \end{equation}
    for 
    all $N,K,M \in \mathbb{N}$ 
    and with $q_{\EES} = \min(\frac{1}{2},2(\gamma-\beta),\gamma)$. 
    The parameters are determined by assumptions (A1)--(A4).
\end{prop}
Compared to the $\DFM$ scheme,
the $\EES$ scheme requires less restrictive assumptions as we do not need 
(A5) or conditions on the second derivative of $B$ and the estimate for
$B'(v) P B(v)-B'(w) P B(w)$ can be omitted.
For the LIE scheme, similar results as in Proposition~\ref{Error:EES} 
can be obtained analogously.
Below, $q$ stands for the order of convergence w.r.t.\ 
the step size $h=\frac{T}{M}$ with $\qd = \qm
= \min(2(\gamma-\beta),\gamma) \geq \min(\frac{1}{2},2(\gamma-\beta),\gamma)
= \qee$.
However, to compare the performance of the
schemes we have to take into account
their computational cost in combination with their error estimates as, e.g., 
iterated stochastic integrals have to be simulated 
for the higher order schemes $\DFM$ and $\MIL$ only.
\subsection{The Cost Model}
In order to compare the efficiency of different approximation algorithms, one 
is usually interested in the dependency of the errors on their computational 
cost. Therefore, we consider a theoretical cost model proposed 
in \cite{MR3842926}. It is assumed that any standard arithmetic operation 
or evaluation of sine, cosine or exponential function etc.\ produces unit cost $1$.
Further, the simulation of any realization of an $N(0,1)$-distributed
real valued random variable is assumed to produce cost one as well.
However, the evaluation $\phi(v)$ of a functional $\phi \in V^*$
with $V=H$ or $V=U$ is assumed to be usually more costly with $\costs(\phi,v) 
= \costs(\phi) \equiv c$
for all $v \in V$ and for some $c \geq 1$ where typically $c \gg 1$. Such
functionals are needed for, e.g.,  the calculation of Fourier coefficients 
$\phi_i(v) = \langle v, \hat{e}_i \rangle_V$ of $v \in V$ for some ONB 
$\{\hat{e}_i\}_{i \in \mathbb{N}}$ of $V$.
Let $L(H,E)_{N} = \{ T\arrowvert_{H_N} : \ T \in L(H,E)\}$ for some
vector space $E$ and let
$L_{HS}(U,H)_{K,N} = \{ P_N T\arrowvert_{U_K} : \ T \in
L_{HS}(U,H)\}$.
As a result, we obtain for any $v,y \in H_N$ and $u \in U_K$ 
the following computational costs due to $|\mathcal{I}_N|=N$ and 
$|\mathcal{J}_K|=K$ \cite{MR3842926}:
\begin{enumerate}[i)]
\item
One evaluation of the mapping $P_N \circ F \colon H \to H_N$ with 
\[ 
    P_N F(y) = \sum_{i \in \mathcal{I}_N} \langle F(y), e_i \rangle_H \, e_i
\]
is determined by the functionals $\langle F(y), e_i \rangle_H$ for $i
\in \mathcal{I}_N$ which results in $\costs(P_N F) = \Oo( N )$.
\item
Evaluating $P_N \circ
B(\cdot)\arrowvert_{U_K} \colon H \to L_{HS}(U,H)_{K,N}$
with 
\[ 
    P_N B(y)u = \sum_{i \in \mathcal{I}_N} \sum_{j \in
    \mathcal{J}_K} \langle B(y) \tilde{e}_j, e_i \rangle_H \, \langle u,
    \tilde{e}_j \rangle_U \, e_i
\]
needs the
evaluation of the functionals $\langle B(y) \tilde{e}_j, e_i
\rangle_H$ for $i \in \mathcal{I}_N$ and $j \in \mathcal{J}_K$
which results in $\costs(P_N \circ B(y)\arrowvert_{U_K}) = \Oo( N K )$.
\item
For $P_N \circ
B'(\cdot)(\cdot,\cdot)\arrowvert_{H_N,U_K} \colon H \to
L(H,L_{HS}(U,H)_{K,N})_N$ with
\[
    P_N \big( (B'(y)v)u \big) = \sum_{k,l \in
    \mathcal{I}_N} \sum_{j \in \mathcal{J}_K} \langle (B'(y) e_k)
    \tilde{e}_j, e_l \rangle_H \, \langle v, e_k \rangle_H \, \langle u,
    \tilde{e}_j \rangle_U \, e_l
\]
the functionals $\langle (B'(y) e_k) \tilde{e}_j, e_l \rangle_H$ have to
be evaluated for all $k,l \in \mathcal{I}_N$ and $j \in \mathcal{J}_K$
and it follows that $\costs(P_N \circ
B'(y)(\cdot,\cdot)\arrowvert_{H_N,U_K}) = \Oo( N^2 K )$.
\end{enumerate}
Considering the computational cost for one time step 
of the Milstein scheme \eqref{Milstein}, one evaluation of 
$P_N \circ F(\cdot)$, one of $P_N \circ B(\cdot)\arrowvert_{U_K}$, 
and one evaluation of $P_N \circ B'(\cdot)\arrowvert_{H_N,U_K}$ are 
needed. Then, the evaluated 
operators $P_N \circ B(\YY_m) \arrowvert_{U_K} \in L(U,H)_{K,N}$,
$P_N \circ B'(\YY_m)\arrowvert_{H_N,U_K} \in L(H,L_{HS}(U,H)_{K,N})_N$,
$P_N \circ B'(\YY_m)(v)\arrowvert_{U_K} \in L_{HS}(U,H)_{K,N}$ and 
$P_N \circ e^{Ah} \in L(H,H)_{N,N}$ have to be applied to the 
corresponding elements of the Hilbert spaces. 
Here, it has to be pointed out that calculating the Fourier coefficients 
of $P_N B(\YY_m) \tilde{e}_j$ for some basis element $\tilde{e}_j \in U_K$ is
for free because they are in the $j$-th column of the matrix
representation $P_N B(\YY_m)\arrowvert_{U_K} = \big(b_{i,j}(\YY_m)
\big)_{i \in \mathcal{I}_N, j \in \mathcal{J}_K}$ with
$b_{i,j}(\YY_m)=\langle B(\YY_m) \tilde{e}_j, e_i \rangle_H$ which
are already determined. The same applies to the operator 
$P_N \circ B'(\YY_m)(v)\arrowvert_{U_K} \in L_{HS}(U,H)_{K,N}$ if 
it is applied to some basis element $\tilde{e}_j \in U_K$.
Thus, the computational cost for $M$ time steps of the Milstein scheme
is $\costs(\MIL(N,K,M)) = \Oo(N^2 K M)$ if the cost for the simulation of
iterated stochastic integrals is not taken into account.
\\ \\
In contrast to the Milstein scheme, in each time step the proposed 
derivative-free Milstein type scheme $\DFM$ needs one evaluation of $P_N \circ F(\cdot)$, 
one evaluation of $P_N \circ B(\cdot)\arrowvert_{U_K}$, and the calculation of
\begin{equation} \label{DFM-comput-costs-eqn01}
  \sum_{j \in \mathcal{J}_K} P_N B \Big( Y_m^{N,K,M} + \sum_{i\in\mathcal{J}_K} P_N B(Y_m^{N,K,M}) 
  \tilde{e}_i \, I^Q_{(i,j),m} \Big) \tilde{e}_j .
\end{equation}
Observe that the calculation of each summand 
requires the computation of the functionals
\begin{equation*}
  \phi_k^j = \big\langle B \big( Y_m^{N,K,M} + \sum_{i\in\mathcal{J}_K} P_N B(Y_m^{N,K,M}) 
  \tilde{e}_i \, I^Q_{(i,j),m} \big) \tilde{e}_j, e_k \big\rangle_H
\end{equation*}
for $k \in \mathcal{I}_N$ with $\costs(\{\phi_k^j : k \in \mathcal{I}_N \})=c N$ 
for each $j \in \mathcal{J}_K$ due to $|\mathcal{I}_N|=N$.
Therefore, the evaluation of \eqref{DFM-comput-costs-eqn01} can be
done with cost $\Oo( N K )$. Here, the crucial point is that although the argument
of $B$ depends on the index $j$, the resulting operator is then applied to
the basis function $\tilde{e}_j$ only, which is responsible for the fundamental reduction 
of the computational complexity.
In addition, the linear operators $P_N \circ e^{Ah} \arrowvert_{H_N}
\in L(H,H)_{N,N}$ and $P_N \circ B(\YY_m)\arrowvert_{U_K} \in
L(U,H)_{K,N}$ (note again that calculating, e.g.,
$P_N B(\YY_m) \tilde{e}_j$ for a
basis $\tilde{e}_j \in U_K$ is for free) have to be applied to the 
corresponding elements of the Hilbert spaces. 
Thus, the total computational cost for $M$ time
steps of the $\DFM$ scheme is $\costs(\DFM(N,K,M))=\Oo(N K M)$
if the cost for the simulation of
iterated stochastic integrals is not taken into account.
\\ \\
\begin{table}[tbp] 
\begin{small}
\begin{center}
  \renewcommand{\arraystretch}{1.4}
  \begin{tabular}{|l|c|c|c|c|}
    \hline
    & \multicolumn{3}{c|}{\# of evaluations of functionals} & 
    \\
    Scheme & $\ \ \ P_N F(\cdot)\arrowvert_{H_N} \ \ \ \ $ &
    $\ \ \ P_N B(\cdot)\arrowvert_{U_K} \ \ \ $ &
    $ \ P_N B'(\cdot)\arrowvert_{H_N,U_K} \ $ & \# of $N(0,1)$  r.\ v.\  
    \\
    \hline \hline
    $\LIE$ & $N$ & $KN$ & $-$ & $K$ 
    \\ \hline
    $\EES$ & $N$ & $KN$ & $-$ & $K$ 
    \\ \hline
    $\MILA$ & $N$ & $KN$ & $KN^2$ & $K(1+2 D_1)$ 
    \\ \hline
    $\MILB$ & $N$ & $KN$ & $KN^2$ & $K(1+2 D_2)+\frac{1}{2} K (K-1)$ 
    \\ \hline
    $\DFMA$ & $N$ & $2KN$ & $-$ & $K(1+2 D_1)$ 
    \\ \hline
    $\DFMB$ & $N$ & $2KN$ & $-$ & $K(1+2 D_2)+\frac{1}{2} K (K-1)$ 
    \\ \hline
  \end{tabular}
  \renewcommand{\arraystretch}{1.0}
\end{center}
\end{small}
\caption{Computational cost given by the number of evaluations of real-valued functionals and
independent $N(0,1)$-distributed random variables needed for each time step.}
\label{funcEval-1}
\end{table}%
Analogously to the derivative-free Milstein type schemes $\DFMA$ and $\DFMB$,
we denote the Milstein scheme by $\MILA$ and $\MILB$ if it is combined with
either Algorithm~1 or Algorithm~2 proposed in \cite{MR3949104} for 
the approximation of the iterated stochastic integrals, respectively.
Then, the dominating computational cost due to necessary evaluations 
of real-valued functionals and the simulation of random numbers 
for each time step can be found in Table~\ref{funcEval-1} for the 
linear implicit Euler scheme $\LIE$ as well as for the $\EES$, 
$\MILA$, $\MILB$, $\DFMA$ and $\DFMB$ schemes. 
It has to be pointed out that, in contrast to finite-dimensional 
stochastic differential equations (SDEs), the computational effort of each 
numerical scheme depends not only on the number of time steps $M$,
but also on the dimensions $N$ and $K$ of the subspaces $H_N$ and $U_K$
which have to increase in order to decrease the approximation error,
compare Theorem~\ref{Error:DFM} and Proposition~\ref{Error:EES}.
However, it turns out that different schemes can attain the same
error estimates like in \eqref{Thm:Error-DFM-Estimate:eqn} and 
\eqref{Thm:Error-EES-Estimate:eqn}, however with significantly different 
computational cost, see also the discussion in \cite{MR3842926}.
In order to compare the performance of different numerical 
schemes, one has to compare the accuracy of each scheme versus the 
needed computational cost instead of just comparing their error 
estimates w.r.t.\ $N$, $K$ and $M$ given in Theorem~\ref{Error:DFM} 
and Proposition~\ref{Error:EES}. For example, the Milstein scheme 
and the derivative-free Milstein type scheme both 
attain the same error estimate
\eqref{Thm:Error-DFM-Estimate:eqn}, however, the computational effort
for $\MIL$ is $\costs(\MIL(N,K,M)) = \Oo(N^2 K M)$ whereas for the $\DFM$ 
scheme it is only $\costs(\DFM(N,K,M))=\Oo(N K M)$ if the random numbers 
are assumed not to be the dominating cost. Therefore, 
in this case, the $\DFM$ scheme performs
a priori with a higher order of convergence compared to the $\MIL$ scheme
if errors versus costs are considered. 
Compared with the $\LIE$ scheme and the 
$\EES$ scheme, the $\DFM$ scheme belongs to the same class $\Oo(N K M)$ of 
computational complexity, which is in some sense optimal for 
one-step approximations for SPDEs of type \eqref{SPDE}. 
Although the $\LIE$ scheme as well as the $\EES$ scheme
have worse error bounds given in \eqref{Thm:Error-EES-Estimate:eqn}
compared to the one for the $\DFM$ and the $\MIL$ scheme
in \eqref{Thm:Error-DFM-Estimate:eqn},
it is not clear which scheme should be preferred because the computational 
cost for simulating the iterated stochastic integrals for 
the $\DFM$ and the $\MIL$ scheme have to be taken into account as well.
Therefore, we derive the effective order of convergence for 
each scheme under consideration.
This concept is also detailed in~\cite{MR3842926}. 
\subsection{Comparison of the Effective Orders of Convergence}
In order to compare the performance of different numerical schemes,
we consider the so-called effective order of convergence which was 
proposed in \cite{MR2669396} and also considered in \cite{MR3842926}.
In the following, we restrict our comparison to the schemes $\DFMA$ and $\MILA$,
both using Algorithm~1, as well as the $\EES$ scheme in order to keep the analysis concise. 
For a detailed analysis and comparison of the effective order of 
convergence for the schemes $\MILA$, $\MILB$ and $\EES$ we refer to \cite{MCQMC2018}.
Since the $\LIE$ scheme and the $\EES$ scheme have the same order of convergence and 
similar computational cost, we restrict our analysis 
to the $\EES$ scheme in the following because one can get exactly the 
same results for the $\LIE$ scheme.
We want to point out that the focus of this article
lies on the introduction and analysis of the derivative-free Milstein type scheme and
a complete comparison taking into account further algorithms next to Algorithm~1
for the simulation of the iterated stochastic integrals
would go beyond the scope of this article and may be object of future research.
\\ \\
For each scheme under consideration
and its approximation process $(Y_m^{N,K,M})_{m\in\{0,\ldots,M\}}$, 
we have to minimize the error term 
\begin{equation*}
  \sup_{m\in\{0,\ldots,M\}}\big( \mathrm{E}\big[ \| X_{t_m} 
  - Y_m^{N,K,M} \|_H^2  \big] \big)^{\frac{1}{2}} 
\end{equation*}
over all $N,K,M \in\mathbb{N}$ under the constraint  
that the computational cost does not exceed some specified value $\bar{c}>0$.
Note that if $D$ is chosen as described in 
Section~\ref{SubSec:Approx-Iterated-Integrals}, then the computational 
cost of each scheme given in Table~\ref{funcEval-1} depends on 
$N$, $K$ and $M$ only. 
In the following, we assume that 
$\sup_{j\in\mathcal{J} \setminus \mathcal{J}_K }\eta_j = \mathcal{O}( K^{-\rho_Q})$
and $( \inf_{i\in\mathcal{I}\setminus \mathcal{I}_N}\lambda_i  )^{-1} 
= \mathcal{O}( N^{-\rho_A})$
for some $\rho_A>0$ and $\rho_Q>1$. Then, we obtain the following expression
for all $N,K,M \in\mathbb{N}$ and some $C>0$, see also~\cite{MR3842926},
\begin{equation*}
  \text{err}(\text{SCHEME}) =\sup_{m\in\{0,\ldots,M\}}
  \Big( \mathrm{E}\Big[ \big\| X_{t_m} - Y_m^{N,K,M} \big\|_H^2  \Big] \Big)^{\frac{1}{2}}
  \leq C \big( N^{-\gamma\rho_A}+K^{-\alpha\rho_Q}+M^{-q} \big).
\end{equation*}
Note that the parameter $q>0$ is determined by the scheme that is considered.
Given some computational cost $\bar{c}>0$, the goal is to minimize 
the error under the constraint that the computational cost is bounded by $\bar{c}$.
Solving this optimization problem yields the effective
order of convergence, denoted by EOC(SCHEME), which is then given by an expression
of the form
\begin{equation*}
  \text{err}(\text{SCHEME}) =\mathcal{O}\big(\bar{c}^{\, \, -\text{EOC(SCHEME)}}\big).
\end{equation*}
Next, we analyze the effective order of convergence for the $\DFMA$ scheme and the $\MILA$ scheme,
which make use of Algorithm~1 for the approximation of the iterated stochastic integrals,
and the $\EES$ scheme. 
Therefore, let $q := \qd = \qm = \min(2(\gamma-\beta),\gamma)$ and let 
$D= \Oo( M^{2q-1} )$ for Algorithm~1 in the following.
\\ \\
First, we consider the scheme $\DFMA$. The computational 
cost for the calculation of one trajectory amounts to 
$\bar{c} = \mathcal{O}(MKN)+\mathcal{O}(KM^{2 q})$, 
see Table~\ref{funcEval-1} and the discussion in the last section. 
Now, two cases have to be distinguished: 
If $\gamma \rho_A (2q -1) \leq q$
is fulfilled, then $\bar{c} = \mathcal{O}(MKN)$.
We solve the optimization problem and obtain 
\begin{equation} \label{Choice-MNK-DFMC2}
  \begin{split}
    M &= \mathcal{O}\Big(\bar{c}^{\, \frac{\gamma\rho_A\alpha\rho_Q}{(\alpha\rho_Q+\gamma\rho_A) q + \alpha\rho_Q\gamma\rho_A}}\Big),
    \quad  N = \mathcal{O}\Big(\bar{c}^{\, \frac{\alpha\rho_Q q}{(\alpha\rho_Q+\gamma\rho_A) q +\alpha\rho_Q\gamma\rho_A}}\Big),\\
    K &= \mathcal{O}\Big(\bar{c}^{\, \frac{\gamma\rho_A q}{(\alpha\rho_Q+\gamma\rho_A) q +\alpha\rho_Q\gamma\rho_A}}\Big) .
  \end{split}
\end{equation}
Further, the effective order of convergence is given by
\begin{equation}\label{StandardEffOrd-DFMC2}
  \text{err}(\DFMA) 
  = \mathcal{O}\Big(\bar{c}^{\, \, -\frac{\gamma\rho_A\alpha\rho_Q q }{(\alpha\rho_Q+\gamma\rho_A) q +\alpha\rho_Q\gamma\rho_A}}\Big) ,
\end{equation}
which is the same result as for the derivative-free Milstein type scheme
in the case of SPDEs with commutative noise, see the computations in~\cite{MR3842926}.
On the other hand, if $\gamma \rho_A (2 q -1) \geq q$ holds, then $\bar{c} = \mathcal{O}(KM^{2 q})$ 
and optimization yields
\begin{align} \label{Choice-MNK-DFM1C1}
  M = \Oo \Big(\bar{c}^{\, \frac{\alpha\rho_Q}
  {(2\alpha\rho_Q+1) q}}\Big), \quad
  N = \Oo \Big(\bar{c}^{\, \frac{\alpha\rho_Q}
  {(2\alpha\rho_Q+1)\gamma\rho_A}}\Big), \quad
  K = \Oo \Big(\bar{c}^{\, \frac{1}
  {2\alpha\rho_Q+1}}\Big) .
\end{align}
In this case, we obtain the effective order of convergence from
\begin{equation}\label{DFM1EffOrd-DFM1C1}
  \text{err}(\DFMA) = \Oo \Big(\bar{c}^{\, \, -\frac{\alpha\rho_Q}
  {2\alpha\rho_Q+1}}\Big) .
\end{equation}
%
%
Next, we consider the Milstein scheme $\MILA$. Here, the computational effort 
for the computation of one trajectory is $\bar{c} = \mathcal{O}(MKN^2)+\mathcal{O}(KM^{2 q})$,
compare Table~\ref{funcEval-1}.
Again, two cases have to be considered:
If $\gamma \rho_A(2q -1) \leq 2 q$, then $\bar{c} = \mathcal{O}(MKN^2)$
and solving the optimization problem yields 
\begin{equation} \label{Choice-MNK-MIL1-C2}
  \begin{split}
    M &= \mathcal{O}\Big(\bar{c}^{\, \frac{\gamma\rho_A\alpha\rho_Q}{(2\alpha\rho_Q+\gamma\rho_A) q + \alpha\rho_Q\gamma\rho_A}}\Big),
    \quad  N = \mathcal{O}\Big(\bar{c}^{\, \frac{\alpha\rho_Q q}{(2\alpha\rho_Q+\gamma\rho_A) q +\alpha\rho_Q\gamma\rho_A}}\Big),\\
    K &= \mathcal{O}\Big(\bar{c}^{\, \frac{\gamma\rho_A q}{(2\alpha\rho_Q+\gamma\rho_A) q +\alpha\rho_Q\gamma\rho_A}}\Big) .
   \end{split}
\end{equation}
As a result of this, we obtain the effective order of convergence from
\begin{equation}\label{EffOrd-MIL1-C2}
  \text{err}(\MILA) 
  = \mathcal{O}\Big(\bar{c}^{\, \, -\frac{\gamma\rho_A\alpha\rho_Q q }{(2\alpha\rho_Q+\gamma\rho_A) q +\alpha\rho_Q\gamma\rho_A}}\Big),
\end{equation}
which is also the same effective order of convergence as for the Milstein scheme
if it is applied to some SPDE with commutative noise, see also~\cite{MR3842926}.
However, in the case of $\gamma \rho_A(2q -1) \geq 2 q$ the computational 
effort for the $\MILA$ scheme is $\bar{c} = \mathcal{O}(KM^{2 q})$ and 
we obtain the same choice for 
$M$, $N$ and $K$ as given in \eqref{Choice-MNK-DFM1C1} and also the same 
effective order of convergence as given by \eqref{DFM1EffOrd-DFM1C1}, 
see also~\cite{MCQMC2018}.
\\ \\
Finally, we consider the $\EES$ scheme where the optimal 
choice for $M$, $N$ and $K$ is given by \eqref{Choice-MNK-DFMC2},
however, with $q=\qee = \min(\frac{1}{2},2(\gamma-\beta),\gamma)$
and the effective order of convergence for the $\EES$  scheme
was computed in~\cite{MR3842926} and is given by
\begin{equation}\label{EESEffOrd}
  \text{err}(\EES) = \mathcal{O}\Big(\bar{c}^{\, \, -\frac{\gamma\rho_A\alpha\rho_Q \qee}{(\alpha\rho_Q+\gamma\rho_A) \qee
  +\gamma\rho_A\alpha\rho_Q}}\Big) .
\end{equation}
Here, we note that the same holds for the $\LIE$ scheme.
\\ \\
\begin{table}[tbp]
\begin{center}
\renewcommand{\arraystretch}{1.5}
 \begin{tabular}{|c|c|c|c|}\hline 
  Conditions & Optimal scheme & Optimal $M$, $N$, $K$ & \ EOC \ \\ 
  \hline\hline
  \ $q \leq \frac{1}{2}$ \ & $\DFMA = \EES$ & \eqref{Choice-MNK-DFMC2} & 
  $\frac{\gamma\rho_A\alpha\rho_Q q }{(\alpha\rho_Q+\gamma\rho_A) q +\alpha\rho_Q\gamma\rho_A}$ \\
  \hline 
  \ $\gamma \rho_A(2 q -1) \leq q$ \ $\wedge$ \ $q > \frac{1}{2}$ \ & $\DFMA$ & \eqref{Choice-MNK-DFMC2} & 
  $\frac{\gamma\rho_A\alpha\rho_Q q }{(\alpha\rho_Q+\gamma\rho_A) q +\alpha\rho_Q\gamma\rho_A}$ \\
  \hline 
  \ $q \leq \gamma \rho_A(2q -1) \leq 2q$ \ & $\DFMA$ & \eqref{Choice-MNK-DFM1C1} & 
  $\frac{\alpha\rho_Q}{2\alpha\rho_Q+1}$ \\
  \hline 
  \ $2q \leq \gamma \rho_A(2q -1)$ \ & $\DFMA = \MILA$ & \eqref{Choice-MNK-DFM1C1} & 
  $\frac{\alpha\rho_Q}{2\alpha\rho_Q+1}$ \\
  \hline 
 \end{tabular}
\caption{For a given parameter set, the conditions in this table
have to be checked in order to determine the optimal scheme 
among the schemes $\DFMA$, $\MILA$ and $\EES$. 
Here, 
$q=\qd=\qm$.
}
\label{Tab:CompareEffOrder}
\end{center}
\end{table}%
In order to determine the scheme which is most efficient
for the approximation of the solution for an SPDE of type~\eqref{SPDE} that 
needs not to fulfill a commutativity condition for the noise,
we have to compare the effective orders of convergence according to
the distinct parameter settings for the schemes $\DFMA$, $\MILA$ and $\EES$. 
\\ \\
If $\gamma \rho_A (2q-1) \leq q$ and $q \leq \frac{1}{2}$, then, 
it follows that $q= \qee$. Thus, the $\EES$ scheme and the $\DFMA$ scheme have the same 
effective order of convergence 
given in \eqref{StandardEffOrd-DFMC2} and \eqref{EESEffOrd}, whereas 
the $\MILA$ scheme obviously has a lower effective order of convergence given 
in \eqref{EffOrd-MIL1-C2}. 
\\ \\
If $\gamma \rho_A (2q-1) \leq q$ and $q > \frac{1}{2}$, then, 
it follows that $q> \qee = \frac{1}{2}$. Here, the $\DFMA$ scheme has obviously a 
higher effective order of convergence compared to the one of the $\EES$ scheme. 
Further, comparing the effective order of convergence of the $\EES$ scheme 
and the $\MILA$ scheme results in
\begin{equation*}
  \frac{\qm \gamma \rho_A \alpha \rho_Q}{(2\alpha\rho_Q+\gamma\rho_A) \qm
  + \gamma \rho_A \alpha\rho_Q} 
  \leq
  \frac{\frac{1}{2}\gamma\rho_A\alpha\rho_Q}{(\alpha\rho_Q+\gamma\rho_A)\frac{1}{2}
  +\gamma\rho_A\alpha\rho_Q} .
\end{equation*}
Here, it follows that the $\EES$ scheme has a higher order of convergence 
than the $\MILA$ scheme and thus the $\DFMA$ scheme attains the highest effective order 
of convergence in this case.
\\ \\
If $q \leq \gamma \rho_A (2q-1) \leq 2q$, then, it follows that $q>\qee=\frac{1}{2}$.
In this case, it holds for the effective orders of convergence of the $\EES$, the 
$\MILA$ and the $\DFMA$ scheme that
\begin{equation*}
  \frac{\frac{1}{2}\gamma\rho_A\alpha\rho_Q}{(\alpha\rho_Q+\gamma\rho_A)\frac{1}{2}
  +\gamma\rho_A\alpha\rho_Q}
  \leq 
  \frac{\qm \gamma \rho_A \alpha \rho_Q}{(2\alpha\rho_Q+\gamma\rho_A) \qm
  + \gamma \rho_A \alpha\rho_Q} 
  \leq
  \frac{\alpha\rho_Q}{2\alpha\rho_Q+1} .
\end{equation*}
Thus, the $\DFMA$ scheme is the one with the highest effective order of convergence
in the present case.
\\ \\
If $2q \leq \gamma \rho_A (2q-1)$, it holds that $q>\qee=\frac{1}{2}$.
In this case, the $\DFMA$ and the $\MILA$ scheme attain the same effective order of 
convergence given in \eqref{DFM1EffOrd-DFM1C1}. 
As a result of this, a comparison of the effective order of the $\EES$ scheme with
the one of the $\MILA$ scheme and $\DFMA$ scheme results in
\begin{equation*}
  \frac{\frac{1}{2}\gamma\rho_A\alpha\rho_Q}{(\alpha\rho_Q+\gamma\rho_A)\frac{1}{2}
  +\gamma\rho_A\alpha\rho_Q}
  \leq 
  \frac{\alpha\rho_Q}{2\alpha\rho_Q+1} .
\end{equation*}
In this case, the $\DFMA$ scheme and the $\MILA$ scheme have the same effective 
order of convergence which is higher than the one of the $\EES$ scheme.
\\ \\
The same holds true if the $\EES$ scheme is replaced by the $\LIE$ scheme.
We summarize the results of our comparison in Table~\ref{Tab:CompareEffOrder}
which shows that the $\DFMA$ scheme always attains the highest possible effective 
order of convergence. However, in the case of $\qd = \qee \leq \frac{1}{2}$, 
although the $\EES$ scheme and the $\DFMA$ scheme have the same effective order of convergence,
one may prefer the $\EES$ scheme because it requires less computational effort compared 
to the $\DFMA$ scheme, see Table~\ref{funcEval-1}. 
On the other hand, in the case of $2q \leq \gamma \rho_A (2q-1)$,
both the $\DFMA$ and the $\MILA$ scheme have the same optimal effective order of
convergence, which is higher than that of the $\EES$ scheme. Here, one may prefer the 
$\DFMA$ scheme because it needs less computational effort compared to the $\MILA$
scheme, see Table~\ref{funcEval-1}, and because it is derivative-free whereas 
one has to calculate the derivative of the operator $B$ for the $\MILA$ scheme. 
\\ \\
Finally, it has to be pointed out that the maximal 
effective order of convergence that can be attained is always bounded by $1/2$
independent of the given parameters whenever Algorithm~1 is applied to simulate 
the iterated stochastic integrals.
\\ \\
For completeness, we want to note that assumption (A5) 
as well as parts of (A3) do not have to be 
fulfilled for the exponential Euler scheme. This means that there might be 
parameter sets that are valid for the $\EES$ scheme but not for the $\DFM$ scheme and in these 
situations the exponential Euler scheme would be the method of choice. 
Moreover, it is not clear if the obtained upper error bounds 
are sharp and thus if the effective order of convergence may be further improved.
\subsection{The Case of a Finite-Dimensional $Q$-Wiener Process}
If the $Q$-Wiener process $W$ is finite-dimensional, i.e., if 
$|\{ j \in \mathcal{J} : \eta_j \neq 0 \}| < \infty$, the error 
estimate only depends on $M$ and $N$ provided we choose $K=|\{ j \in \mathcal{J} : \eta_j \neq 0 \}|$. 
Then, we obtain new solutions for $M$ and $N$ solving the optimization problem
that minimizes the error under the constraint of a prescribed computational
cost budget $\bar{c}$. Therefore, we compare once more the $\DFMA$ scheme, the $\MILA$ 
scheme and the $\EES$ scheme. Now, the computational cost required to approximate one trajectory
of the solution of SPDE~\eqref{SPDE}
by the $\DFMA$ scheme becomes $\bar{c} = \Oo(MN)+\Oo(M^{2q})$, for the $\MILA$ scheme 
we get $\bar{c}=\Oo(MN^2)+\Oo(M^{2q})$ and for the $\EES$ scheme it is $\bar{c}=\Oo(MN)$.
\\ \\
If $\gamma \rho_A (2q-1) \leq q$, the computational cost for the $\DFMA$ 
scheme is $\bar{c} = \Oo(MN)$ and solving the optimization problem yields
\begin{equation} \label{Choice-MN-DFM1C2}
    M = \Oo \Big(\bar{c}^{\, \frac{\gamma \rho_A}{\gamma \rho_A +q}} \Big), \quad \quad
    N = \Oo \Big(\bar{c}^{\, \frac{q}{\gamma \rho_A +q}} \Big) .
\end{equation}
Then, the effective order of convergence is given by
\begin{equation}\label{DFM1EffOrd-DFM1C2-Kendl}
  \text{err}(\DFMA) = \Oo \Big(\bar{c}^{\, \, -\frac{\gamma \rho_A q}
  {\gamma\rho_A+q}}\Big) .
\end{equation}
If $\gamma \rho_A (2q-1) \geq q$, then $q > \frac{1}{2}$ and $\bar{c} = \Oo(M^{2q})$
for the $\DFMA$ scheme. Here, optimization results in
\begin{equation} \label{Choice-MN-DFM1C1}
    M = \Oo \Big(\bar{c}^{\, \frac{\gamma \rho_A}{2 \gamma \rho_A q}} \Big), \quad \quad
    N = \Oo \Big(\bar{c}^{\, \frac{q}{2\gamma \rho_A q}} \Big) ,
\end{equation}
and the effective order of convergence can be calculated as
\begin{equation}\label{DFM1EffOrd-DFM1C1-Kendl}
  \text{err}(\DFMA) = \Oo \Big(\bar{c}^{\, \, -\frac{1}
  {2}}\Big) .
\end{equation}
Considering the $\MILA$ scheme, again two cases have to be distinguished:
If $\gamma \rho_A (2q-1) \leq 2q$, the $\MILA$ has computational cost $\bar{c} = \Oo(MN^2)$
and optimization yields
\begin{equation} \label{Choice-MN-MIL1C2}
    M = \Oo \Big(\bar{c}^{\, \frac{\gamma \rho_A}{\gamma \rho_A +2q}} \Big), \quad \quad
    N = \Oo \Big(\bar{c}^{\, \frac{q}{\gamma \rho_A +2q}} \Big) 
\end{equation}
and the effective order of convergence is given by
\begin{equation}\label{MIL1EffOrd-MIL1C2-Kendl}
  \text{err}(\MILA) = \Oo \Big(\bar{c}^{\, \, -\frac{\gamma \rho_A q}
  {\gamma\rho_A+2q}}\Big) .
\end{equation}
If $2q \leq \gamma \rho_A (2q-1)$, it follows that $q>\frac{1}{2}$ and that 
the $\MILA$ scheme attains the same computational cost $\bar{c}=\Oo(M^{2q})$ as the $\DFMA$
scheme in the second case. Thus, we also get \eqref{Choice-MN-DFM1C1} for $M$ and 
$N$, and also the same effective order of convergence as given by \eqref{DFM1EffOrd-DFM1C1-Kendl}.
\\ \\
Clearly, for the $\EES$ scheme it holds $\qee \leq \tfrac{1}{2}$ and thus
we get the same results as for the $\DFMA$ scheme given 
in \eqref{Choice-MN-DFM1C2} for $M$ and $N$ as well as by \eqref{DFM1EffOrd-DFM1C2-Kendl}
for the effective order of convergence with $q=\qee$.
\\ \\
\begin{table}[tbp]
\begin{center}
\renewcommand{\arraystretch}{1.5}
 \begin{tabular}{|c|c|c|c|}\hline 
  Conditions & Optimal scheme & Optimal $M$, $N$ & \ EOC \ \\ 
  \hline\hline
  \ $q \leq \frac{1}{2}$ \ & $\DFMA = \EES$ & \eqref{Choice-MN-DFM1C2} & 
  $\frac{\gamma \rho_A q}{\gamma\rho_A+q}$ \\
  \hline 
  \ $\gamma \rho_A(2 q -1) \leq q$ \ $\wedge$ \ $q > \frac{1}{2}$ \ & $\DFMA$ & \eqref{Choice-MN-DFM1C2} & 
  $\frac{\gamma \rho_A q}{\gamma\rho_A+q}$ \\
  \hline 
  \ $q \leq \gamma \rho_A(2q -1) \leq 2q$ \ & $\DFMA$ & \eqref{Choice-MN-DFM1C1} & 
  $\frac{1}{2}$ \\
  \hline 
  \ $2q \leq \gamma \rho_A(2q -1)$ \ & $\DFMA = \MILA$ & \eqref{Choice-MN-DFM1C1} & 
  $\frac{1}{2}$ \\
  \hline 
 \end{tabular}
\caption{In case of a finite-dimensional $Q$-Wiener process and $K=|\{ j \in \mathcal{J} : \eta_j \neq 0 \}| < \infty$, 
the conditions in this table have to be checked in order to determine the optimal scheme 
among the schemes $\DFMA$, $\MILA$ and $\EES$. 
Here, 
$q=\qd=\qm$.
}
\label{Tab:CompareEffOrder-Kendl}
\end{center}
\end{table}%
Finally, comparing the effective orders of convergence for the schemes under consideration,
we easily derive the results presented in Table~\ref{Tab:CompareEffOrder-Kendl}. Here, 
again the $\DFMA$ scheme performs better or at least as good as one of the other schemes.
Clearly, in the case of $\qd=\qm=\qee \leq \frac{1}{2}$ one may prefer the $\EES$ scheme or the 
$\LIE$ scheme although they have the same effective order of convergence as the $\DFMA$ scheme
because they are easier to implement. However, in the case of $2q \leq \gamma \rho_A(2q -1)$
where the $\DFMA$ scheme and the $\MILA$ scheme attain the same effective order of convergence
one may prefer the $\DFMA$ scheme because it needs less computational effort and because 
no derivative of the operator $B$ is needed by the $\DFMA$ scheme. Again, the 
effective order of convergence is always bounded by $1/2$ as for the infinite-dimensional 
noise case.
\section{Numerical Analysis} \label{NumExampleNonComm}
In this section, we compare the $\DFMA$ scheme to the $\MILA$ and 
the $\EES$ schemes to demonstrate the theoretical results presented above,
summarized in Tables~\ref{Tab:CompareEffOrder} and~\ref{Tab:CompareEffOrder-Kendl}.
Here, we approximate the mild solution of SPDE~\eqref{SPDE}, that is,
\begin{equation*}
	X_t = e^{At}\xi + \int_0^t e^{A(t-s)}F(X_s)\, \mathrm{d}s 
	+ \int_0^t e^{A(t-s)}B(X_s)\, \mathrm{d}W_s, \quad t\in(0,T].
\end{equation*}
For the numerical analysis, we consider the following setting.
We fix $H=U=L^2((0,1),\mathbb{R})$, set $T=1$,
and $\mathcal{I}=\mathcal{J} =\mathbb{N}$.
Let $A$ be the Laplace operator with Dirichlet 
boundary conditions. To be precise, $A=\frac{\Delta}{100}$ with eigenvalues 
$\lambda_i = \frac{\pi^2 i^2}{100}$ of $-A$
and eigenvectors $e_i= \sqrt{2} \sin(i\pi x)$ for $i\in\mathbb{N}$, $x\in(0,1)$
and on the boundary, we have
$X_t(0) = X_t(1) = 0$ for all $t\in(0,T]$.
The covariance operator $Q$ is defined by the eigenvalues $\eta_j = j^{-\rho_Q}$ 
for some $\rho_Q>1$ which is given separately 
for each example  below
and $\tilde{e}_j = \sqrt{2} \sin(j\pi x)$ for $j\in\mathbb{N}$, $x\in(0,1)$. 
For the operator $B$, we present
the general setting introduced for
the numerical analysis in~\cite{MR3842926}.
Define the functionals $\mu_{ij}:H_{\beta}\rightarrow \mathbb{R}$, 
$\phi_{ij}^k:H_{\beta}\rightarrow \mathbb{R}$ 
for $i,k\in \mathcal{I}$, $j\in\mathcal{J}$
such that $\phi_{ij}^k$ is the Fr\'{e}chet derivative of $\mu_{ij}$ in 
direction $e_k$ and let
\begin{align*}
	B(y)u =\sum_{i\in\mathcal{I}}\sum_{j\in\mathcal{J}} 
	\mu_{ij}(y)\langle u,\tilde{e}_j\rangle_U e_i,
\end{align*}
as well as
\begin{align*}
	B'(y)(B(y)v,u) =\sum_{i,k\in\mathcal{I}}\sum_{j,r\in\mathcal{J}} 
	\phi_{ij}^k(y)\mu_{kr}(y)\langle v,\tilde{e}_r\rangle_U \langle u,\tilde{e}_j\rangle_U e_i
\end{align*}
for $y\in H_{\beta}$ and $u,v\in U_0$, see~\cite[Sec.~5.3]{MR3842926}
for details.
\\ \\
We choose $\mu_{ij}(y) = \frac{\langle y, e_j \rangle_H}{i^p+j^4}$ 
for all $i \in \mathcal{I}$, $j\in\mathcal{J}$, 
$y\in H$ and some $p>1$ that differs in the examples presented below,
which leads to
$\phi_{ij}^k(y) = \left\{\begin{array}{rr} 0, & k \neq j \\ \frac{1}{i^p+j^4},
& k = j \end{array}\right.$ for all $i,k \in\mathcal{I}$, $j\in \mathcal{J}$, $y\in H$. 
This is the setting considered in~\cite{MCQMC2018}.
The assumptions (A1), (A2), and 
(A4) are obviously fulfilled. We only elaborate on (A3). 
By the definition of the $L(U,H_{\delta})$-norm and the operator $B$, we obtain
\begin{equation*}
\|B(y)\|_{L(U,H_{\delta})}
=\sup_{\substack{u\in U \\ \|u\|_U = 1}}\Big \|\sum_{i\in\mathcal{I}}
\sum_{j\in\mathcal{J}} \lambda_i^{\delta} \mu_{ij}(y) \langle u,\tilde{e}_j \rangle_U e_i \Big\|_H.
\end{equation*}
In the next steps, we employ the Parseval equality and the triangle inequality
\begin{align*}
\|B(y)\|_{L(U,H_{\delta})}
&=\sup_{\substack{u\in U \\ \|u\|_U = 1}}
\Big( \sum_{i\in\mathcal{I}} \Big| \sum_{j\in\mathcal{J}} \lambda_i^{\delta} \mu_{ij}(y)
\langle u,\tilde{e}_j \rangle_U \Big|^2 \Big)^{\frac{1}{2}} \\
&\leq \sup_{\substack{u\in U \\ \|u\|_U = 1}}
\Big( \sum_{i\in\mathcal{I}} \Big( \sum_{j\in\mathcal{J}} |\lambda_i^{\delta}|
\cdot | \mu_{ij}(y) | \cdot | \langle u,\tilde{e}_j \rangle_U | \Big)^2 \Big)^{\frac{1}{2}}.
\end{align*}
It holds by Parseval's equality that
\begin{align*}
\|y\|^2_{H_{\delta}} = \| (-A)^{\delta}y \|^2_H 
= \sum_{i\in\mathcal{I}} |\lambda_i^{\delta} \langle y,e_i \rangle_H|^2
\end{align*}
and therewith 
	\begin{equation} \label{NormParseval}
	|\langle y, e_j \rangle_H |^2 
	= \lambda_j^{-2\delta}|\lambda_j^{\delta} \langle y, e_j \rangle_H |^2
	\leq \lambda_j^{-2\delta} \|y\|_{H_{\delta}}^2
	\end{equation}
for all $j \in \mathcal{J}$.
As $|\langle u,\tilde{e}_j\rangle_U|^2 \leq 1$ by Parseval, we obtain
\begin{align*}
\|B(y)\|_{L(U,H_{\delta})}
&\leq \Big( \sum_{i\in\mathcal{I}} \lambda_i^{2\delta} 
 \Big( \sum_{j\in\mathcal{J}} |\mu_{ij}(y)| \Big)^2 \Big)^{\frac{1}{2}} \\
&= \Big( \sum_{i\in\mathcal{I}} \frac{\pi^{4\delta}i^{4\delta}}{100^{2\delta}} 
\Big( \sum_{j\in\mathcal{J}} \frac{| \langle y,e_j\rangle_H |}{i^p+j^4} \Big)^2 \Big)^{\frac{1}{2}} \\ 
&\leq \Big( \sum_{i\in\mathcal{I}} \frac{\pi^{4\delta}i^{4\delta}}{100^{2\delta}} 
\Big( \sum_{j\in\mathcal{J}} \frac{\lambda_j^{-\delta} \|y\|_{H_{\delta}} }{i^p+j^4} \Big)^2 \Big)^{\frac{1}{2}}.
\end{align*}
Then, for some $\varepsilon \in (0,2\delta)$, 
some $C_1=C_1(\varepsilon,\delta)>0$ and with 
$r=\frac{4}{3-\varepsilon+2\delta}>1$, $q=\frac{4}{1+\varepsilon-2\delta}>1$
such that $\frac{1}{r}+\frac{1}{q}=1$, Young's inequality gives the estimate
\begin{align*}
\|B(y)\|_{L(U,H_{\delta})}
&\leq \Big( \sum_{i\in\mathcal{I}}  i^{4\delta} 
\Big( \sum_{j\in\mathcal{J}} \frac{j^{-2\delta}}{
	 r^{\frac{1}{r}} \, q^{\frac{1}{q}} \, 
	i^{\frac{3-\varepsilon+2\delta}{4}p} \, j^{1+\varepsilon-2\delta}}
\Big)^2 \Big)^{\frac{1}{2}} \|y\|_{H_{\delta}} \\
 &\leq C_1  \Big( \sum_{i\in\mathcal{I}} i^{4\delta-\frac{3-\varepsilon+2\delta}{2}p} 
\Big)^{\frac{1}{2}} \|y\|_{H_{\delta}}.
\end{align*}
If for $\delta \in (0,\tfrac{1}{2})$ it holds that 
$p > \tfrac{2+8 \delta}{3+2 \delta}$, then it follows that
$\|B(y)\|_{L(U,H_{\delta})}\leq C(1+\|y\|_{H_{\delta}})$ 
for all $y\in H_{\delta}$.
\\ \\
Next, we compute the term
\begin{align*}
\|(-A)^{-\vartheta}B(y)Q^{-\alpha}\|_{L_{HS}(U_0,H)} 
&= \Big( \sum_{j\in\mathcal{J}} \|(-A)^{-\vartheta} B(y) 
Q^{-\alpha +\frac{1}{2} }\tilde{e}_j \|_H^2 \Big)^{\frac{1}{2}}
\end{align*}
for all $y\in H_{\gamma}$.
We rewrite the expression above to obtain
\begin{align*}
\|(-A)^{-\vartheta}B(y)Q^{-\alpha}\|_{L_{HS}(U_0,H)} 
&= \Big( \sum_{k\in\mathcal{I}} \sum_{j\in\mathcal{J}} |\langle (-A)^{-\vartheta} B(y)
Q^{-\alpha+\frac{1}{2}} \tilde{e}_j , e_k \rangle_H |^2 \Big)^{\frac{1}{2}} \\
&= \Big( \sum_{i\in\mathcal{I}} \sum_{j\in\mathcal{J}} | \lambda_i^{-\vartheta} \langle B(y)
Q^{-\alpha+\frac{1}{2}} \tilde{e}_j , e_i\rangle_H |^2 \Big)^{\frac{1}{2}} \\
&= \Big( \sum_{i\in\mathcal{I}} \sum_{j\in\mathcal{J}} \lambda_i^{-2\vartheta} 
| \langle B(y) \eta_j^{-\alpha+\frac{1}{2}} \tilde{e}_j , e_i\rangle_H |^2 \Big)^{\frac{1}{2}}.
\end{align*}
Here, we employed the definition of the operators $A$ and $Q$.
In the next step, we insert the definition of the operator $B$
\begin{align*}
\|(-A)^{-\vartheta}B(y)Q^{-\alpha}\|_{L_{HS}(U_0,H)} 
&= \Big( \sum_{i\in\mathcal{I}} \sum_{j\in\mathcal{J}} \lambda_i^{-2\vartheta}
 \eta_j^{-2\alpha+1}  |\mu_{ij}(y)|^2 \Big)^{\frac{1}{2}} \\
&= \Big( \sum_{i\in\mathcal{I}} \sum_{j\in\mathcal{J}} \frac{\pi^{-4\vartheta} i^{-4\vartheta}}{100^{-2\vartheta}}
j^{(2\alpha-1) \rho_Q} \frac{| \langle y,  e_j  \rangle_H |^2}{|i^p+j^4|^2} \Big)^{\frac{1}{2}}.
\end{align*}
By Parseval's equality and calculations as 
in~\eqref{NormParseval}, we obtain for some $C_2>0$ that
\begin{align*}
\|(-A)^{-\vartheta} B(y) Q^{-\alpha} \|_{L_{HS}(U_0,H)} 
&\leq C_2 \Big( \sum_{i\in\mathcal{I}} \sum_{j\in\mathcal{J}} i^{-4\vartheta}
j^{(2\alpha-1) \rho_Q} j^{-4\gamma} \frac{\|y \|_{H_{\gamma}}^2}{|i^p+j^4|^2} \Big)^{\frac{1}{2}}.
\end{align*}
Then, for all $\varepsilon\in(4\vartheta-1,4\vartheta-1+2p)$ such that 
$r=\frac{2p}{1+\varepsilon-4\vartheta}>1$, $q=\frac{2p}{2p-1-\varepsilon+4\vartheta}>1$, 
Young's inequality yields that
\begin{align*}
\|(-A)^{-\vartheta}B(y)Q^{-\alpha}\|_{L_{HS}(U_0,H)} 
&\leq C_2 \Big( \sum_{i\in\mathcal{I}} \sum_{j\in\mathcal{J}} i^{-4\vartheta}
j^{(2\alpha-1) \rho_Q - 4 \gamma} \frac{\|y \|_{H_{\gamma}}^2}{ \big( r^{\frac{1}{r}} \,
	i^{\frac{1+\varepsilon-4\vartheta}{2}} \, q^{\frac{1}{q}} \, j^{\frac{4p-2-2\varepsilon+8\vartheta}{p}} \big)^2} 
\Big)^{\frac{1}{2}} \\
&\leq C_3 \Big(\sum_{i\in\mathcal{I}} \frac{1}{i^{1+\varepsilon}}\Big)^{\frac{1}{2}}
\Big( \sum_{j\in\mathcal{J}} \frac{1}{j^{(1-2\alpha)\rho_Q+4\gamma+8-\frac{4}{p}
		-\frac{4\varepsilon}{p}+\frac{16\vartheta}{p}}} \Big)^{\frac{1}{2}} \|y \|_{H_{\gamma}}
\end{align*}
with $C_3=C_3(\varepsilon,\vartheta,p)>0$. Therefore,
$\|(-A)^{-\vartheta}B(y)Q^{-\alpha}\|_{L_{HS}(U_0,H)} \leq C (1+\|y\|_{H_{\gamma}})$ 
holds for all $y\in H_{\gamma}$ and some $C>0$ if 
$\alpha <\frac{7+\rho_Q+4\gamma}{2\rho_Q}$, if $p > \frac{1-4\vartheta}{2}$
and $\varepsilon \in (\max (0,4\vartheta-1),4\vartheta-1+2p)$. 
In the following examples, $p > \max( \tfrac{2+8\delta}{3+2\delta}, 
\tfrac{1-4\vartheta}{2})$ and $\rho_Q$ are specified and we select $\gamma$ 
and $\alpha$ to be maximal.
We do not state any other condition given in (A3) as these do not pose 
a restriction on the parameters but note that these are fulfilled as well.
%
%
Finally, we examine the commutativity condition~\eqref{Comm}.
On the one hand, it holds that
\begin{align*}
	\sum_{k\in\mathcal{I}} \phi_{im}^k(y)\mu_{kn}(y)=  \frac{1}{i^p+m^4}
	\frac{\langle y,e_n\rangle_H}{m^p+n^4}
\end{align*}
but on the other hand, it holds that
\begin{align*}
	\sum_{k\in\mathcal{I}} \phi_{in}^k(y)\mu_{km}(y) =  \frac{1}{i^p+n^4}
	\frac{\langle y,e_m\rangle_H}{n^p+m^4}
\end{align*}
for all $y\in H$ and all $i\in\mathcal{I}$, $m,n \in\mathcal{J}$. 
Obviously, these two expressions differ for some choice 
of $m, n \in\mathcal{J}$. Thus, the considered example does not fulfill
the commutativity condition~\eqref{Comm}. 
\subsection{Example 1} \label{Sub-Sec-Example-1}
In the first example, we set the parameters to $p=\tfrac{4}{3}$, 
$\rho_Q = 3$ and the nonlinearity $F(y) = 1-y$, $y\in H$.
This allows for $\beta \in [0,1)$ and we choose $\beta=0$.
Moreover, we set the initial value $\xi(x)= X_0(x) =0$ for all $x\in(0,1)$.
From condition $p > \max( \tfrac{2+8\delta}{3+2\delta}, 
\tfrac{1-4\vartheta}{2})$, it follows that $\delta \in (0,\tfrac{3}{8})$
and $\vartheta \in (0,\tfrac{1}{2})$. Therefore, we set
$\delta = \tfrac{3}{8}-\varepsilon_{\delta}$ for some arbitrarily 
small $\varepsilon_{\delta}>0$. From these parameter values, we compute 
$\gamma \in [\tfrac{3}{8}-\varepsilon_{\delta},\tfrac{7}{8}-\varepsilon_{\delta})$
and we thus choose $\gamma= \tfrac{7}{8}-\varepsilon_{\delta} -\varepsilon_{\gamma}$
for some arbitrarily small $\varepsilon_{\gamma}>0$.
As a result of this, it follows that $q= q_{\DFM} = q_{\MIL} = \frac{7}{8}-\hat{\varepsilon}$
with $\hat{\varepsilon} = \varepsilon_{\delta} + \varepsilon_{\gamma}>0$
arbitrarily small. From the condition $\alpha <\frac{7+\rho_Q+4\gamma}{2\rho_Q}$,
we directly get that $\alpha \in (0,\tfrac{27}{12}-\tfrac{2}{3} \hat{\varepsilon})$
and we choose $\alpha= \frac{27}{12}-\varepsilon_{\alpha}$ for some
arbitrarily small $\varepsilon_{\alpha} > \tfrac{2}{3} \hat{\varepsilon}>0$.
Thus, assumption (A3) holds, as discussed above.
Furthermore, condition (A5\ref{A5a}) is fulfilled as $\rho_Q >2$. 
\\ \\
With these parameters, we can identify the scheme that is superior.
For this example, it holds that $q < \gamma \rho_A (2q-1) \leq 2q$
for sufficiently small $\hat{\varepsilon}>0$.
Thus, the $\DFMA$ scheme is optimal, i.e., it is the scheme with the highest 
effective order of convergence according to Table~\ref{Tab:CompareEffOrder}. 
In order to compare the $\DFMA$ scheme to the other schemes 
under consideration, we calculate the effective orders of convergence
for each of the schemes.
We expect that the scheme $\DFMA$ obtains the highest effective 
order of convergence in this setting with 
\begin{equation*}
	\text{error}(\DFMA)= \Oo \big( \bar{c}^{-\frac{27-12\varepsilon_{\alpha}}{58-24 \varepsilon_{\alpha}}} \big)
\end{equation*}
given by~\eqref{DFM1EffOrd-DFM1C1}, i.e., $\text{EOC}(\DFMA) \approx \tfrac{27}{58}$.
Moreover, we fix some arbitrary $N\in\mathbb{N}$ and compute the relation $M = N^2$ and 
$K= \big\lceil N^{\frac{\frac{7}{4}-2 \hat{\varepsilon}}{\frac{27}{4}-3 \varepsilon_{\alpha}}} 
\big\rceil \approx \big\lceil N^{\frac{7}{27}} \big\rceil$ 
as given in \eqref{Choice-MNK-DFM1C1} for the implementation of the $\DFMA$ scheme. 
\\ \\
Considering the scheme $\MILA$, the effective order of convergence
for this scheme is given by \eqref{EffOrd-MIL1-C2} with
\begin{equation*}
	\text{error}(\MILA)= \Oo \Big( \bar{c}^{
	-\frac{\frac{189}{32} -\frac{27}{4} \hat{\varepsilon} -\frac{21}{8} \varepsilon_{\alpha}
		+3 \hat{\varepsilon} \varepsilon_{\alpha}}{\frac{115}{8} -6\varepsilon_{\alpha} -\hat{\varepsilon}}}
	\Big),
\end{equation*}
i.e., $\text{EOC}(\MILA) \approx \tfrac{189}{460}$. For this example, the relations between 
$N$, $K$ and $M$ for the $\MILA$ scheme given in \eqref{Choice-MNK-MIL1-C2}
are exactly the same as for the $\DFMA$ scheme.
\\ \\
For the $\EES$ scheme, on the other hand, we obtain from \eqref{Choice-MNK-DFMC2} 
for some arbitrarily fixed 
$N \in \mathbb{N}$ the relation $M= \big\lceil N^{\frac{7}{2} -4 \hat{\varepsilon}} \big\rceil
\approx \big\lceil N^{\frac{7}{2}} \big\rceil$ and $K= \big\lceil N^{\frac{\frac{7}{4} 
-2 \hat{\varepsilon}}{\frac{27}{4} -3 \varepsilon_{\alpha}}} \big\rceil \approx \big\lceil 
N^{\frac{7}{27}} \big\rceil$ as an optimal choice. The effective order of convergence
for the $\EES$ scheme is given as
\begin{equation*}
	\text{error}(\EES) = \Oo \Big( \bar{c}^{-\frac{\frac{189}{32} -\frac{21}{8} \varepsilon_{\alpha}
		-\frac{27}{4} \hat{\varepsilon} + 3 \hat{\varepsilon} \varepsilon_{\alpha}}{
		\frac{257}{16} -\frac{27}{4} \varepsilon_{\alpha} -\frac{29}{2} \hat{\varepsilon} 
		+6 \hat{\varepsilon} \varepsilon_{\alpha}}} \Big)
\end{equation*}
as stated in \eqref{EESEffOrd}, i.e., it holds $\text{EOC}(\EES) \approx \tfrac{189}{514}$. 
\\ \\
As a result of this, for this example, it holds that $\text{EOC}(\EES) 
< \text{EOC}(\MILA) < \text{EOC}(\DFMA)$ and thus the $\DFMA$ scheme
performs better than the other schemes.
For the numerical evaluation, we compare the schemes $\DFMA$, $\MILA$ and $\EES$ to
an approximation computed with the linear implicit Euler scheme
with $N = 2^6$, $K= \lceil2^{\frac{14}{9}}\rceil$ and $M = \lceil2^{\frac{35}{2}}\rceil$ that serves as the 
reference solution. We simulate 500 paths with each scheme 
and each prescribed computational cost to compare the mean-square error 
versus computational cost, see Figure~\ref{Plot:NC1}. Then, the slope indicates
the effective order of convergence if $\log$-$\log$-scales are used. This confirms 
that for this example the $\DFMA$ scheme performs significantly better than
the $\MILA$ scheme and the $\EES$ scheme.
The results are also stated in Table~\ref{Tab:NC1}.
\begin{table}[htbp!]
	\begin{small}
		\begin{center}
			\begin{tabular}{|p{0.6cm}|p{0.9cm}|p{0.9cm}||p{1.6cm}|p{1.6cm}|p{1.6cm}||p{1.6cm}|p{1.6cm}|p{1.6cm}|}\hline
				\multicolumn{3}{|c||}{} & \multicolumn{3}{c||}{$\DFMA$ scheme}      &    \multicolumn{3}{c|}{$\MILA$ scheme} \\ 
				\hline
				$N$   & $M$       &   $K$                & $\CC$                        &Error                  & Std                   &  $\CC$       & Error                  & Std     \\ \hline
				2   & 4     & $\lceil 2^{\frac{7}{27}} \rceil$  & 94          		& $3.77\cdot 10^{-2}$	&  $2.38\cdot 10^{-3}$   & 110			       &$3.78\cdot 10^{-2}$	&  $2.30\cdot 10^{-3}$   \\ \hline
				4   & $2^4$   & $\lceil 2^{\frac{14}{27}} \rceil$   & 864			& $2.95\cdot 10^{-2}$ 	&  $1.25\cdot 10^{-3}$	& 1248		               &$2.95\cdot 10^{-2}$ &  $1.25\cdot 10^{-3}$  \\ \hline		
				8   & $2^6$   & $\lceil 2^{\frac{21}{27}} \rceil$             & 8481			& $1.81\cdot 10^{-2}$	&  $5.33\cdot 10^{-4}$   & 15649 	 &$1.81\cdot 10^{-2}$	&  $5.15\cdot 10^{-4}$   \\ \hline
				16  & $2^8$   & $\lceil 2^{\frac{28}{27}} \rceil$   & 127744			& $6.84\cdot 10^{-3}$ 	&  $8.63\cdot 10^{-5}$ 	& 312064	     &$6.84\cdot 10^{-3}$ 	&  $8.31\cdot 10^{-5}$   \\ \hline
				32  & $2^{10}$  & $\lceil 2^{\frac{35}{27}} \rceil$  & 1344631			& $1.85\cdot 10^{-3}$	&  $4.84\cdot 10^{-5}$	& 4392055	   &$1.85\cdot 10^{-3}$	&  $5.23\cdot 10^{-5}$   \\ \hline
			\end{tabular}
			\quad \\[0.2cm]
			\begin{tabular}{|p{0.6cm}|p{0.9cm}|p{0.9cm}||p{1.6cm}|p{1.6cm}|p{1.6cm}|}\hline
				\multicolumn{3}{|c||}{} & \multicolumn{3}{|c|}{$\EES$ scheme} \\ 
				\hline
				$N$    & $M$       & $K$                & $\CC$                 & Error                    & Std   \\ \hline
				2    &$\lceil 2^{\tfrac{7}{2}} \rceil$   & $\lceil 2^{\frac{7}{27}} \rceil$   & 96			&$2.65\cdot 10^{-2}$ 	&  $2.46\cdot 10^{-3}$  \\ \hline
				4    &$2^{7}$   &$\lceil 2^{\frac{14}{27}} \rceil$  & 1792			&$3.06\cdot 10^{-2}$ 	&  $1.41\cdot 10^{-3}$   \\ \hline		
				8    &$\lceil 2^{\tfrac{21}{2}} \rceil$ & $\lceil 2^{\frac{21}{27}} \rceil$     & 37674		&$1.83\cdot 10^{-2}$ 	&  $5.11\cdot 10^{-4}$   \\ \hline
				16   &$2^{14}$ & $\lceil 2^{\frac{28}{27}} \rceil$ & 1097728 		&$6.81\cdot 10^{-3}$     &  $1.15\cdot 10^{-4}$   \\ \hline
				32   &$\lceil 2^{\tfrac{35}{2}} \rceil$ & $\lceil 2^{\frac{35}{27}} \rceil$ & 24282684 		&$1.83\cdot 10^{-3}$     &  $4.74\cdot 10^{-5}$  \\ \hline
			\end{tabular}
		\end{center}
	\end{small}
	\caption[Error and standard deviation for Example~1]{Error and standard deviation for
		Example~1 obtained from 500 paths. The computational cost $\CC$ is computed as
			$\CC(\DFMA) = MN+2MNK+MK(1+2M^{2q-1})$, $\CC(\text{\MILA}) = MN+MNK+MN^2K+MK(1+2M^{2q-1})$,
			and $\CC(\EES) =  MN+MNK+MK$.}
	\label{Tab:NC1}
\end{table}
\begin{figure}[H]
	\begin{center}
		\includegraphics[height = 6cm, width = 0.49\textwidth]{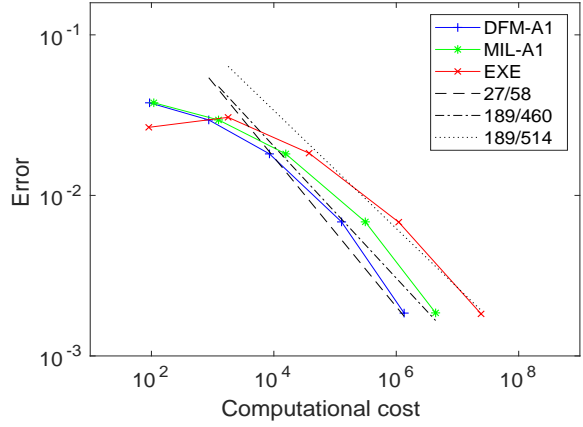}
		\caption[Error against computational cost for Example~1]{Error against computational cost for
			Example~1 computed from 500 paths for $N \in\{ 2, 4, 8, 16, 32\}$ in $\log$-$\log$ scale.}
		\label{Plot:NC1}
	\end{center}
\end{figure}
\subsection{Example 2} \label{Sub-Sec-Example-2}
Here, we choose a smaller value $p = \tfrac{44}{41}$ and the 
same covariance operator $Q$ as in Example~1
with $\rho_Q =3$. Thus, condition (A5\ref{A5a}) is fulfilled. As 
in Example~1, we consider the nonlinearity 
$F(y) = 1-y$, $y\in H$ and choose $\beta=0$. 
Again, the initial value is chosen as $\xi(x)= X_0(x) =0$ for all $x\in(0,1)$.
Further, we calculate
the condition $\delta \in (0, \tfrac{5}{24})$ and choose $\delta 
= \tfrac{5}{24} -\varepsilon_{\delta}$ for some arbitrarily small
$\varepsilon_{\delta} >0$. Then, we get $\gamma \in [ \tfrac{5}{24},
\tfrac{17}{24} -\varepsilon_{\delta} )$ and we set $\gamma = 
\tfrac{17}{24} -\varepsilon_{\delta} - \varepsilon_{\gamma}$ 
for some arbitrarily small $\varepsilon_{\gamma}>0$. This implies
$q= q_{\DFMA} = q_{\MILA} = \tfrac{17}{24} - \hat{\varepsilon}$
with $\hat{\varepsilon} = \varepsilon_{\delta} + \varepsilon_{\gamma} >0$
arbitrarily small. Moreover, one can choose $\vartheta \in (0, \tfrac{1}{2})$
arbitrarily. Finally, we calculate that $\alpha \in (0, \tfrac{77}{36} 
-\tfrac{2}{3} \hat{\varepsilon} )$ and therefore set $\alpha = \tfrac{77}{36}
- \varepsilon_{\alpha}$ with $\varepsilon_{\alpha} > \tfrac{2}{3} 
\hat{\varepsilon} > 0$ arbitrarily small. 
\\ \\
Checking the conditions in Table~\ref{Tab:CompareEffOrder}, 
we are in the case of $q> \tfrac{1}{2}$ and $\gamma\rho_A(2q-1) \leq q$
for sufficiently small $\hat{\varepsilon}>0$.
In this case, the optimal effective order of convergence is obtained by the 
$\DFMA$ scheme according to Table~\ref{Tab:CompareEffOrder}.
For the $\DFMA$ scheme, we get from \eqref{StandardEffOrd-DFMC2}
that
\begin{equation*}
	\text{error}(\DFMA) = \Oo \Big( \bar{c}^{-\frac{\frac{1309}{144} 
		-\frac{17}{4} \varepsilon_{\alpha} -\frac{77}{6} \hat{\varepsilon}}{
		\frac{62}{3} -9 \varepsilon_{\alpha} -2 \hat{\varepsilon}}} \Big) ,
\end{equation*}
i.e., $\text{EOC}(\DFMA) \approx \tfrac{1309}{2976}$. The optimal
choice of $M$ and $K$ given some $N \in \mathbb{N}$ is then 
determined in \eqref{Choice-MNK-DFMC2}, which results in 
$M = N^2$ and $K = \big\lceil N^{\frac{\frac{17}{12} -2 \hat{\varepsilon}}{
\frac{77}{12} -3 \varepsilon_{\alpha}}} \big\rceil \approx \big\lceil 
N^{\frac{17}{77}} \big\rceil$.
\\ \\
Considering the $\MILA$ scheme, we obtain from \eqref{EffOrd-MIL1-C2}
the effective order of convergence 
\begin{equation*}
	\text{error}(\MILA) = \Oo \Big( \bar{c}^{-\frac{\frac{1309}{144} 
		-\frac{17}{4} \varepsilon_{\alpha} -\frac{77}{6} \hat{\varepsilon}}{
		\frac{325}{12} -12 \varepsilon_{\alpha} -2 \hat{\varepsilon}}} \Big) ,	
\end{equation*}
i.e., it holds that $\text{EOC}(\MILA) \approx \frac{1309}{3900}$.
Given some $N \in \mathbb{N}$, the optimal choice for $M$ and $K$
is given in \eqref{Choice-MNK-MIL1-C2} and yields the same results 
as for the $\DFMA$ scheme in this example.
\\ \\
For the Euler scheme, it holds $q_{\EES}=\frac{1}{2}$ which in turn
yields with \eqref{Choice-MNK-DFMC2} that $M = \big\lceil 
N^{\frac{17}{6} +4 \hat{\varepsilon}} \big\rceil \approx \big\lceil N^{\frac{17}{6}} \big\rceil$
and $K= \big\lceil N^{\frac{\frac{17}{12} -2 \hat{\varepsilon}}{\frac{77}{12} -3 \varepsilon_{\alpha}}} \big\rceil
\approx \big\lceil N^{\frac{17}{77}} \big\rceil$. For the effective order of convergence,
we obtain 
\begin{equation*}
	\text{error}(\EES) =\Oo \Big( \bar{c}^{-\frac{\frac{1309}{288} -\frac{17}{8} \varepsilon_{\alpha} 
		- \frac{77}{12} \hat{\varepsilon}}{\frac{1873}{144} - \frac{23}{4} \varepsilon_{\alpha}
		- \frac{83}{6} \hat{\varepsilon} + 6 \varepsilon_{\alpha} \hat{\varepsilon}}} \Big) ,
\end{equation*}
i.e., it holds $\text{EOC}(\EES) \approx \tfrac{1309}{3746}$.
\\ \\
Now, if we compare the effective orders of convergence for the schemes under
consideration, then we have $\text{EOC}(\MILA) < \text{EOC}(\EES) < \text{EOC}(\DFMA)$.
In this example, again the $\DFMA$ scheme performs best with the highest 
effective order of convergence and the original Milstein scheme
$\MILA$ has the lowest effective order of convergence, which is even less than that 
of the $\EES$ scheme. As in Example~1, we substitute the exact solution 
with an approximation computed by the linear implicit Euler scheme. Precisely, we choose
$N = 2^6$, $K= \lceil2^{\frac{102}{77}}\rceil$ and $M = \lceil2^{\frac{85}{6}}\rceil$ for the computation of 
the reference solution and compute 500 paths. Table~\ref{Tab:NC2} and Figure~\ref{Plot:NC2}
show the difference in the effective order of convergence between the 
derivative-free Milstein type scheme $\DFMA$, the Milstein scheme $\MILA$ 
and the exponential Euler scheme $\EES$.
\begin{table}[htbp!]
	\begin{small}
		\begin{center}
			\begin{tabular}{|p{0.6cm}|p{0.9cm}|p{0.9cm}||p{1.6cm}|p{1.6cm}|p{1.6cm}||p{1.6cm}|p{1.6cm}|p{1.6cm}|}\hline
				\multicolumn{3}{|c||}{} & \multicolumn{3}{c||}{$\DFMA$ scheme}      &    \multicolumn{3}{c|}{$\MILA$ scheme}       \\ \hline
				$N$   & $M$       & $K$                & $\CC$                &Error             & Std                & $\CC$                          & Error                  & Std     \\ \hline
				2   & $2^2$  & $\lceil 2^{\frac{17}{77}} \rceil$ & 77  & $4.42\cdot 10^{-2}$	&  $4.10\cdot 10^{-3}$   & 93	&$4.43\cdot 10^{-2}$	&  $4.16\cdot 10^{-3}$   \\ \hline
				4   & $2^4$   & $\lceil 2^{\frac{34}{77}} \rceil$   & 556	& $3.56\cdot 10^{-2}$ 	&  $2.08\cdot 10^{-3}$	& 940  &$3.56\cdot 10^{-2}$ 	&  $2.08\cdot 10^{-3}$  \\ \hline		
				8   & $2^6$   & $\lceil 2^{\frac{51}{77}} \rceil$   & 4137	& $2.13\cdot 10^{-2}$	&  $8.73\cdot 10^{-4}$   & 11305  &$2.13\cdot 10^{-2}$	&  $8.61\cdot 10^{-4}$  \\ \hline
				16  & $2^8$   & $\lceil 2^{\frac{68}{77}} \rceil$    & 31314 & $8.66\cdot 10^{-3}$ 	&  $5.30\cdot 10^{-4}$ 	& 154194  &$8.66\cdot 10^{-3}$ 	&  $5.33\cdot 10^{-4}$   \\ \hline
				32  & $2^{10}$  & $\lceil 2^{\frac{85}{77}} \rceil$   & 342791 & $3.10\cdot 10^{-3}$	&  $3.08\cdot 10^{-4}$	& 3390215 &$3.09\cdot 10^{-3}$	&  $3.07\cdot 10^{-4}$   \\ \hline
			\end{tabular}
			\quad \\[0.2cm]
			\begin{tabular}{|p{0.6cm}|p{0.9cm}|p{0.9cm}||p{1.6cm}|p{1.6cm}|p{1.6cm}|}\hline
				\multicolumn{3}{|c||}{} &\multicolumn{3}{|c|}{$\EES$ scheme}   \\ \hline
				$N$    & $M$       & $K$                & $\CC$                 & Error                    & Std  \\ 
				\hline
				2    &$\lceil 2^{\frac{17}{6}} \rceil$   & $\lceil 2^{\frac{17}{77}} \rceil$    &  64		&$3.48\cdot 10^{-2}$ 	   &  $4.07\cdot 10^{-3}$  \\ \hline
				4    &$\lceil 2^{\frac{17}{3}} \rceil$   & $\lceil 2^{\frac{34}{77}} \rceil$   &  714		&$3.70\cdot 10^{-2}$ 	   &  $1.43\cdot 10^{-3}$ \\ \hline		
				8    &$\lceil 2^{\frac{17}{2}} \rceil$ & $\lceil 2^{\frac{51}{77}} \rceil$   &  9438 	&$2.22\cdot 10^{-2}$ 	   &  $9.95\cdot 10^{-4}$ \\ \hline
				16   &$\lceil 2^{\frac{34}{3}} \rceil$ & $\lceil 2^{\frac{68}{77}} \rceil$    &  129050 		&$9.07\cdot 10^{-3}$    &  $5.73\cdot 10^{-4}$\\ \hline
				32   &$\lceil 2^{\frac{85}{6}} \rceil$ & $\lceil 2^{\frac{85}{77}} \rceil$   &  2409221 		&$3.24\cdot 10^{-3}$	   &  $3.44\cdot 10^{-4}$ \\ \hline
			\end{tabular}
		\end{center}
	\end{small}
	\caption[Error and standard deviation for Example~2]{Error and standard deviation for
		Example~2 obtained from 500 paths. The computational cost $\CC$ is computed as
			$\CC(\DFMA) = MN+2MNK+MK(1+2M^{2q-1})$, $\CC(\text{MIL1}) = MN+MNK+MN^2K+MK(1+2M^{2q-1})$,
			and $\CC(\EES) =  MN+MNK+MK$.}
	\label{Tab:NC2}
\end{table}
\begin{figure}[H]
	\begin{center}
		\includegraphics[height = 6cm, width = 0.49\textwidth]{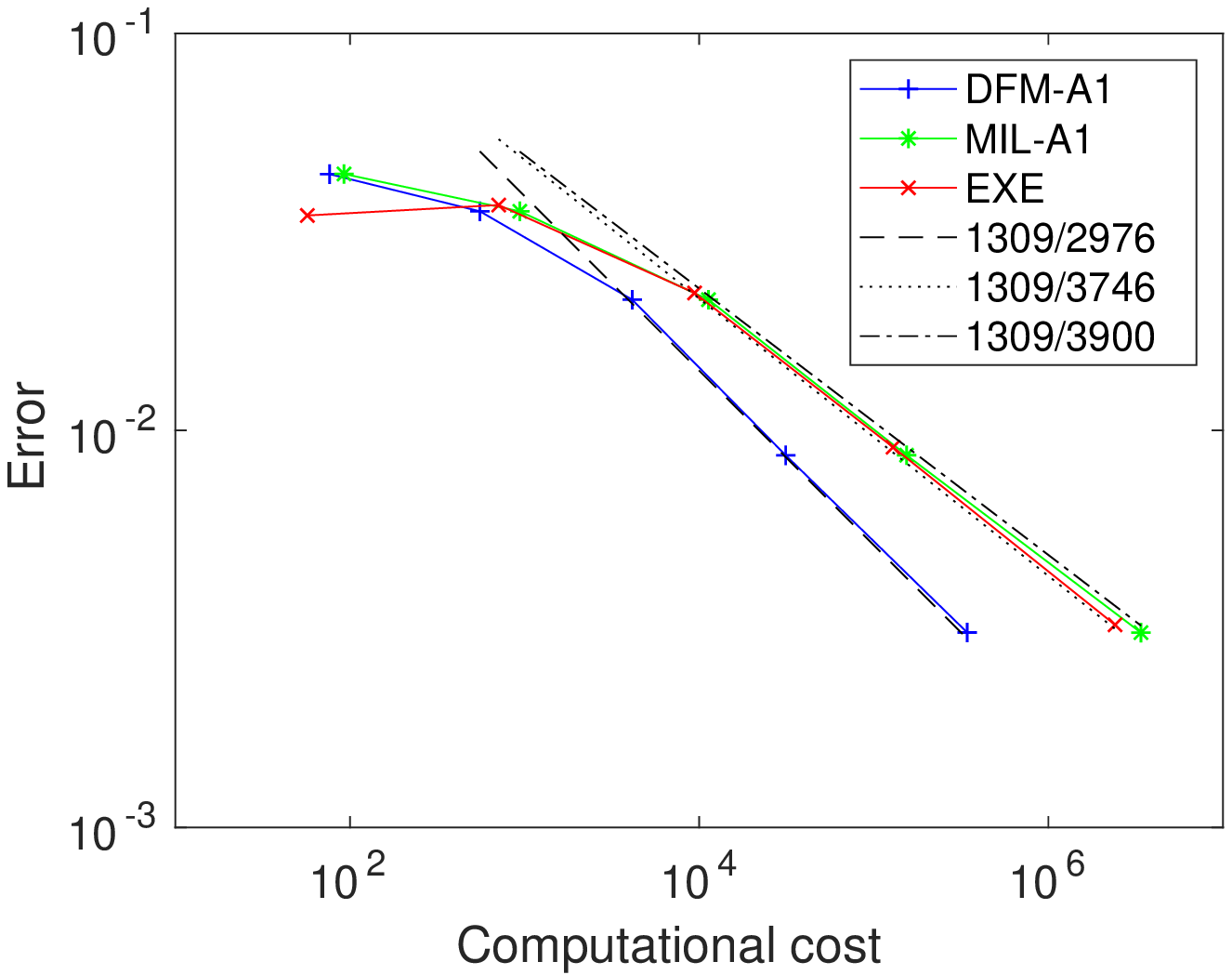}
		\caption[Error against computational cost for Example~2]{Error against computational cost
			for Example~2 computed from 500 paths for $N \in\{ 2, 4, 8, 16, 32\}$ in log-log scale.}
		\label{Plot:NC2}
	\end{center}
\end{figure}
\subsection{Example 3} \label{Sub-Sec-Example-3}
Compared to the first two examples, we choose a different nonlinearity 
for the third example in order to obtain restrictions for the parameter $\beta$.
Therefore, we consider the mapping $F \colon H_{\beta} \to H$ given by 
\begin{align*}
F(v) = \sum_{i \in \mathcal{I}} f_i(v) \, e_i
\end{align*}
for $v \in H_{\beta}$ with some $f_i \colon H_{\beta} \to \mathbb{R}$ 
for $i \in \mathcal{I}$. In this example,
we choose $f_i(v) = i^{-s} \sin(i^r \langle v,e_i \rangle_H )$ for $v \in H_{\beta}$,
$s> \tfrac{1}{2}$, $r \leq \min (s, 2 \beta + \tfrac{s}{2} )$ and $i \in \mathcal{I}$. Then, we get
\begin{align*}
	\| F(v) \|_H^2 &= \sum_{i \in \mathcal{I}} |f_i(v) |^2 
	= \sum_{i \in \mathcal{I}} \frac{| \sin(i^r \langle v,e_i \rangle_H) |^2}{i^{2s}}
	< \infty .
\end{align*}
Further, $F$ is twice continuously Fr\'{e}chet differentiable and it holds
\begin{align*}	
	\sup_{v \in H_{\beta}} \| F'(v) \|_{L(H)}^2 &= \sup_{v \in H_{\beta}} 
	\sup_{\substack{u \in H \\ \| u \|_H = 1}} \sum_{i \in \mathcal{I}} 
	\Big| \sum_{k \in \mathcal{I}} \frac{\partial f_i}{\partial v_k}(v) \, 
	\langle u,e_k \rangle_H \Big|^2 \\
	&= \sup_{v \in H_{\beta}} \sup_{\substack{u \in H \\ \| u \|_H = 1}} \sum_{i \in \mathcal{I}}
	i^{2(r-s)} | \cos(i^r \langle v,e_i \rangle_H ) |^2 \, | \langle u,e_i \rangle_H |^2 \\
	&\leq \sup_{\substack{u \in H \\ \| u \|_H = 1}} \sum_{i \in \mathcal{I}} | \langle u,e_i \rangle_H |^2 
	= 1,
\end{align*}
because $r \leq s$. 
Further, considering the second Fr\'{e}chet derivative, we get
\begin{align*}
	\sup_{v \in H_{\beta}} \| F''(v) \|_{L^{(2)}(H_{\beta},H)}^2 &= \sup_{v \in H_{\beta}}
	\sup_{\substack{u,w \in H_{\beta} \\ \| u \|_{H_{\beta}} = \| w \|_{H_{\beta}} = 1 }}
	\sum_{i \in \mathcal{I}}  \Big| \sum_{k,l \in \mathcal{I}} \frac{\partial^2 f_i}{\partial v_k 
	\partial v_l}(v) \, \langle u,e_k \rangle_H \, \langle w,e_l \rangle_H \Big|^2 \\
	&= \sup_{v \in H_{\beta}} \sup_{\substack{u,w \in H_{\beta} \\ \| u \|_{H_{\beta}} = \| w \|_{H_{\beta}} = 1 }}
	\sum_{i \in \mathcal{I}} i^{4r-2s} | \sin(i^r \langle v,e_i \rangle_H ) |^2
	| \langle u,e_i \rangle_H |^2 | \langle w,e_i \rangle_H |^2 \\
	&\leq \sup_{\substack{u,w \in H_{\beta} \\ \| u \|_{H_{\beta}} = \| w \|_{H_{\beta}} = 1 }}
	\sum_{i \in \mathcal{I}} i^{4r-2s} | \langle u,e_i \rangle_H |^2 | \langle w,e_i \rangle_H |^2 \\
	&\leq \frac{100^2}{\pi^4} \sup_{\substack{u \in H_{\beta} \\ \| u \|_{H_{\beta}} = 1 }}
	\sum_{i \in \mathcal{I}} i^{4r-2s-4 \beta} | \langle u,e_i \rangle_H |^2 \\
	&\leq \frac{100^4}{\pi^8} \sup_{\substack{u \in H_{\beta} \\ \| u \|_{H_{\beta}} = 1 }} \| u \|_{H_{\beta}}^2
	= \frac{100^4}{\pi^8} < \infty,
\end{align*}
since $\| z \|_{H_{\beta}}^2 = \tfrac{\pi^4}{100^2} \sum_{i \in \mathcal{I}} i^{4 \beta}
| \langle z,e_i \rangle_H |^2$ for any $z \in H_{\beta}$ and because 
$r \leq \min (s, 2 \beta + \tfrac{s}{2} )$. Thus, assumption (A2) is fulfilled. \\ \\
Again, we choose $\rho_Q = 3$.
Moreover, we select $r=s=\tfrac{7}{2}$ in the definition of $F$ and $p=4$ in the 
definition of the operator $B$.
As the initial value, we choose $X_0 = \xi \in H$ with $\langle \xi(x),e_i \rangle_H 
= i^{-2}$ for 
$x\in(0,1)$ and $i \in \mathcal{I}$. First, we calculate $\beta \in [\tfrac{7}{8},1)$
from the condition $r \leq \min ( s, 2 \beta + \tfrac{s}{2} )$. Therefore, we
choose $\beta = \tfrac{7}{8}$ minimal possible. Analogously to Example~1,
we derive $\delta, \vartheta \in (0,\tfrac{1}{2})$ and choose $\delta=\tfrac{1}{2}
-\varepsilon_{\delta}$ and $\vartheta = \tfrac{1}{2} -\varepsilon_{\vartheta}$
for arbitrarily small $\varepsilon_{\delta}, \varepsilon_{\vartheta} >0$. Then, we
choose $\gamma \in [ \frac{7}{8}, 1 -\varepsilon_{\delta})$ maximal,
i.e., we choose $\gamma = 1 -\varepsilon_{\delta} -\varepsilon_{\gamma}$
for arbitrarily small $\varepsilon_{\gamma}>0$. Let $\hat{\varepsilon} = 
\varepsilon_{\delta} + \varepsilon_{\gamma} >0$ be arbitrarily small. It follows 
that $q=q_{\DFMA} = q_{\MILA} = \tfrac{1}{4} -2 \hat{\varepsilon}$.
Finally, we calculate that $\alpha \in (0, \tfrac{7}{3} -\tfrac{2}{3} \hat{\varepsilon})$ and
we set $\alpha = \tfrac{7}{3} - \varepsilon_{\alpha}$ for some arbitrarily small
$\varepsilon_{\alpha} > \tfrac{2}{3} \hat{\varepsilon} >0$. \\ \\
Since we have $q \leq \tfrac{1}{2}$, the optimal schemes are the $\EES$ 
scheme and the $\DFMA$ scheme, both attaining the same effective 
order of convergence for this example, see Table~\ref{Tab:CompareEffOrder}.
Taking into account all parameters, we get from \eqref{StandardEffOrd-DFMC2}
and \eqref{EESEffOrd} that
\begin{align*}
	\text{error}(\EES / \DFMA) 
	= \Oo \Big( \bar{c}^{-\frac{\frac{7}{2} -33 \hat{\varepsilon} 
		+ \frac{83}{2} \hat{\varepsilon}^2 -12 \hat{\varepsilon}^3}{\frac{65}{4} 
		-\frac{65}{2} \hat{\varepsilon} -\frac{27}{4} \varepsilon_{\alpha} 
		+ 12 \varepsilon_{\alpha} \hat{\varepsilon} + 4 \hat{\varepsilon}^2}} \Big) ,
\end{align*}
i.e., for the effective order of convergence it holds that
$\text{EOC}(\EES) = \text{EOC}(\DFMA) \approx \tfrac{14}{65}$.
For some arbitrarily fixed $N \in \mathbb{N}$, we obtain for the 
$\EES$ scheme as well as for the $\DFMA$ scheme from \eqref{Choice-MNK-DFMC2} that
$M = \big\lceil N^{\frac{2 -2 \hat{\varepsilon}}{\frac{1}{4} -2 \hat{\varepsilon}}}
\big\rceil \approx N^8$ and 
$K = \big\lceil N^{\frac{2 -2 \hat{\varepsilon}}{7 -3 \varepsilon_{\alpha}}}
\big\rceil \approx \lceil N^{\frac{2}{7}} \rceil$ as the optimal choice. In this case, the computation of the double integrals is not expensive as it holds that $D \geq D_1 = M^{-\frac{1}{2}-\varepsilon}$ for some $\varepsilon >0$ such that $D=1$ can be fixed or it can even be neglected.
On the other hand, for the $\MILA$ scheme, we compute from 
\eqref{EffOrd-MIL1-C2} that
\begin{align*}
	\text{error}(\MILA) 
	= \Oo \Big( \bar{c}^{-\frac{\frac{7}{2} -33 \hat{\varepsilon} 
		+ \frac{83}{2} \hat{\varepsilon}^2 -12 \hat{\varepsilon}^3}{18 
		-\frac{93}{2} \hat{\varepsilon} -\frac{15}{2} \varepsilon_{\alpha} 
		+ 18 \varepsilon_{\alpha} \hat{\varepsilon} + 4 \hat{\varepsilon}^2}} \Big) ,
\end{align*}
which gives us the effective order of convergence 
$\text{EOC}(\MILA) \approx \tfrac{7}{36}$. Moreover, the optimal
choice for $M$ and $K$ given some $N \in \mathbb{N}$ can
be calculated from \eqref{Choice-MNK-MIL1-C2} to be exactly the 
same as for the $\EES$ scheme and the $\DFMA$ scheme. 
For this example, we have $\text{EOC}(\MILA) < \text{EOC}(\EES) 
= \text{EOC}(\DFMA)$. 
For this example , the computational effort involved in computing a convergence
plot is very high due to the relation $M\approx N^8$. Therefore, we do not 
present a convergence plot for this setting.
\\ \\
The examples presented above confirm the theoretical 
analysis that we conducted in Section~\ref{Sec:EffOrder}.
The numerical experiments show that the derivative-free Milstein type scheme for equations
with non-commutative noise defined in~\eqref{DFM},
in combination with 
Algorithm~1, has always at least the same and it in many 
cases an even higher effective order of convergence
compared to the exponential Euler scheme and the original
Milstein scheme. 
\section{Conclusion}\label{Sec:Conclusion}
We proposed the derivative-free Milstein type scheme $\DFM$ for the approximation 
of the mild solution of SPDEs that need not fulfill a commutativity condition for the noise
and we proved an upper bound for the mean-square error.
As the main novelty, the introduced $\DFM$ scheme is derivative-free and has computational 
cost $\Oo(N K M)$ which is of the same magnitude as for the Euler schemes $\EES$ and $\LIE$. 
This is a significant reduction of the 
computational complexity compared to the original Milstein scheme $\MIL$ that is not derivative-free
and which has computational cost $\Oo(N^2 K M)$.
In addition, the convergence of the $\DFM$ method is proved if it is combined with any suitable 
simulation method for the iterated stochastic integrals. 
As an example, the effective order of convergence of the $\DFM$ scheme combined with Algorithm~1 
in \cite{MR3949104} for the simulation of the iterated stochastic integrals is analyzed in detail. 
For Algorithm~1, the effective order of convergence of the $\DFM$ scheme is at least that for 
the Euler schemes or the Milstein scheme $\MIL$ 
and turns out to be even significantly higher for many parameter settings depending 
on the specific SPDE to be approximated. Thus, in many cases the proposed $\DFM$ scheme outperforms 
the Euler schemes as well as the original Milstein scheme.
\\ \\ 
The maximal possible effective 
order of convergence that can be attained by the $\DFM$ scheme combined with Algorithm~1 is bounded 
by $1/2$, which is in accordance with 
the upper bound for the order of strong convergence in case of finite-dimensional SDEs if Algorithm~1
is applied, see also \cite{MR609181}. 
However, in contrast to the finite-dimensional SDEs setting, for SPDEs the Euler schemes 
often attain some effective order of convergence less than $1/2$. This gap in the
order of convergence for the Euler schemes is the reason why the use of higher order approximation 
methods can be reasonable and which is in strong contrast to the finite-dimensional SDE setting. 
To the best of the authors knowledge, this is the first attempt to give a rigorous analysis of the 
error versus computational cost for higher order approximation methods applied to SPDEs 
without any commutativity condition where the computational cost for the approximation of 
iterated stochastic integrals is incorporated within the framework of a cost model.
It remains an open question whether 
the application of higher order numerical methods that incorporate further iterated stochastic
integrals from the stochastic Taylor expansion may close the gap for the order of convergence 
to the upper bound of $1/2$ if, e.g., naive approximations like Algorithm~1
are applied for the approximation of these iterated stochastic integrals. 
As a result of this, higher order approximation methods may be of strong interest, especially 
in the case of SPDEs. On the other hand, it may be possible to overcome the upper bound of 
$1/2$ for the order of convergence if some more sophisticated algorithm for the simulation 
of the iterated stochastic integrals is combined with the $\DFM$ scheme, see e.g., Algorithm~2 in 
\cite{MR3949104}. 
\section{Proofs}\label{Proof}\label{Sec:Proofs} 
Here, we give the proof of the convergence 
result for the derivative-free Milstein scheme~\eqref{DFM}
as stated in Theorem~\ref{Error:DFM}. Moreover, we prove the estimate 
given in Theorem~\ref{Thm:ErrorTotal}
which incorporates the approximation of the 
stochastic double integrals additionally.
In the following, we always denote $Y_m = Y_m^{N,K,M}$ for simplicity
and let $I_{(i,j),l} = (\eta_i \eta_j)^{-\frac{1}{2}} I^Q_{(i,j),l}$.
Attention should be paid to the fact that for ease of notation 
the constants in our proofs may differ from line to line even though their denomination is not changed.
We need the following estimate on the moments of the approximation process
$(Y_m)_{m\in\{0,\ldots,M\}}$ for the proof of Theorem~\ref{Error:DFM}.
Note that, without loss of generality, we present the proofs with an equidistant time step $h=h_m$ for all $m\in\{0,\ldots,M-1\}$.
\begin{lma}\label{Proof:Lemma-Moment-nonComm}
	Assume that (A1)--(A4) and (A5) hold. Then, it holds that
	\begin{equation*}
	\sup_{m\in\{0,\ldots,M\}}\big(\mathrm{E}\big[\|Y_m\|_{H_{\delta}}^p\big]\big)^{\frac{1}{p}}
	\leq C_{p,Q,T,\delta} \big(1+\big(\mathrm{E}\big[\|X_0\|_{H_{\delta}}^p\big]\big)^{\frac{1}{p}}\big)
	\end{equation*}
	for all $p \in [2,\infty)$ in case of (A5\ref{A5a}) and for $p=2$ in case of (A5\ref{A5c})
	for some arbitrary
	$N,K,M \in \mathbb{N}$ and some constant $C_{p,Q,T,\delta} >0$ independent of $N$, $K$ and $M$.
\end{lma}
\begin{proof}[Proof of Lemma~\ref{Proof:Lemma-Moment-nonComm}]\quad \\
	We conduct the proof of this lemma iteratively. Fix some $N,K,M \in \mathbb{N}$ and let $p\in[2,\infty)$. 
	The statement obviously holds for $m=0$. Then, for some $m\in\{1,\ldots,M\}$, we 
	assume that the statement is true for all $Y_l$ with $l \in \{0,\ldots,m-1\}$. 
	%
	%
	By the triangle inequality, we get 
	\begin{align*}
	& \big( \mathrm{E}\big[\| Y_m\|_{H_{\delta}}^p\big]\big)^{\frac{2}{p}}\\
	&\leq \Bigg( C\big(\mathrm{E}\big[\|X_0\|_{H_{\delta}}^p\big]\big)^{\frac{1}{p}}
	+ \su \bigg(\mathrm{E}\bigg[\Big\|\I e^{A(t_m-t_l)}F(Y_l)\,
	\mathrm{d}s\Big\|_{H_{\delta}}^p\bigg]\bigg)^{\frac{1}{p}} \\
	&\quad + \bigg(\mathrm{E}\bigg[\Big\| \int_{t_0}^{t_m} \su e^{A(t_m-t_l)}B(Y_l)
	\mathds{1}_{[t_l,t_{l+1})}(s)\, \mathrm{d}W^K_s\Big\|_{H_{\delta}}^p\bigg]\bigg)^{\frac{1}{p}}\\
	&\quad  +  \bigg(\mathrm{E}\bigg[\Big\|\su e^{A(t_m-t_l)}  \sum_{j\in\mathcal{J}_K}
	\Big(B\big(Y_l +\sum_{i\in\mathcal{J}_K}P_NB(Y_l)\tilde{e}_iI_{(i,j),l}^Q\big)
	\tilde{e}_j-B(Y_l)\tilde{e}_j\Big)  \Big\|_{H_{\delta}}^p\bigg]\bigg)^{\frac{1}{p}}\Bigg)^2.
	\end{align*}
	Case~1: Assume that assumption (A5\ref{A5a}) is fulfilled.
	We estimate the individual terms by a Burkholder-Davis-Gundy type 
	inequality~\cite[Theorem 4.37]{MR3236753}, 
	and a Taylor expansion of the difference approximation.
	Precisely, we use 
	\begin{align}\label{Taylor}
	B\Big(Y_l +\sum_{i\in\mathcal{J}_K}P_NB(Y_l)\tilde{e}_iI_{(i,j),l}^Q\Big) = 
	B(Y_l) + \int_0^1B'(\xi(Y_l,j,u))
	\sum_{i\in\mathcal{J}_K}P_NB(Y_l)\tilde{e}_iI_{(i,j),l}^Q\,\mathrm{d}u
	\end{align}
	for some $\xi(Y_l,j,u)= Y_l+u\sum_{i\in\mathcal{J}_K}P_NB(Y_l)
	\tilde{e}_iI_{(i,j),l}^Q \in H_{\beta}$, 
	$l\in\{0,\ldots,m-1\}$, $j\in\mathcal{J}_K$, $u\in[0,1]$.
	Therewith, we get 
	\begin{align*}
	&\big(\mathrm{E}\big[\| Y_m\|_{H_{\delta}}^p\big]\big)^{\frac{2}{p}} \\
	&\leq  C_p \Bigg( \big(\mathrm{E}\big[\|X_0\|_{H_{\delta}}^p\big]\big)^{\frac{2}{p}}
	+\bigg(\su \Big( \mathrm{E}\Big[\Big(\I\big\|(-A)^{\delta}e^{A(t_m-t_l)}F(Y_l)\big\|_{H}\,
	\mathrm{d}s\Big)^p\Big]\Big)^{\frac{1}{p}}\bigg)^2\\
	&\quad + \int_{t_0}^{t_m} \bigg(\mathrm{E}\bigg[\Big\|
	\su e^{A(t_m-t_l)}B(Y_l)\mathds{1}_{[t_l,t_{l+1})}(s)\Big\|_{L_{HS}(U_0,H_{\delta})}^p
	\bigg]\bigg)^{\frac{2}{p}} \, \mathrm{d}s \\
	&\quad  +\bigg( \mathrm{E}\bigg[\Big\| \su (-A)^{\delta}e^{A(t_m-t_l)}
	\sum_{j\in\mathcal{J}_K}\int_0^1 B'(\xi(Y_l,j,u))\Big(\sum_{i\in\mathcal{J}_K} P_N B(Y_l)
	\tilde{e}_i I_{(i,j),l}^Q, \tilde{e}_j\Big)\,\mathrm{d}u\Big\|_H^p\bigg]\bigg)^{\frac{2}{p}} \Bigg).
	\end{align*}
	The estimates on the analytic semigroup, see Lemma~6.3 and~6.13 in~\cite[Ch.2]{MR710486}, and assumptions
	(A2), (A3), yield
	\begin{align*}
	&\big(\mathrm{E}\big[\| Y_m\|_{H_{\delta}}^p\big])\big)^{\frac{2}{p}}\\
	&\leq C_p \big( \mathrm{E} \big[ \| X_0 \|_{H_{\delta}}^p \big] \big)^{\frac{2}{p}}
	+ C_{p,\delta} M \su \Big(h^p(t_m-t_l)^{-\delta p}\Big)^{\frac{2}{p}}
	\big( \mathrm{E}\big[\|F(Y_l)\|_{H}^p\big]\big)^{\frac{2}{p}} \\
	&\quad + C_p \su \int_{t_l}^{t_{l+1}}
	\bigg( \mathrm{E} \bigg[ \Big\| \sum_{k=0}^{m-1} e^{A(t_m-t_k)} B(Y_k)
	\mathds{1}_{[t_k,t_{k+1})}(s) \Big\|_{L_{HS}(U_0,H_{\delta})}^p \bigg] \bigg)^{\frac{2}{p}} \,
	\mathrm{d}s \\
	&\quad + C_{p,\delta} M \su  (t_m-t_l)^{-2\delta} \bigg(\mathrm{E}\bigg[\Big\|\sum_{j\in\mathcal{J}_K}
	\int_0^1 B'(\xi(Y_l,j,u)) \Big( \sum_{i \in \mathcal{J}_K} P_N B(Y_l)
	\tilde{e}_i I_{(i,j),l}^Q,\tilde{e}_j\Big)\,\mathrm{d}u\Big\|_{H}^p\bigg]\bigg)^{\frac{2}{p}} \\
	&\leq  C_p \big(\mathrm{E}\big[\|X_0\|_{H_{\delta}}^p\big]\big)^{\frac{2}{p}} 
	+ C_{p,T,\delta} h  \su (t_m-t_l)^{-2\delta}
	\big(1+\big(\mathrm{E}\big[\|Y_l\|_{H_{\delta}}^p\big]\big)^{\frac{2}{p}}\big) \\
	&\quad + C_p \su  \Big(\mathrm{E}\big[ \|B(Y_l)\|_{L_{HS}(U_0,H_{\delta})}^p\big]\Big)^{\frac{2}{p}}
	\int_{t_l}^{t_{l+1}}  \big\|(-A)^{-\delta}\big\|_{L(H)}^2
	\big\|(-A)^{\delta}e^{A(t_m-t_l)}\big\|_{L(H)}^2 \, \mathrm{d}s \\
	&\quad + C_{p,\delta} M \su  (t_m-t_l)^{-2\delta} \\
	&\quad \times \bigg( \sum_{j\in\mathcal{J}_K}
	\Big( \mathrm{E} \Big[ \Big( \int_0^1 \big\| B'(\xi(Y_l,j,u)) \big\|_{L(H,L(U,H))} \, \mathrm{d}u \Big)^p 
	\Big\| B(Y_l) \sum_{i \in \mathcal{J}_K} I_{(i,j),l}^Q\tilde{e}_i\Big\|_H^p\Big]\Big)^{\frac{1}{p}}\bigg)^2 \\
	&\leq C_p \big(\mathrm{E}\big[\|X_0\|_{H_{\delta}}^p\big]\big)^{\frac{2}{p}}
	+ h^{1-2\delta} C_{p,T,\delta} \su (m-l)^{-2\delta}
	\big(1+\big(\mathrm{E}\big[\|Y_l\|_{H_{\delta}}^p\big]\big)^{\frac{2}{p}}\big) \\
	&\quad + C_{p,\delta} \su h (t_m-t_l)^{-2\delta}
	\Big( \mathrm{E}\big[\| B(Y_l) \|_{L_{HS}(U_0,H_{\delta})}^p\big]\Big)^{\frac{2}{p}} \\
	&\quad + C_{p,\delta} M h^{-2\delta}  \su (m-l)^{-2\delta}\bigg( \sum_{j\in\mathcal{J}_K}
	\Big(\mathrm{E}\big[\|B(Y_l)\|^p_{L(U,H_{\delta})}\big]\Big)^{\frac{1}{p}}
	\big(\mathrm{E}\big[ \big(\sum_{i\in\mathcal{J}_K}  
	\big(I_{(i,j),l}^Q\big)^2\big)^{\frac{p}{2}}\big]\big)^{\frac{1}{p}}
	\bigg)^2.
	\end{align*}
	This expression can further be simplified by
	the distributional properties of $I^Q_{(i,j)}$, $i,j\in\mathcal{J}_K$, 
	see~\cite{MR1214374}. 
	Therewith, we obtain 
	\begin{align*}
	\big( \mathrm{E} \big[ \| Y_m \|_{H_{\delta}}^p \big] \big)^{\frac{2}{p}} 
	&\leq C_p \big(\mathrm{E}\big[\|X_0\|_{H_{\delta}}^p\big]\big)^{\frac{2}{p}}
	+ C_{p,T,\delta} h^{1-2\delta}  \su (m-l)^{-2\delta}
	\big(1+\big(\mathrm{E}\big[\|Y_l\|_{H_{\delta}}^p\big]\big)^{\frac{2}{p}}\big)\\
	&\quad + C_{p,Q,\delta} h \su (t_m-t_l)^{-2\delta}
	\big(1+\big(\mathrm{E}\big[\|Y_l\|_{H_{\delta}}^p\big]\big)^{\frac{2}{p}}\big)\\
	&\quad + C_{p,\delta} M h^{-2\delta}  \su (m-l)^{-2\delta}
	\Big( \big(1+\mathrm{E}\big[\|Y_l\|_{H_{\delta}}^p\big]\big)^{\frac{1}{p}}
	\sum_{i,j\in\mathcal{J}_K} 
	\big(\mathrm{E}\big[ |I_{(i,j),l} 
	\sqrt{\eta_i}\sqrt{\eta_j} |^p \big] \big)^{\frac{1}{p}}\Big)^2\\
	&\leq C_p \big(\mathrm{E}\big[\|X_0\|_{H_{\delta}}^p\big]\big)^{\frac{2}{p}}
	+ C_{p,Q,T,\delta} h^{1-2\delta}  \su (m-l)^{-2\delta}
	\big(1+\big(\mathrm{E}\big[\|Y_l\|_{H_{\delta}}^p\big]\big)^{\frac{2}{p}}\big)\\
	&\quad + C_{p,\delta} M h^{-2\delta}  \su (m-l)^{-2\delta}
	\Big(  \big(1+\mathrm{E}\big[\|Y_l\|_{H_{\delta}}^p\big]\big)^{\frac{1}{p}}
	\sum_{i,j\in\mathcal{J}_K}\sqrt{\eta_i}\sqrt{\eta_j}\,h\Big)^2.
	\end{align*}
	Case~2: Assume $p=2$ and that assumption (A5\ref{A5c}) is fulfilled.
	Again, we estimate the individual terms by a Burkholder-Davis-Gundy type 
	inequality~\cite[Theorem 4.37]{MR3236753}, 
	but a first order Taylor expansion of the difference approximation.
	Thus, we use 
	\begin{equation} \label{Proof-Lem-Case2-Taylor2}
	\begin{split}
	&B\Big(Y_l +\sum_{i\in\mathcal{J}_K}P_NB(Y_l)\tilde{e}_iI_{(i,j),l}^Q\Big) = 
	B(Y_l) + B'(Y_l) \sum_{i\in\mathcal{J}_K} P_N B(Y_l) \tilde{e}_i I_{(i,j),l}^Q \\
	&\quad \quad + \int_0^1 \int_0^u B''(\xi(Y_l,j,r))
	\Big( \sum_{i \in \mathcal{J}_K} P_N B(Y_l) \tilde{e}_i I_{(i,j),l}^Q ,
	\sum_{i \in \mathcal{J}_K} P_N B(Y_l) \tilde{e}_i I_{(i,j),l}^Q \Big) \, \mathrm{d}r \, \mathrm{d}u
	\end{split}
	\end{equation}
	for some $\xi(Y_l,j,r)= Y_l + r \sum_{i \in \mathcal{J}_K} P_N B(Y_l)
	\tilde{e}_i I_{(i,j),l}^Q \in H_{\beta}$, 
	$l\in\{0,\ldots,m-1\}$, $j\in\mathcal{J}_K$, $r\in[0,1]$.
	With estimates on the analytic semigroup, see Lemma~6.3 and 6.13 
	in~\cite[Ch.2]{MR710486}, we get that
	\begin{align*}
	&\big(\mathrm{E}\big[\| Y_m\|_{H_{\delta}}^p\big]\big)^{\frac{2}{p}} \\
	&\leq  C_p \Bigg( \big(\mathrm{E}\big[\|X_0\|_{H_{\delta}}^p\big]\big)^{\frac{2}{p}}
	+\bigg(\su \Big( \mathrm{E}\Big[\Big(\I\big\|(-A)^{\delta}e^{A(t_m-t_l)}F(Y_l)\big\|_{H}\,
	\mathrm{d}s\Big)^p\Big]\Big)^{\frac{1}{p}}\bigg)^2\\
	&\quad + \int_{t_0}^{t_m} \bigg(\mathrm{E}\bigg[\Big\|
	\su e^{A(t_m-t_l)}B(Y_l)\mathds{1}_{[t_l,t_{l+1})}(s)\Big\|_{L_{HS}(U_0,H_{\delta})}^p
	\bigg]\bigg)^{\frac{2}{p}} \, \mathrm{d}s \\
	&\quad + \bigg( \mathrm{E}\bigg[\Big\| \su (-A)^{\delta}e^{A(t_m-t_l)}
	\sum_{j \in \mathcal{J}_K} \Big( B'(Y_l) 
	\Big( \sum_{i\in\mathcal{J}_K} P_N B(Y_l) \tilde{e}_i I_{(i,j),l}^Q, \tilde{e}_j \Big) \\
	&\quad + \int_0^1 \int_0^u B''(\xi(Y_l,j,r))
	\Big( \sum_{i \in \mathcal{J}_K} P_N B(Y_l) \tilde{e}_i I_{(i,j),l}^Q ,
	\sum_{i \in \mathcal{J}_K} P_N B(Y_l) \tilde{e}_i I_{(i,j),l}^Q, \tilde{e}_j \Big) 
	\, \mathrm{d}r \, \mathrm{d}u \Big)
	\Big\|_H^p\bigg]\bigg)^{\frac{2}{p}} \Bigg) \\
	&\leq C_p \big(\mathrm{E}\big[\|X_0\|_{H_{\delta}}^p\big]\big)^{\frac{2}{p}}
	+ C_{p,\delta} M \su \Big(h^p(t_m-t_l)^{-\delta p}\Big)^{\frac{2}{p}}
	\big( \mathrm{E}\big[\|F(Y_l)\|_{H}^p\big]\big)^{\frac{2}{p}} \\
	&\quad + C_p \su \int_{t_l}^{t_{l+1}}
	\bigg(\mathrm{E}\bigg[\Big\|\sum_{k=0}^{m-1} e^{A(t_m-t_k)} B(Y_k)
	\mathds{1}_{[t_k,t_{k+1})}(s)\Big\|_{L_{HS}(U_0,H_{\delta})}^p\bigg]\bigg)^{\frac{2}{p}} \,
	\mathrm{d}s \\
	&\quad + C_{p,\delta} M \su  (t_m-t_l)^{-2\delta} \bigg( \mathrm{E}\bigg[\Big\|
	\sum_{i, j \in \mathcal{J}_K} I_{(i,j),l}^Q B'(Y_l) \Big( P_N B(Y_l) \tilde{e}_i, \tilde{e}_j \Big)
	\Big\|_H^p\bigg]\bigg)^{\frac{2}{p}} \\
	&\quad + C_{p,\delta} M \su  (t_m-t_l)^{-2\delta} \bigg( \mathrm{E}\bigg[\Big\| 
	\sum_{j \in \mathcal{J}_K} 
	\int_0^1 \int_0^u B''(\xi(Y_l,j,r)) \\
	&\quad \times \Big( P_N B(Y_l) \sum_{i \in \mathcal{J}_K} \tilde{e}_i I_{(i,j),l}^Q,
	P_N B(Y_l) \sum_{i \in \mathcal{J}_K} \tilde{e}_i I_{(i,j),l}^Q, \tilde{e}_j \Big) 
	\, \mathrm{d}r \, \mathrm{d}u \Big)
	\Big\|_H^p\bigg]\bigg)^{\frac{2}{p}} .
	\end{align*}
	Making use of $p=2$, assumptions (A2), (A3) and (A5\ref{A5c})
	yield
	\begin{align*}
	&\mathrm{E}\big[\| Y_m\|_{H_{\delta}}^2\big] \\
	&\leq C \mathrm{E} \big[ \|X_0\|_{H_{\delta}}^2 \big]
	+ C_{T,\delta} 
	\su h (t_m-t_l)^{-2\delta}
	\big( 1 + \mathrm{E} \big[ \|Y_l\|_{H_{\delta}}^2 \big] \big) \\
	&\quad + C \su \mathrm{E} \big[ \|B(Y_l)\|_{L_{HS}(U_0,H_{\delta})}^2 \big]
	\int_{t_l}^{t_{l+1}}  \big\|(-A)^{-\delta}\big\|_{L(H)}^2
	\big\|(-A)^{\delta}e^{A(t_m-t_l)}\big\|_{L(H)}^2 \, \mathrm{d}s \\
	&\quad + C_{\delta} \, M \su (t_m-t_l)^{-2\delta} \sum_{i_1,i_2,j_1,j_2 \in \mathcal{J}_K} \mathrm{E}\bigg[
	I_{(i_1,j_1),l}^Q I_{(i_2,j_2),l}^Q \\
	&\quad \times \big\langle B'(Y_l) \big( P_N B(Y_l) \tilde{e}_{i_1}, \tilde{e}_{j_1} \big),
	B'(Y_l) \big( P_N B(Y_l) \tilde{e}_{i_2}, \tilde{e}_{j_2} \big) \big\rangle_H
	\bigg] \\
	&\quad + C_{\delta} \, M \su (t_m-t_l)^{-2\delta} \bigg( \sum_{j \in \mathcal{J}_K} \bigg( 
	\mathrm{E}\bigg[ \Big(
	\int_0^1 \int_0^u \big\| B''(\xi(Y_l,j,r)) \big( P_N B(Y_l), P_N B(Y_l) \big) \big\|_{L^{(2)}(U,L(U,H))} \\
	&\quad \times \Big\| \sum_{i \in \mathcal{J}_K}  \tilde{e}_i I_{(i,j),l}^Q \Big\|_U^2 
	\| \tilde{e}_j \|_U
	\, \mathrm{d}r \, \mathrm{d}u \Big)^2
	\bigg] \bigg)^{\frac{1}{2}} \bigg)^{2} \\
	&\leq C \mathrm{E} \big[ \|X_0\|_{H_{\delta}}^2 \big]
	+ C_{T,\delta} \su h (t_m-t_l)^{-2\delta}
	\big( 1 + \mathrm{E} \big[ \|Y_l\|_{H_{\delta}}^2 \big] \big) \\
	&\quad + C \su \mathrm{E} \big[ \|B(Y_l)\|_{L_{HS}(U_0,H_{\delta})}^2 \big]
	\big\|(-A)^{-\delta}\big\|_{L(H)}^2 \big\|(-A)^{\delta}e^{A(t_m-t_l)}\big\|_{L(H)}^2 \, h \\
	&\quad + C_{\delta} \, M \su (t_m-t_l)^{-2\delta} \sum_{i_1,i_2,j_1,j_2 \in \mathcal{J}_K}
	\mathrm{E}\big[ I_{(i_1,j_1),l}^Q I_{(i_2,j_2),l}^Q \big] \\
	&\quad \times \mathrm{E}\big[ \big\langle B'(Y_l) \big( P_N B(Y_l) \tilde{e}_{i_1}, \tilde{e}_{j_1} \big),
	B'(Y_l) \big( P_N B(Y_l) \tilde{e}_{i_2}, \tilde{e}_{j_2} \big) \big\rangle_H
	\big] \\
	&\quad + C_{\delta} \, M \su (t_m-t_l)^{-2\delta} \bigg( \sum_{j \in \mathcal{J}_K} \bigg( 
	\mathrm{E}\bigg[ \Big(
	\int_0^1 \int_0^u \big(1 + \big\| \xi(Y_l,j,r) \big\|_H + \big\| Y_l \big\|_H \big) \\
	&\quad \times \sum_{i \in \mathcal{J}_K} \big( I_{(i,j),l}^Q \big)^2
	\, \mathrm{d}r \, \mathrm{d}u \Big)^2
	\bigg] \bigg)^{\frac{1}{2}} \bigg)^{2} .
	\end{align*}
	Due to $\mathrm{E} \big[ I_{(i_1,j_1),l}^Q I_{(i_2,j_2),l}^Q \big] = \tfrac{1}{2} \eta_{i1}\eta_{i2}h_l^2$
	if $i_1 = i_2$ and $j_1=j_2$ and $\mathrm{E} \big[ I_{(i_1,j_1),l}^Q I_{(i_2,j_2),l}^Q \big] = 0$ 
	otherwise, we get
	\allowdisplaybreaks{
		\begin{align*}
		&\mathrm{E}\big[\| Y_m\|_{H_{\delta}}^2\big]) \\
		&\leq C \mathrm{E} \big[ \|X_0\|_{H_{\delta}}^2 \big]
		+ C_{T,\delta} \su h (t_m-t_l)^{-2\delta}
		\big( 1 + \mathrm{E} \big[ \|Y_l\|_{H_{\delta}}^2 \big] \big) \\
		&\quad + C \su \mathrm{E} \big[ \|B(Y_l)\|_{L_{HS}(U_0,H_{\delta})}^2 \big]
		\big\|(-A)^{-\delta}\big\|_{L(H)}^2 \big\|(-A)^{\delta}e^{A(t_m-t_l)}\big\|_{L(H)}^2 \, h \\
		&\quad + C_{\delta} \, M \su (t_m-t_l)^{-2\delta} \sum_{i,j \in \mathcal{J}_K} \eta_i \eta_j h^2 \,
		\mathrm{E}\big[ \| B'(Y_l) \big( P_N B(Y_l) \tilde{e}_i, \tilde{e}_j \big)\|_H^2 \big] \\
		&\quad + C_{\delta} \, M \su (t_m-t_l)^{-2\delta} \bigg( \sum_{j \in \mathcal{J}_K} \bigg( 
		\mathrm{E}\bigg[ \Big(
		\int_0^1 \int_0^u \Big(1 + 2 \big\| Y_l \big\|_H + r \Big\| P_N B(Y_l) 
		\sum_{i \in \mathcal{J}_K} \tilde{e}_i I_{(i,j),l}^Q \Big\|_H \Big) \\
		&\quad \times \sum_{i \in \mathcal{J}_K} \big( I_{(i,j),l}^Q \big)^2
		\, \mathrm{d}r \, \mathrm{d}u \Big)^2
		\bigg] \bigg)^{\frac{1}{2}} \bigg)^{2} \\
		&\leq C \mathrm{E} \big[ \|X_0\|_{H_{\delta}}^2 \big]
		+ C_{T,\delta} \su h (t_m-t_l)^{-2\delta}
		\big( 1 + \mathrm{E} \big[ \|Y_l\|_{H_{\delta}}^2 \big] \big) \\
		&\quad + C \su \mathrm{E} \big[ \tr Q \, \|B(Y_l)\|_{L(U,H_{\delta})}^2 \big]
		\big\|(-A)^{-\delta}\big\|_{L(H)}^2 \big\|(-A)^{\delta}e^{A(t_m-t_l)}\big\|_{L(H)}^2 \, h \\
		&\quad + C_{Q,\delta} \, M \su (t_m-t_l)^{-2\delta} h^2 \, \mathrm{E}\big[ \big\| B'(Y_l) P_N B(Y_l)
		\big\|_{L(U,L(U,H))}^2 \big] \\
		&\quad + C_{\delta} \, M \su (t_m-t_l)^{-2\delta} \bigg( \sum_{j \in \mathcal{J}_K} \bigg( 
		\mathrm{E}\bigg[
		\Big( 1 + \big\| Y_l \big\|_{H_{\delta}} + \big\|(-A)^{-\delta} \big\|_{L(H)} \big\| B(Y_l) \big\|_{L(U,H_{\delta})} \\
		&\quad \times \Big( \sum_{i \in \mathcal{J}_K} \big( I_{(i,j),l}^Q\big)^2 \Big)^{\frac{1}{2}} \Big)^2
		\Big( \sum_{i \in \mathcal{J}_K} \big( I_{(i,j),l}^Q \big)^2 \Big)^2
		\bigg] \bigg)^{\frac{1}{2}} \bigg)^{2} \\
		&\leq C \mathrm{E} \big[ \|X_0\|_{H_{\delta}}^2 \big]
		+ C_{T,\delta} \su h (t_m-t_l)^{-2\delta}
		\big( 1 + \mathrm{E} \big[ \|Y_l\|_{H_{\delta}}^2 \big] \big)
		+ C_{Q,\delta} h \su (t_m-t_l)^{-2\delta} \big(1 + \mathrm{E} \big[ \|Y_l\|_{H_{\delta}}^2 \big] \big) \\
		&\quad + C_{Q,T,\delta} h \su (t_m-t_l)^{-2\delta} \mathrm{E}\big[ \big\| B'(Y_l) \big\|_{L(H,L(U,H))}^2
		\big\| (-A)^{-\delta} \big\|_{L(H)}^2 \big\| B(Y_l) \|_{L(U,H_{\delta})}^2 \big] \\
		&\quad + C_{\delta} \, M \su (t_m-t_l)^{-2\delta} \bigg( \sum_{j \in \mathcal{J}_K} \bigg( 
		\mathrm{E}\bigg[ \Big( \Big( \sum_{i \in \mathcal{J}_K} \big( I_{(i,j),l}^Q \big)^2 \Big)^2
		+ \Big( \sum_{i \in \mathcal{J}_K} \big( I_{(i,j),l}^Q \big)^2 \Big)^3 \Big) \\
		&\quad \times \big( 1 + \big\| Y_l \big\|_{H_{\delta}}^2 \big) \bigg] \bigg)^{\frac{1}{2}} \bigg)^{2} \\
		&\leq C \mathrm{E} \big[ \|X_0\|_{H_{\delta}}^2 \big]
		+ C_{Q,T,\delta} h^{1-2\delta} \su (m-l)^{-2\delta}
		\big( 1 + \mathrm{E} \big[ \|Y_l\|_{H_{\delta}}^2 \big] \big) \\
		&\quad + C_{\delta} \, M \su (t_m-t_l)^{-2\delta} \bigg( \sum_{j \in \mathcal{J}_K} \bigg( 
		\bigg( \mathrm{E}\bigg[ \Big( \sum_{i \in \mathcal{J}_K} \big( I_{(i,j)}^Q(h_l) \big)^2 \Big)^2 \bigg]
		+ \mathrm{E}\bigg[  \Big( \sum_{i \in \mathcal{J}_K} \big( I_{(i,j)}^Q(h_l) \big)^2 \Big)^3  \bigg] \bigg) \\
		&\quad \times \big( 1 + \mathrm{E}\big[  \big\| Y_l \big\|_{H_{\delta}}^2 \big] \big) \bigg)^{\frac{1}{2}} \bigg)^{2} \\
		&\leq C \mathrm{E} \big[ \|X_0\|_{H_{\delta}}^2 \big]
		+ C_{Q,T,\delta} h^{1-2\delta} \su (m-l)^{-2\delta}
		\big( 1 + \mathrm{E} \big[ \|Y_l\|_{H_{\delta}}^2 \big] \big) \\
		&\quad + C_{Q,\delta} \, M \, h^{-2\delta} \su (m-l)^{-2\delta} \bigg( \tr Q \bigg( (\tr Q)^2 h^4 + (\tr Q)^3 h^6 
		\bigg)^{\frac{1}{2}} \bigg)^{2} \big( 1 + \mathrm{E}\big[  \big\| Y_l \big\|_{H_{\delta}}^2 \big] \big) \\
		&\leq C \mathrm{E} \big[ \|X_0\|_{H_{\delta}}^2 \big]
		+ C_{Q,T,\delta} h^{1-2\delta} \su (m-l)^{-2\delta}
		\big( 1 + \mathrm{E} \big[ \|Y_l\|_{H_{\delta}}^2 \big] \big) \\
		&\quad + C_{Q,T,\delta} \, h^{1-2\delta} \su (m-l)^{-2\delta} \big( h^2 + h^4 \big) 
		\big( 1 + \mathrm{E}\big[  \big\| Y_l \big\|_{H_{\delta}}^2 \big] \big).
		\end{align*}
	}%
	%
	%
	Now, we continue with the final estimates in case~1 and case~2 simultaneously having in mind 
	that case~2 is restricted to $p=2$.
	Interpreting the terms $\su (m-l)^{-2\delta}$ as lower Darboux sums, 
	we estimate these expressions as in the proof of the scheme for SPDEs with commutative noise
	in~\cite{MR3842926},
	see also~\cite{MR2471778},
	for  $\delta \in (0,\frac{1}{2})$ and all $m\in\{1,\ldots,M\}$, 
	$M\in\mathbb{N}$
	\begin{equation*}
	\su (m-l)^{-2\delta}  = 
	\sum_{l=1}^m l^{-2\delta} \leq 1+\int_1^M r^{-2\delta} \, \mathrm{d}r \leq \frac{M^{1-2\delta}}{1-2\delta}.
	\end{equation*}
	This yields
	\begin{equation*}
	\big(\mathrm{E}\big[\| Y_m\|_{H_{\delta}}^p\big]\big)^{\frac{2}{p}}
	\leq C_p \big( \mathrm{E} \big[ \| X_0 \|_{H_{\delta}}^p \big] \big)^{\frac{2}{p}} 
	+ C_{p,Q,T,\delta}
	+ h^{1-2\delta} C_{p,Q,T,\delta} \su (m-l)^{-2\delta}
	\big(\mathrm{E}\big[\|Y_l\|_{H_{\delta}}^p\big]\big)^{\frac{2}{p}}
	\end{equation*}
	in a first step. 
	Further, the discrete Gronwall Lemma implies
	the boundedness of the moments
	\begin{align*}
	\big(\mathrm{E}\big[\| Y_m\|_{H_{\delta}}^p\big]\big)^{\frac{2}{p}} 
	&\leq \Big(C_p \big( \mathrm{E} \big[ \| X_0 \|_{H_{\delta}}^p \big] \big)^{\frac{2}{p}}
	+ C_{p,Q,T,\delta} \Big) e^{C_{p,Q,T,\delta}  \su (m-l)^{-2\delta}h^{1-2\delta}} \\
	&\leq C_{p,Q,T,\delta} \big(1+\big(\mathrm{E}\big[\|X_0\|_{H_{\delta}}^p\big]\big)^{\frac{2}{p}}\big)
	\end{align*}
	for all $m\in\{1,\ldots,M\}$, $M\in\mathbb{N}$, for $p\in[2,\infty)$ in case of (A5\ref{A5a})
	and for $p=2$ in case of (A5\ref{A5c}).
\end{proof}
We address the proof of Theorem~\ref{Error:DFM} now and show that the 
scheme converges with the specified order. This estimate does not yet
involve any approximation of the stochastic iterated integrals.
\begin{proof}[Proof of Theorem~\ref{Error:DFM}]
	First, we express the mild solution of~\eqref{SPDE} as 
	\begin{equation*}
	X_{t_m} = e^{At_m}X_0 + \sI e^{A(t_m-s)}F(X_s)\,\mathrm{d}s + \sI e^{A(t_m-s)}B(X_s)\,\mathrm{d}W_s
	\end{equation*}
	for all $m\in\{0,\ldots,M\}$, $M\in\mathbb{N}$ to align the components with the corresponding terms in the 
	approximation below.
	We define
	the following auxiliary processes for $m\in\{0,\ldots,M\}$, $M,N,K\in\mathbb{N}$
	\begin{align*}
	Y_{m}^{\MIL} &= P_N \bigg( e^{At_m}X_0 + \sI e^{A(t_m-t_l)}F(Y_l^{\MIL})\,\mathrm{d}s
	+ \sI e^{A(t_m-t_l)}B(Y_l^{\MIL})\,\mathrm{d}W^K_s  \\
	&\quad + \sI e^{A(t_m-t_l)} B'(Y_l^{\MIL})\Big(P_N\int_{t_l}^s B(Y_l^{\MIL})\,\mathrm{d}W_r^K\Big)\,\mathrm{d}W_s^K\bigg),
	\\
	\bar{Y}^{\MIL}_{m} &= P_N\bigg(e^{At_m} X_0 + \sI e^{A(t_m-t_l)}F(Y_l)\,\mathrm{d}s
	+ \sI e^{A(t_m-t_l)}B(Y_l)\,\mathrm{d}W^K_s  \\
	&\quad +\sI e^{A(t_m-t_l)} B'(Y_l)\Big(P_N\int_{t_l}^sB(Y_l)\,\mathrm{d}W_r^K\Big)\,\mathrm{d}W_s^K\bigg).
	\end{align*}
	The discrete process $(Y_m)_{m\in\{0,\ldots,M\}}$ denotes
	the approximation obtained 
	by the $\DFM$ scheme in~\eqref{DFM}.
	The auxiliary processes are introduced 
	in order to split the approximation error
	such that we can employ some known prior estimates. We
	analyze the following terms separately
	\begin{align}\label{ErrorSplit}
	\Big(\mathrm{E}\big[\|X_{t_m}-\Y_m\|_H^2\big]\Big)^{\frac{1}{2}} 
	&\leq \Big(\mathrm{E}\big[\|X_{t_m}-Y_m^{\MIL}\|_H^2\big]\Big)^{\frac{1}{2}}
	+\Big(\mathrm{E}\big[\|Y_m^{\MIL}-\Y_m\|_H^2\big]\Big)^{\frac{1}{2}}
	\nonumber \\  
	&\leq \Big(\mathrm{E}\big[\|X_{t_m}-Y_m^{\MIL}\|_H^2\big]\Big)^{\frac{1}{2}}
	+\Big(\mathrm{E}\big[\|Y_m^{\MIL}-\bar{Y}^{\MIL}_{m}\|_H^2\big]\Big)^{\frac{1}{2}} \nonumber\\ 
	&\quad +\Big(\mathrm{E}\big[\|\bar{Y}^{\MIL}_{m}-\Y_m\|_H^2\big]\Big)^{\frac{1}{2}}
	\end{align}
	for all $m\in\{0,\ldots,M\}$, $M\in\mathbb{N}$. The first term is similar to the error that results from 
	the approximation of~\eqref{SPDE} with the Milstein scheme by Jentzen and R\"ockner 
	presented in~\cite{MR3320928}. A slight difference arises as we introduce the projection operator
	$P_N$ in the definition of $Y_{m}^{\MIL}$, see 
	the computations 
	in~\cite{2015arXiv150908427L,MR3842926}.
	The main reasoning, however, is the same.
	In the error analysis in~\cite{MR3320928}, 
	the commutativity condition is not needed - it is only 
	employed to facilitate implementation - 
	whereas all conditions required in
	the proof in~\cite{MR3320928} are fulfilled due to 
	assumptions (A1)--(A4). 
	Therefore, the estimate 
	\begin{equation} \label{Proof-Main-Thm-Estimate-Mil-Orig}
	\sup_{m\in\{0,\ldots,M\}} \Big(\mathrm{E}\big[\|X_{t_m}-Y_m^{\MIL}\|_H^2\big]\Big)^{\frac{1}{2}} \leq
	C_{Q,T} \Big( \Big( \inf_{i \in \mathcal{I} \setminus
		\mathcal{I}_N} \lambda_i \Big)^{-\gamma}
	+ \Big( \sup_{j \in \mathcal{J} \setminus \mathcal{J}_K}
	\eta_j \Big)^{\alpha} + M^{-\min(2(\gamma-\beta),\gamma)}\Big)
	\end{equation}
	for arbitrary $N,M,K\in\mathbb{N}$ is valid. 
	For details, we refer to~\cite{MR3320928}. \\ \\
	The error estimate of the second term in~\eqref{ErrorSplit}, $
	\mathrm{E}\big[\|Y_m^{\MIL}-\bar{Y}^{\MIL}_m\|_H^2\big]$, $m\in\{0,\ldots,M\}$, $M\in\mathbb{N}$,
	can be obtained by the same means as in the proof 
	of convergence of the Milstein scheme 
	in~\cite{MR3320928} and mainly relies on the Lipschitz properties 
	of the operators.
	We transfer this reasoning from \cite[Section 6.3]{MR3320928}, which yields
	\begin{equation}\label{Milbar}
	\mathrm{E}\big[\|Y_m^{\MIL}-\bar{Y}^{\MIL}_m\|_H^2\big] \leq C_Th\sum_{l=0}^{m-1}
	\mathrm{E}\big[\|Y_l^{\MIL}-Y_l\|_H^2\big]
	\end{equation}
	for all $m\in\{0,\ldots,M\}$, $M,N,K\in\mathbb{N}$.\\ \\
	Next, we analyze the third term in~\eqref{ErrorSplit} 
	which represents the error that results from the approximation
	of the derivative.
	We can show that the theoretical order of convergence
	that the Milstein scheme obtains
	is not reduced by this approximation.
	We rewrite the expression in~\eqref{ErrorSplit} as
	\begin{align*}
	&\mathrm{E}\big[\|\bar{Y}^{\MIL}_{m}-\Y_m\|_H^2\big] \\
	&= \mathrm{E}\bigg[\Big\|P_N\Big(\su \sum_{i,j \in\mathcal{J}_K}
	\sqrt{\eta_j}\sqrt{\eta_i}e^{A(t_m-t_l)}B'(Y_l)\big(P_NB(Y_l)\tilde{e}_i,\tilde{e}_j\big)
	I_{(i,j),l}\Big) \\
	&\quad - P_N\Big(\su e^{A(t_m-t_l)}\sum_{j\in\mathcal{J}_K}
	\Big(B\Big(\Y_l+\sum_{i\in\mathcal{J}_K}\sqrt{\eta_j}\sqrt{\eta_i}P_NB(\Y_l)
	\tilde{e}_iI_{(i,j),l}\Big)\tilde{e}_j-B(\Y_l)\tilde{e}_j\Big)\Big)\Big\|_H^2\bigg].
	\end{align*}
	We employ a Taylor approximation of first 
	order for the second term, see~\eqref{Proof-Lem-Case2-Taylor2}, 
	such that the first order derivatives cancel. Moreover,
	the triangle inequality and assumption (A3) imply
	\begin{equation} \label{Proof-Main-Thm-Last-Eqn-before-cases}
	\begin{split}
	&\mathrm{E}\big[\|\bar{Y}^{\MIL}_{m}-\Y_m\|_H^2\big] \\
	&\leq
	\mathrm{E}\bigg[\Big\|\su e^{A(t_{m}-t_l)}\sum_{j\in\mathcal{J}_K}
	\int_0^1\int_0^uB''\Big(\Y_l+r\sum_{i\in\mathcal{J}_K}\sqrt{\eta_j}
	\sqrt{\eta_i}P_NB(\Y_l)\tilde{e}_iI_{(i,j),l}\Big)  \\
	&\quad \Big(\sum_{i\in\mathcal{J}_K}\sqrt{\eta_j}\sqrt{\eta_i}P_NB(\Y_l)\tilde{e}_i
	I_{(i,j),l},\sum_{i\in\mathcal{J}_K}\sqrt{\eta_j}
	\sqrt{\eta_i}P_NB(\Y_l)\tilde{e}_iI_{(i,j),l}\Big)\tilde{e}_j \, \mathrm{d}r\,\mathrm{d}u
	\Big\|_H^2\bigg] .
	\end{split}
	\end{equation}
	%
	%
	Case~1: Assume that assumption (A5\ref{A5a}) is fulfilled, i.e., 
	Lemma~\ref{Proof:Lemma-Moment-nonComm} is valid for any $p \geq 2$.
	Thus, it follows from \eqref{Proof-Main-Thm-Last-Eqn-before-cases} that
	\begin{align*}
	&\mathrm{E}\big[\|\bar{Y}^{\MIL}_{m}-\Y_m\|_H^2\big] \\
	&\leq \mathrm{E}\bigg[\Big(\su \sum_{j\in\mathcal{J}_K}\int_0^1\int_0^u\Big\|e^{A(t_{m}-t_l)}
	B''\Big(\Y_l+r\sum_{i\in\mathcal{J}_K}\sqrt{\eta_j}\sqrt{\eta_i}P_NB(\Y_l)\tilde{e}_iI_{(i,j),l}\Big) \\ 
	&\quad \Big(\sum_{i\in\mathcal{J}_K}\sqrt{\eta_j}\sqrt{\eta_i}P_NB(\Y_l)\tilde{e}_iI_{(i,j),l},
	\sum_{i\in\mathcal{J}_K}\sqrt{\eta_j}\sqrt{\eta_i}P_N
	B(\Y_l)\tilde{e}_iI_{(i,j),l}\Big)\tilde{e}_j\Big\|_H  \,\mathrm{d}r\,\mathrm{d}u\Big)^2\bigg] \\
	&\leq C_T \mathrm{E}\bigg[\Big(\su \sum_{j\in\mathcal{J}_K}\int_0^1\int_0^u\Big\|
	B''\Big(\Y_l+r\sum_{i\in\mathcal{J}_K}\sqrt{\eta_j}\sqrt{\eta_i}P_NB(\Y_l)
	\tilde{e}_iI_{(i,j),l}\Big)\Big\|_{L^{(2)}(H,L(U,H))}\\
	&\quad \times \Big\|\sum_{i\in\mathcal{J}_K}
	\sqrt{\eta_j}\sqrt{\eta_i}P_NB(\Y_l)\tilde{e}_iI_{(i,j),l}\Big\|_H^2\,\mathrm{d}r
	\,\mathrm{d}u\Big)^2\bigg]\\
	&\leq C_T \mathrm{E}\bigg[\Big(\su \sum_{j\in\mathcal{J}_K}\Big\|B(\Y_l)
	\sum_{i\in\mathcal{J}_K}\sqrt{\eta_j}\sqrt{\eta_i}\tilde{e}_iI_{(i,j),l}\Big\|_H^2\Big)^2\bigg].
	\end{align*}
	Then, we obtain with Lemma~\ref{Proof:Lemma-Moment-nonComm} in the case
	that condition (A5\ref{A5a}) is valid that
	\begin{align*}
	\mathrm{E}\big[\|\bar{Y}^{\MIL}_{m}-\Y_m\|_H^2\big]
	&\leq C_T \mathrm{E}\bigg[\Big(\su \sum_{j\in\mathcal{J}_K}\|B(\Y_l)\|_{L(U,H)}^2
	\Big\|\sum_{i\in\mathcal{J}_K}\sqrt{\eta_j}\sqrt{\eta_i}
	\tilde{e}_iI_{(i,j),l}\Big\|_U^2\Big)^2\bigg] \\
	&\leq C_T \mathrm{E}\bigg[\Big(\su \sum_{j\in\mathcal{J}_K}\|B(\Y_l)\|_{L(U,H_{\delta})}^2 \\
	&\quad \times \sum_{i_1,i_2\in\mathcal{J}_K}
	\eta_j\sqrt{\eta_{i_1}}\sqrt{\eta_{i_2}}I_{(i_1,j),l}I_{(i_2,j),l}
	\langle\tilde{e}_{i_1},\tilde{e}_{i_2}\rangle_U\Big)^2\bigg]\\
	&= C_T \mathrm{E}\bigg[\Big(\su \sum_{j\in\mathcal{J}_K} \|B(\Y_l)\|_{L(U,H_{\delta})}^2
	\sum_{i\in\mathcal{J}_K}\eta_j\eta_iI_{(i,j),l}^2\Big)^2\bigg]\\
	&\leq C_T \bigg(\su\sum_{j\in\mathcal{J}_K}  \sum_{i\in\mathcal{J}_K}
	\big( \mathrm{E}[\|B(\Y_l)\|_{L(U,H_{\delta})}^4]\big)^{\frac{1}{2}}
	\Big(\mathrm{E}\Big[\Big(\eta_j\eta_iI_{(i,j),l}^2\Big)^2\Big]\Big)^{\frac{1}{2}}\bigg)^2\\
	&\leq C_{Q,T,\delta} \Big(\su\sum_{j\in\mathcal{J}_K}  \sum_{i\in\mathcal{J}_K}\eta_j\eta_i
	\big(\mathrm{E}\big[I_{(i,j),l}^4\big]\big)^{\frac{1}{2}}\Big)^2.
	\end{align*}
	Finally, we get
	\begin{equation}\label{BarY}
	\mathrm{E}\big[\|\bar{Y}^{\MIL}_{m}-\Y_m\|_H^2\big]
	\leq C_{Q,T,\delta} \Big(\su(\tr Q)^2 h^2\Big)^2
	\leq C_{Q,T,\delta} h^2 (\tr Q)^4 
	\leq C_{Q,T,\delta} h^2
	\end{equation}
	by the distributional properties of $I_{(i,j),l}$, $l\in\{0,\ldots,m-1\}$, 
	$i,j\in\mathcal{J}_K$ for all $m\in\{1,\ldots,M\}$, $M,K\in\mathbb{N}$,
	see \cite{MR1214374}. \\ \\
	%
	%
	%
	%
	Case~2: If assumption (A5\ref{A5c}) is fulfilled, then Lemma~\ref{Proof:Lemma-Moment-nonComm} 
	is valid for $p = 2$. Therefore, we need a customized proof to proceed and
	we get for $p=2$ from \eqref{Proof-Lem-Case2-Taylor2} and 
	\eqref{Proof-Main-Thm-Last-Eqn-before-cases} that
	\begin{align*}
	&\mathrm{E}\big[\|\bar{Y}^{\MIL}_{m}-\Y_m\|_H^2\big] \\
	&\leq
	\mathrm{E} \bigg[ \Big\| \su e^{A(t_{m}-t_l)} \sum_{j \in \mathcal{J}_K}
	\int_0^1 \int_0^u B''\Big( \Y_l + r \sum_{i \in \mathcal{J}_K}
	P_N B(\Y_l) \tilde{e}_i I^Q_{(i,j),l} \Big) \\
	&\quad \Big( \sum_{i \in \mathcal{J}_K} P_N B(\Y_l) \tilde{e}_i
	I^Q_{(i,j),l}, \sum_{i \in \mathcal{J}_K} P_N B(\Y_l) \tilde{e}_i
	I^Q_{(i,j),l} \Big) \tilde{e}_j \, \mathrm{d}r\,\mathrm{d}u
	\Big\|_H^2\bigg] \\
	&\leq
	M \su \bigg( \sum_{j \in \mathcal{J}_K} \bigg( \mathrm{E} \bigg[ 
	\Big\| e^{A(t_{m}-t_l)} \Big\|_{L(H)}^2
	\Big\| \int_0^1 \int_0^u B''\Big( \Y_l + r \sum_{i \in \mathcal{J}_K}
	P_N B(\Y_l) \tilde{e}_i I^Q_{(i,j),l} \Big) \\
	&\quad \Big( \sum_{i \in \mathcal{J}_K} P_N B(\Y_l) \tilde{e}_i
	I^Q_{(i,j),l}, \sum_{i \in \mathcal{J}_K} P_N B(\Y_l) \tilde{e}_i
	I^Q_{(i,j),l} \Big) \tilde{e}_j \, \mathrm{d}r\,\mathrm{d}u
	\Big\|_H^2\bigg] \bigg)^{\frac{1}{2}} \bigg)^2 \\
	&\leq C_T \, M \su \bigg( \sum_{j \in \mathcal{J}_K} \bigg( 
	\mathrm{E}\bigg[ \Big(
	\int_0^1 \int_0^u \big\| B''(\xi(Y_l,j,r)) \big( P_N B(Y_l), P_N B(Y_l) \big) \big\|_{L^{(2)}(U,L(U,H))} \\
	&\quad \times \Big\| \sum_{i \in \mathcal{J}_K}  \tilde{e}_i I_{(i,j),l}^Q \Big\|_U^2 
	\| \tilde{e}_j \|_U \, \mathrm{d}r \, \mathrm{d}u \Big)^2
	\bigg] \bigg)^{\frac{1}{2}} \bigg)^{2} \\
	&\leq C_{T,\delta} \, M \su \bigg( \sum_{j \in \mathcal{J}_K} \bigg( 
	\mathrm{E}\bigg[ \Big(
	\int_0^1 \int_0^u \big(1 + \big\| \xi(Y_l,j,r) \big\|_H + \big\| Y_l\big\|_H \big)
	\sum_{i \in \mathcal{J}_K} \big( I_{(i,j),l}^Q \big)^2
	\, \mathrm{d}r \, \mathrm{d}u \Big)^2
	\bigg] \bigg)^{\frac{1}{2}} \bigg)^{2} .
	\end{align*}
	Due to $\mathrm{E} \big[ I_{(i_1,j_1),l}^Q I_{(i_2,j_2),l}^Q \big] = \tfrac{1}{2} \eta_{i_1} 
	\eta_{j_1} h_l^2$
	if $i_1 = i_2$ and $j_1=j_2$ and $\mathrm{E} \big[ I_{(i_1,j_1),l}^Q I_{(i_2,j_2),l}^Q \big] = 0$ 
	otherwise, we get
	\allowdisplaybreaks{
		\begin{align}
		&\mathrm{E}\big[\|\bar{Y}^{\MIL}_{m}-\Y_m\|_H^2\big] \nonumber \\
		&\leq
		C_{T,\delta} \, M \su \bigg( \sum_{j \in \mathcal{J}_K} \bigg( 
		\mathrm{E}\bigg[ \Big(
		\int_0^1 \int_0^u \Big(1 + 2 \big\| Y_l \big\|_H + r \Big\| P_N B(Y_l) 
		\sum_{i \in \mathcal{J}_K} \tilde{e}_i I_{(i,j),l}^Q \Big\|_H \Big) \nonumber \\
		&\quad \times \sum_{i \in \mathcal{J}_K} \big( I_{(i,j),l}^Q \big)^2
		\, \mathrm{d}r \, \mathrm{d}u \Big)^2
		\bigg] \bigg)^{\frac{1}{2}} \bigg)^{2} \nonumber \\
		&\leq C_{T,\delta} \, M \su \bigg( \sum_{j \in \mathcal{J}_K} \bigg( 
		\mathrm{E}\bigg[
		\Big( 1 + \big\| Y_l \big\|_H + \big\|(-A)^{-\delta} \big\|_{L(H)} \big\| B(Y_l) \big\|_{L(U,H_{\delta})} \nonumber \\
		&\quad \times \Big( \sum_{i \in \mathcal{J}_K} \big( I_{(i,j),l}^Q \big)^2 \Big)^{\frac{1}{2}} \Big)^2
		\Big( \sum_{i \in \mathcal{J}_K} \big( I_{(i,j),l}^Q \big)^2 \Big)^2
		\bigg] \bigg)^{\frac{1}{2}} \bigg)^{2} \nonumber \\
		&\leq C_{T,\delta} \, M \su \bigg( \sum_{j \in \mathcal{J}_K} \bigg( 
		\mathrm{E}\bigg[ \Big( \Big( \sum_{i \in \mathcal{J}_K} \big( I_{(i,j),l}^Q \big)^2 \Big)^2
		+ \Big( \sum_{i \in \mathcal{J}_K} \big( I_{(i,j),l}^Q \big)^2 \Big)^3 \Big)
		\big( 1 + \big\| Y_l \big\|_H^2 \big) \bigg] \bigg)^{\frac{1}{2}} \bigg)^{2} \nonumber \\
		&\leq C_{T,\delta} \, M \su \bigg( \sum_{j \in \mathcal{J}_K} \bigg( 
		\bigg( \mathrm{E}\bigg[ \Big( \sum_{i \in \mathcal{J}_K} \big( I_{(i,j),l}^Q \big)^2 \Big)^2 \bigg]
		+ \mathrm{E}\bigg[  \Big( \sum_{i \in \mathcal{J}_K} \big( I_{(i,j),l}^Q \big)^2 \Big)^3 \Big) \bigg] \bigg) 
		\big( 1 + \mathrm{E}\big[  \big\| Y_l \big\|_H^2 \big] \big) \bigg)^{\frac{1}{2}} \bigg)^{2} \nonumber \\
		&\leq C_{T,\delta} \, M \su \bigg( \tr Q \bigg( (\tr Q)^2 h^4 + (\tr Q)^3 h^6 
		\bigg)^{\frac{1}{2}} \bigg)^{2} \big( 1 + \mathrm{E}\big[  \big\| Y_l \big\|_{H_{\delta}}^2 \big] \big) \nonumber \\
		&\leq C_{T,\delta} \, (\tr Q)^4 h \su \big( h^2 + \tr Q \, h^4 \big) 
		\big( 1 + \mathrm{E}\big[  \big\| X_0 \big\|_{H_{\delta}}^2 \big] \big) \nonumber \\
		&\leq C_{Q,T,\delta} h^2 . 
		\label{Proof-Main-Thm-Case2-FinEst}
		\end{align}
	}%
	%
	%
	%
	Now, we proceed for both cases similarly.
	A combination of estimates~\eqref{Proof-Main-Thm-Estimate-Mil-Orig},
	\eqref{Milbar} and \eqref{BarY} or 
	\eqref{Proof-Main-Thm-Case2-FinEst} for case~1 and case~2, respectively, 
	with \eqref{ErrorSplit}, and Gronwall's Lemma imply
	\begin{align*}
	\mathrm{E}\big[\|Y_m^{\MIL}-Y_m\|_H^2\big]
	& \leq C_T h \sum_{l=0}^{m-1}   
	\mathrm{E}\big[\|Y_l^{\MIL}-Y_l\|_H^2\big] 
	+ C_{Q,T,\delta} h^2 
	\leq  C_{Q,T,\delta} h^2.
	\end{align*}
	This results in
	\begin{align*}
	\Big(\mathrm{E}\big[\|X_{t_m}-\Y_m\|_H^2\big]\Big)^{\frac{1}{2}} 
	&\leq C_{Q,T,\delta} \Big( \Big( \inf_{i \in \mathcal{I} \setminus
		\mathcal{I}_N} \lambda_i \Big)^{-\gamma}
	+ \Big( \sup_{j \in \mathcal{J} \setminus \mathcal{J}_K}
	\eta_j \Big)^{\alpha} + M^{-\min(2(\gamma-\beta),\gamma)}\Big)
	\end{align*}
	for the overall error.
\end{proof}
In the last part of this section, we prove the estimate that we obtain in the 
case that the stochastic double integrals are approximated additionally, that is,
this estimate incorporates the error of the algorithm
which is used to compute $\bar{I}^Q_{(i,j),l}$, $i,j\in\mathcal{J}_K$,
$l \in \{0,\ldots,M-1\}$.
\begin{remark}\label{Rem:MomentYBar}
	Under the assumption that
	\begin{equation*}
	\sum_{j\in\mathcal{J}_K}
	\bigg(\mathrm{E} \bigg[ \Big( \sum_{i\in\mathcal{J}_K}  
	\big(\bar{I}^Q_{(i,j),t,t+h} \big)^2 \Big)^{\frac{p}{2}} \bigg]\bigg)^{\frac{1}{p}}
	\leq C_Q h
	\end{equation*}
	for $p \in [2,\infty)$ in case of assumption (A5\ref{A5a}) or 
	\begin{equation*}
	\sum_{j \in \mathcal{J}_K} \bigg( 
	\mathrm{E}\bigg[  \Big( \sum_{i \in \mathcal{J}_K} \big( \bar{I}^Q_{(i,j),t,t+h} 
	\big)^2 \Big)^q \bigg] \bigg)^{\frac{1}{2}}
	\leq C_Q h^q
	\end{equation*}
	for $q=2,3$ in case of assumption (A5\ref{A5c}) for any $h>0$ and $t \in [0,T-h]$,
	a statement similar to Lemma~\ref{Proof:Lemma-Moment-nonComm} also holds
	for the process $(\bar{\Y}_l)_{l\in\{0,\ldots,M\}}$
	which includes the approximation of the stochastic double integral,
	i.e., it holds
	\begin{equation*}
	\sup_{m\in\{0,\ldots,M\}}\big(
	\mathrm{E}\big[\| \bar{\Y}_m \|_{H_{\delta}}^p \big] \big)^{\frac{1}{p}}
	\leq C_{Q,T,\delta} \big(1
	+ \big(\mathrm{E} \big[ \|X_0\|_{H_{\delta}}^p \big] \big)^{\frac{1}{p}} \big) ,
	\end{equation*}
	however with the restriction $p=2$ in case of (A5\ref{A5c}).
\end{remark}
\begin{proof}[Proof of Theorem~\ref{Thm:ErrorTotal}]
	From the proof of Theorem~\ref{Error:DFM} we get an estimate for
	$\Big(\mathrm{E}\big[\|X_{t_m}-\Y_m\|_H^2\big]\Big)^{\frac{1}{2}}$.
	It remains to prove the expression for
	the error caused by the approximation 
	of the iterated stochastic integrals, that is,
	\begin{align}\label{Eq:SplitErrorDI}
	\Big( \mathrm{E}\big[\|\Y_m-\bar{\Y}_m\|_H^2\big] \Big)^{\frac{1}{2}}
	\leq \Big( \mathrm{E}\big[\|\Y_m-\Y_{m,\bar{\Y}}\|_H^2\big] \Big)^{\frac{1}{2}}
	+ \Big( \mathrm{E}\big[\|\Y_{m,\bar{\Y}}-\bar{\Y}_m\|_H^2\big] \Big)^{\frac{1}{2}}
	\end{align}
	where
	\begin{align*}
	\Y_{m,\bar{\Y}} &=
	P_N\bigg( e^{At_m}X_0+ \sI e^{A(t_m-t_l)}F(\bar{\Y}_l)\,\mathrm{d}s
	+ \sI e^{A(t_m-t_l)} B(\bar{\Y}_l)\,\mathrm{d}W^K_s \\
	&\quad +\sum_{l=0}^{m-1}\sum_{j\in\mathcal{J}_K}e^{A(t_m-t_l)}
	\Big(B\Big(\bar{\Y}_l+\sum_{i\in\mathcal{J}_K}P_NB(\bar{\Y}_l)
	\tilde{e}_iI^Q_{(i,j),l}\Big)\tilde{e}_j-B(\bar{\Y}_l)\tilde{e}_j\Big)\bigg).
	\end{align*}
	For the terms inside the two integrals, we employ Taylor approximations of 
	first order of the difference operators 
	as in \eqref{Proof-Lem-Case2-Taylor2}
	where $\xi(\Y_l,j,r) = \Y_l + r \sum_{i \in \mathcal{J}_K}
	B(\Y_l) \tilde{e}_i I^Q_{(i,j),l}$ for all $j\in\mathcal{J}_K$, 
	$l \in \{0,\ldots,m-1\}$ and $r \in [0,1]$; below
	$\bar{\xi}(\bar{\Y}_l,j,r)$ is defined analogously.
	This yields 
	\begin{align} \label{Proof-Thm22-Eqn01}
	&\mathrm{E}\big[\|\Y_m-\Y_{m,\bar{\Y}}\|_H^2\big] \nonumber \\
	&= \mathrm{E}\bigg[\Big\| P_N \Big(
	\sI e^{A(t_m-t_l)}\big(F(\Y_l)-F(\bar{\Y}_l)\big)\,\mathrm{d}s 
	+ \sI e^{A(t_m-t_l)}\big( B(\Y_l)-B(\bar{\Y}_l)\big)\,\mathrm{d}W^K_s \nonumber \\
	&\quad
	+\sum_{l=0}^{m-1}\sum_{j\in \mathcal{J}_K} e^{A(t_m-t_l)}
	\Big( B'(\Y_l) \Big( \sum_{i \in \mathcal{J}_K}
	P_N B(\Y_l) \tilde{e}_i I_{(i,j),l}^Q,\tilde{e}_j \Big) \nonumber \\
	&\quad
	- B'(\bar{\Y}_l)\Big(\sum_{i\in\mathcal{J}_K} 
	P_N B(\bar{\Y}_l)\tilde{e}_iI_{(i,j),l}^Q,\tilde{e}_j \Big)\Big)
	+ \sum_{l=0}^{m-1} \sum_{j \in \mathcal{J}_K} e^{A(t_m-t_l)} \nonumber \\
	&\quad
	\Big(\int_0^1 \int_0^u B''(\xi(Y_l,j,r))
	\Big(\sum_{i\in\mathcal{J}_K} P_N B(\Y_l)
	\tilde{e}_i I_{(i,j),l}^Q, \sum_{i \in \mathcal{J}_K} P_N B(\Y_l) \tilde{e}_i I_{(i,j),l}^Q
	\Big) \tilde{e}_j \,\mathrm{d}r \, \mathrm{d}u \nonumber \\
	&\quad
	- \int_0^1 \int_0^u B''(\bar{\xi}(\bar{\Y}_l,j,r))
	\Big( \sum_{i \in \mathcal{J}_K} P_N B(\bar{\Y}_l)
	\tilde{e}_i I_{(i,j),l}^Q, \sum_{i \in \mathcal{J}_K}
	P_N B(\bar{\Y}_l) \tilde{e}_i I_{(i,j),l}^Q
	\Big) \tilde{e}_j \, \mathrm{d}r \, \mathrm{d}u \Big) \Big)
	\Big\|_H^2 \bigg] \nonumber \\
	&\leq
	C_{Q,T,\beta,\gamma} h \sum_{l=0}^{m-1} \mathrm{E} \Big[\big\|\Y_l-\bar{\Y}_l\big\|^2_H\Big]
	+ \mathrm{E} \bigg[\Big\| \sum_{l=0}^{m-1}\sum_{j\in \mathcal{J}_K} e^{A(t_m-t_l)} \nonumber \\
	&\quad
	\Big( \int_0^1 \int_0^u B''(\xi(Y_l,j,r))
	\Big(\sum_{i \in \mathcal{J}_K} P_N B(\Y_l)
	\tilde{e}_i I_{(i,j),l}^Q, \sum_{i \in \mathcal{J}_K} P_N B(\Y_l) \tilde{e}_i I_{(i,j),l}^Q
	\Big) \tilde{e}_j \, \mathrm{d}r \, \mathrm{d}u \Big) \Big) \nonumber \\
	&\quad
	- \int_0^1 \int_0^u B''(\bar{\xi}(\bar{\Y}_l,j,r))
	\Big( \sum_{i \in \mathcal{J}_K} P_N B(\bar{\Y}_l)
	\tilde{e}_i I_{(i,j),l}^Q, \sum_{i \in \mathcal{J}_K}
	P_N B(\bar{\Y}_l) \tilde{e}_i I_{(i,j),l}^Q
	\Big) \tilde{e}_j \, \mathrm{d}r \, \mathrm{d}u \Big) \Big)
	\Big\|_H^2 \bigg] \nonumber \\
	&\leq 
	C_{Q,T,\beta,\gamma} h \sum_{l=0}^{m-1} \mathrm{E} \Big[ \big\| \Y_l - \bar{\Y}_l \big\|^2_H \Big]
	+ C M \sum_{l=0}^{m-1} \bigg( \sum_{j \in \mathcal{J}_K} \bigg(
	\mathrm{E} \bigg[ \Big\| \int_0^1 \int_0^u \nonumber \\
	&\quad
	B''(\xi(Y_l,j,r)) \big( \sum_{i \in \mathcal{J}_K} P_N B(\Y_l) \tilde{e}_i I_{(i,j),l}^Q,
	\sum_{i \in \mathcal{J}_K} P_N B(\Y_l) \tilde{e}_i I_{(i,j),l}^Q \Big)
	\tilde{e}_j 
	\nonumber \\
	&\quad
	- 
	B''(\bar{\xi}(\bar{\Y}_l,j,r))
	\Big( \sum_{i \in \mathcal{J}_K} P_N B(\bar{\Y}_l) \tilde{e}_i
	I_{(i,j),l}^Q, \sum_{i \in \mathcal{J}_K} P_N B(\bar{\Y}_l)
	\tilde{e}_i I_{(i,j),l}^Q \Big)
	\tilde{e}_j \, \mathrm{d}r \, \mathrm{d}u \Big\|^2_H \bigg] \bigg)^{\frac{1}{2}} \bigg)^2
	\end{align}
	where in the second step the computations are the same as in
	\cite[Section 6.3]{MR3320928}, see also~\eqref{Milbar}. 
	This estimate mainly employs the Lipschitz
	continuity of the involved operators. \\ \\
	%
	%
	%
	%
	Case~1: Assume that assumption (A5\ref{A5a}) is fulfilled, i.e., 
	Lemma~\ref{Proof:Lemma-Moment-nonComm} is valid for any $p \geq 2$.
	Then, by the triangle inequality, the norm properties as well as 
	assumption (A3), \eqref{Proof-Thm22-Eqn01} results in
	\begin{align} \label{Proof-Thm2.2-Eqn03}
	&\mathrm{E}\big[\|\Y_m-\Y_{m,\bar{\Y}}\|_H^2\big] \nonumber \\
	&\leq C_{Q,T,\beta,\gamma} h \sum_{l=0}^{m-1}\mathrm{E}\Big[\big\|\Y_l-\bar{\Y}_l\big\|^2_H\Big] \nonumber \\
	&\quad + C M \sum_{l=0}^{m-1}\Big(\sum_{j\in\mathcal{J}_K} 
	\Big(\mathrm{E}\Big[\big\| B''(\xi(Y_l,j))
	\big\|^2_{L^{(2)}(H,L(U,H))}
	\big\|\sum_{i\in\mathcal{J}_K} 
	B(\Y_l)\tilde{e}_iI_{(i,j),l}^Q\big\|^4_H\Big]\Big)^{\frac{1}{2}}\Big)^2 \nonumber \\
	&\quad + C M \sum_{l=0}^{m-1}\Big(\sum_{j\in\mathcal{J}_K}
	\Big( \mathrm{E}\Big[\big\|B''(\bar{\xi}
	(\bar{\Y}_l,j))\big\|^2_{L^{(2)}(H,L(U,H))}\big\|
	\sum_{i\in\mathcal{J}_K} B(\bar{\Y}_l)\tilde{e}_i
	I_{(i,j),l}^Q\big\|^4_H\Big]\Big)^{\frac{1}{2}}\Big)^2 \nonumber \\
	&\leq C_{Q,T,\beta,\gamma} h \sum_{l=0}^{m-1}
	\mathrm{E}\Big[\big\|\Y_l-\bar{\Y}_l\big\|^2_H\Big] \nonumber \\
	&\quad + C M \sum_{l=0}^{m-1}\bigg(\sum_{j\in\mathcal{J}_K}
	\bigg(\mathrm{E}\bigg[  \Big( \sum_{i_1,i_2\in\mathcal{J}_K}
	I_{(i_1,j),l}^QI_{(i_2 ,j),l}^Q
	\langle \tilde{e}_{i_1},\tilde{e}_{i_2}\rangle_U \Big)^2 \big\|
	B(\Y_l)\big\|^4_{L(U,H)}\bigg]\bigg)^{\frac{1}{2}}\bigg)^2 \nonumber \\
	&\quad + C M \sum_{l=0}^{m-1}\bigg(\sum_{j\in\mathcal{J}_K}
	\bigg(\mathrm{E}\bigg[ \Big( \sum_{i_1,i_2\in\mathcal{J}_K} I_{(i_1,j),l}^Q
	I_{(i_2 ,j),l}^Q  \langle \tilde{e}_{i_1},\tilde{e}_{i_2}\rangle_U \Big)^2
	\big\|B(\bar{\Y}_l)\big\|^4_{L(U,H)}\bigg]\bigg)^{\frac{1}{2}}\bigg)^2 \nonumber \\
	&\leq 
	C_{Q,T,\beta,\gamma} h \sum_{l=0}^{m-1}\mathrm{E}\Big[\big\|\Y_l-\bar{\Y}_l\big\|^2_H\Big] \nonumber \\
	&\quad + C M \sum_{l=0}^{m-1}\bigg(\sum_{j\in\mathcal{J}_K}
	\bigg(\mathrm{E}\bigg[\Big( \sum_{i\in\mathcal{J}_K} \big(I_{(i,j),l}^Q\big)^2\Big)^2 \big\|
	B(\Y_l)\big\|^4_{L(U,H)}\bigg]\bigg)^{\frac{1}{2}}\bigg)^2 \nonumber \\
	&\quad + C M \sum_{l=0}^{m-1}\Big(\sum_{j\in\mathcal{J}_K}
	\bigg(\mathrm{E}\bigg[ \Big( \sum_{i\in\mathcal{J}_K} \big(I_{(i,j),l}^Q\big)^2\Big)^2
	\big\|B(\bar{\Y}_l)\big\|^4_{L(U,H)}\bigg]\bigg)^{\frac{1}{2}}\bigg)^2.
	\end{align}
	This expression can further be simplified by
	the properties of $I_{(i,j),l}^Q$ 
	for  $l\in\{0,\ldots,m-1\}$, $m\in\{1,\ldots,M\}$, $i,j\in\mathcal{J}_K$, $M,K\in\mathbb{N}$ 
	and assumption (A3). Furthermore, (A5\ref{A5a}),
	Lemma~\ref{Proof:Lemma-Moment-nonComm} and
	Remark~\ref{Rem:MomentYBar} imply
	\begin{align*}
	&\mathrm{E}\big[\|\Y_m-\Y_{m,\bar{\Y}}\|_H^2\big] \nonumber \\
	&\leq  
	C_{Q,T,\beta,\gamma} h \sum_{l=0}^{m-1}\mathrm{E}\Big[\big\|\Y_l-\bar{\Y}_l\big\|^2_H\Big]
	+ C_Q M \sum_{l=0}^{m-1}\Big(\Big(h^4\,
	\Big(\mathrm{E}\big[\|B(\Y_l)\|^4_{L(U,H)}\big]
	+ \mathrm{E}\big[\|B(\bar{\Y}_l)\|^4_{L(U,H)}\big]\Big)\Big)^{\frac{1}{2}}\Big)^2 \\
	&\leq C_{Q,T,\beta,\gamma} h \sum_{l=0}^{m-1}\mathrm{E}\Big[\big\|\Y_l-\bar{\Y}_l\big\|^2_H\Big]
	+ C_{Q,T,\beta,\gamma}  \sum_{l=0}^{m-1}h^3 \\
	&\leq 
	C_{Q,T,\beta,\gamma} h \sum_{l=0}^{m-1} \mathrm{E}\Big[\big\|\Y_l-\bar{\Y}_l\big\|^2_H\Big] 
	+ C_{Q,T,\beta,\gamma} h^2.
	\end{align*}
	%
	%
	%
	%
	Case~2: If assumption (A5\ref{A5c}) is fulfilled, then Lemma~\ref{Proof:Lemma-Moment-nonComm} 
	is valid for $p = 2$. By applying the triangle inequality, we get
	analogously to case~2 in the proof of Theorem~\ref{Error:DFM} that
	\begin{align*}
	&\mathrm{E}\big[\|\Y_m-\Y_{m,\bar{\Y}}\|_H^2\big] \\
	&\leq 
	C_{Q,T,\beta,\gamma} \, h \sum_{l=0}^{m-1} \mathrm{E} \Big[ \big\| \Y_l - \bar{\Y}_l \big\|^2_H \Big]
	+ C \, M \sum_{l=0}^{m-1} \bigg( \sum_{j \in \mathcal{J}_K} 
	\bigg( \mathrm{E} \bigg[ \Big\| \int_0^1 \int_0^u \\
	&\quad
	B''(\xi(Y_l,j,r)) \big( \sum_{i \in \mathcal{J}_K} P_N B(\Y_l) \tilde{e}_i I_{(i,j),l}^Q,
	\sum_{i \in \mathcal{J}_K} P_N B(\Y_l) \tilde{e}_i I_{(i,j),l}^Q \Big)
	\tilde{e}_j  \, \mathrm{d}r \, \mathrm{d}u \Big\|^2_H \bigg] \bigg)^{\frac{1}{2}}
	\\
	&\quad + \sum_{j \in \mathcal{J}_K} 
	\bigg( \mathrm{E} \bigg[ \Big\| \int_0^1 \int_0^u \\
	&\quad
	B''(\bar{\xi}(\bar{\Y}_l,j,r))
	\Big( \sum_{i \in \mathcal{J}_K} P_N B(\bar{\Y}_l) \tilde{e}_i
	I_{(i,j),l}^Q, \sum_{i \in \mathcal{J}_K} P_N B(\bar{\Y}_l)
	\tilde{e}_i I_{(i,j),l}^Q \Big)
	\tilde{e}_j \, \mathrm{d}r \, \mathrm{d}u \Big\|^2_H \bigg] \bigg)^{\frac{1}{2}} 
	\bigg)^2 \\
	&\leq 
	C_{Q,T,\beta,\gamma} \, h \sum_{l=0}^{m-1} \mathrm{E} \Big[ \big\| \Y_l - \bar{\Y}_l \big\|^2_H \Big] \\
	&\quad
	+ C \, M \su \bigg( \sum_{j \in \mathcal{J}_K} \bigg( 
	\mathrm{E}\bigg[ \Big(
	\int_0^1 \int_0^u \big\| B''(\xi(Y_l,j,r)) \big( P_N B(Y_l), P_N B(Y_l) \big) \big\|_{L^{(2)}(U,L(U,H))} \\
	&\quad \times \Big\| \sum_{i \in \mathcal{J}_K}  \tilde{e}_i I_{(i,j),l}^Q \Big\|_U^2 
	\| \tilde{e}_j \|_U \, \mathrm{d}r \, \mathrm{d}u \Big)^2
	\bigg] \bigg)^{\frac{1}{2}} \\
	&\quad
	+ \sum_{j \in \mathcal{J}_K} \bigg( \mathrm{E}\bigg[ \Big(
	\int_0^1 \int_0^u \big\| B''(\bar{\xi}(\bar{\Y}_l,j,r)) 
	\big( P_N B(\bar{\Y}_l), P_N B(\bar{\Y}_l) \big) \big\|_{L^{(2)}(U,L(U,H))} \\
	&\quad \times \Big\| \sum_{i \in \mathcal{J}_K}  \tilde{e}_i I_{(i,j),l}^Q \Big\|_U^2 
	\| \tilde{e}_j \|_U \, \mathrm{d}r \, \mathrm{d}u \Big)^2
	\bigg] \bigg)^{\frac{1}{2}}  
	\bigg)^{2} \\
	&\leq 
	C_{Q,T,\beta,\gamma} \, h \sum_{l=0}^{m-1} \mathrm{E} \Big[ \big\| \Y_l - \bar{\Y}_l \big\|^2_H \Big] \\
	&\quad + C \, M \su \bigg( \sum_{j \in \mathcal{J}_K} \bigg( 
	\mathrm{E}\bigg[ \Big(
	\int_0^1 \int_0^u \big(1 + \big\| \xi(\Y_l,j,r) \big\|_H + \big\| \Y_l\big\|_H \big)
	\sum_{i \in \mathcal{J}_K} \big( I_{(i,j),l}^Q \big)^2
	\, \mathrm{d}r \, \mathrm{d}u \Big)^2
	\bigg] \bigg)^{\frac{1}{2}} \\
	&\quad
	+ \sum_{j \in \mathcal{J}_K} \bigg( 
	\mathrm{E}\bigg[ \Big(
	\int_0^1 \int_0^u \big(1 + \big\| \bar{\xi}(\bar{\Y}_l,j,r) \big\|_H + \big\| \bar{\Y}_l \big\|_H \big)
	\sum_{i \in \mathcal{J}_K} \big( I_{(i,j),l}^Q \big)^2
	\, \mathrm{d}r \, \mathrm{d}u \Big)^2
	\bigg] \bigg)^{\frac{1}{2}}
	\bigg)^{2} .
	\end{align*}
	Making use of the distributional characteristics of $I_{(i,j),l}^Q$, we get
	\begin{align} \label{Proof-Thm22-Case2-FinEst}
	&\mathrm{E}\big[\|\Y_m-\Y_{m,\bar{\Y}}\|_H^2\big] \nonumber \\
	&\leq
	C_{Q,T,\beta,\gamma} \, h \sum_{l=0}^{m-1} \mathrm{E} \Big[ \big\| \Y_l - \bar{\Y}_l \big\|^2_H \Big] \nonumber \\
	&\quad + C \, M \su \bigg( 
	\sum_{j \in \mathcal{J}_K} \bigg( \mathrm{E}\bigg[ \Big(
	\int_0^1 \int_0^u \Big(1 + 2 \big\| \Y_l \big\|_H + r \Big\| P_N B(\Y_l) 
	\sum_{i \in \mathcal{J}_K} \tilde{e}_i I_{(i,j),l}^Q \Big\|_H \Big) \nonumber \\
	&\quad \times \sum_{i \in \mathcal{J}_K} \big( I_{(i,j),l}^Q \big)^2
	\, \mathrm{d}r \, \mathrm{d}u \Big)^2 \bigg] \bigg)^{\frac{1}{2}} \nonumber \\
	&\quad + \sum_{j \in \mathcal{J}_K} \bigg( \mathrm{E}\bigg[ \Big(
	\int_0^1 \int_0^u \Big(1 + 2 \big\| \bar{\Y}_l \big\|_H + r \Big\| P_N B(\bar{\Y}_l) 
	\sum_{i \in \mathcal{J}_K} \tilde{e}_i I_{(i,j),l}^Q \Big\|_H \Big) \nonumber \\
	&\quad \times \sum_{i \in \mathcal{J}_K} \big( I_{(i,j),l}^Q \big)^2
	\, \mathrm{d}r \, \mathrm{d}u \Big)^2 \bigg] \bigg)^{\frac{1}{2}} \bigg)^{2} \nonumber \\
	&\leq 
	C_{Q,T,\beta,\gamma} \, h \sum_{l=0}^{m-1} \mathrm{E} \Big[ \big\| \Y_l - \bar{\Y}_l \big\|^2_H \Big] 
	+ C_{Q,T,\delta} \, M \nonumber \\
	&\quad \times \su \bigg( 
	\sum_{j \in \mathcal{J}_K} \bigg( 
	\bigg( \mathrm{E}\bigg[ \Big( \sum_{i \in \mathcal{J}_K} \big( I_{(i,j),l}^Q \big)^2 \Big)^2 \bigg]
	+ \mathrm{E}\bigg[  \Big( \sum_{i \in \mathcal{J}_K} \big( I_{(i,j),l}^Q \big)^2 \Big)^3 \Big) \bigg] \bigg) 
	\big( 1 + \mathrm{E}\big[  \big\| \Y_l \big\|_H^2 \big] \big) \bigg)^{\frac{1}{2}} \nonumber \\
	&\quad + \sum_{j \in \mathcal{J}_K} \bigg( 
	\bigg( \mathrm{E}\bigg[ \Big( \sum_{i \in \mathcal{J}_K} \big( I_{(i,j),l}^Q \big)^2 \Big)^2 \bigg]
	+ \mathrm{E}\bigg[  \Big( \sum_{i \in \mathcal{J}_K} \big( I_{(i,j),l}^Q \big)^2 \Big)^3 \Big) \bigg] \bigg) 
	\big( 1 + \mathrm{E}\big[ \big\| \bar{\Y}_l \big\|_H^2 \big] \big) \bigg)^{\frac{1}{2}} \bigg)^{2} \nonumber \\
	&\leq 
	C_{Q,T,\beta,\gamma} \, h \sum_{l=0}^{m-1} \mathrm{E} \Big[ \big\| \Y_l - \bar{\Y}_l \big\|^2_H \Big] \nonumber \\
	&\quad + C_{Q,T,\delta} \, M \su \bigg( \tr Q \bigg( (\tr Q)^2 h^4 + (\tr Q)^3 h^6 
	\bigg)^{\frac{1}{2}} \bigg)^{2}
	\big( 1 + \mathrm{E}\big[ \big\| Y_l \big\|_{H_{\delta}}^2 \big] 
	+ \mathrm{E}\big[ \big\| \bar{\Y}_l \big\|_{H_{\delta}}^2 \big] \big) \nonumber \\
	&\leq 
	C_{Q,T,\beta,\gamma} \, h \sum_{l=0}^{m-1} \mathrm{E} \Big[ \big\| \Y_l - \bar{\Y}_l \big\|^2_H \Big]
	+ C_{Q,T,\delta} \, (\tr Q)^4 h \su \big( h^2 + \tr Q \, h^4 \big) 
	\big( 1 + \mathrm{E}\big[  \big\| X_0 \big\|_{H_{\delta}}^2 \big] \big) \nonumber \\
	&\leq 
	C_{Q,T,\beta,\gamma} \, h \sum_{l=0}^{m-1} \mathrm{E} \Big[ \big\| \Y_l - \bar{\Y}_l \big\|^2_H \Big]
	+ C_{Q,T,\delta} h^2 .
	\end{align}
	%
	%
	Summarizing, we have $\mathrm{E}\big[\|\Y_m-\Y_{m,\bar{\Y}}\|_H^2\big]
	\leq C_{Q,T,\beta,\gamma} \, h \sum_{l=0}^{m-1} \mathrm{E} \Big[ \big\| \Y_l - \bar{\Y}_l\big\|^2_H \Big]
	+ C_{Q,T,\delta} h^2$ for both cases. \\ \\
	Finally, we analyze the second term in~\eqref{Eq:SplitErrorDI}.
	We basically employ the same techniques as for the previous term.
	At first, we replace the difference operator by a first order
	Taylor expansion.
	\begin{align*}
	&\mathrm{E}\big[\|\Y_{m,\bar{\Y}}-\bar{\Y}_m\|_H^2\big] \\
	&= \mathrm{E}\bigg[\Big\| P_N 
	\sum_{l=0}^{m-1}\sum_{j\in\mathcal{J}_K}e^{A(t_m-t_l)}
	\Big(\Big(B\Big(\bar{\Y}_l+\sum_{i\in\mathcal{J}_K}P_NB(\bar{\Y}_l)
	\tilde{e}_iI^Q_{(i,j),l}\Big)\tilde{e}_j-B(\bar{\Y}_l)\tilde{e}_j\Big)\\
	&\quad - \Big(B\Big(\bar{\Y}_l+\sum_{i\in\mathcal{J}_K}P_NB(\bar{\Y}_l)
	\tilde{e}_i\bar{I}^Q_{(i,j),l}\Big)\tilde{e}_j
	-B(\bar{\Y}_l)\tilde{e}_j\Big)\Big)\Big\|_H^2\bigg]\\
	&=
	\mathrm{E}\bigg[\Big\|P_N\Big(\sum_{l=0}^{m-1}
	\sum_{j\in\mathcal{J}_K}e^{A(t_m-t_l)}
	\Big(B'(\bar{\Y}_l)\Big(\sum_{i\in\mathcal{J}_K}
	P_N B(\bar{\Y}_l)\tilde{e}_iI_{(i,j),l}^Q,\tilde{e}_j \Big)\\
	&\quad +
	\int_0^1 \int_0^u B''(\bar{\xi}(\bar{Y}_l,j,r)) 
	\Big(\sum_{i\in\mathcal{J}_K} P_NB(\bar{\Y}_l)
	\tilde{e}_iI_{(i,j),l}^Q,\sum_{i\in\mathcal{J}_K} 
	P_N B(\bar{\Y}_l)\tilde{e}_iI_{(i,j),l}^Q
	\Big)\tilde{e}_j \,\mathrm{d}r\,\mathrm{d}u\\
	&\quad
	- B'(\bar{\Y}_l)\Big(\sum_{i\in\mathcal{J}_K} P_NB(\bar{\Y}_l)
	\tilde{e}_i\bar{I}_{(i,j),l}^Q,\tilde{e}_j\Big) \\
	&\quad -
	\int_0^1 \int_0^u B''(\bar{\bar{\xi}}(\bar{\Y}_l,j,r)) 
	\Big(\sum_{i\in\mathcal{J}_K}
	P_N B(\bar{\Y}_l)\tilde{e}_i\bar{I}_{(i,j),l}^Q,\sum_{i\in\mathcal{J}_K} 
	P_N B(\bar{\Y}_l ) \tilde{e}_i
	\bar{I}_{(i,j),l}^Q\Big)\tilde{e}_j\,\mathrm{d}r
	\,\mathrm{d}u \Big) \Big) \Big\|_H^2 \bigg].
	\end{align*}
	As above, we obtain for the terms involving the second derivative
	\begin{align} \label{Proof-Thm2.2-Eqn02}
	&\mathrm{E}\big[\|\Y_{m,\bar{\Y}}-\bar{\Y}_m\|_H^2\big] \nonumber \\
	&\leq C \mathrm{E}\bigg[\Big\|\sum_{l=0}^{m-1}e^{A(t_m-t_l)}
	\Big(\int_{t_l}^{t_{l+1}}B'(\bar{\Y}_l)
	\Big(\int_{t_l}^s P_NB(\bar{\Y}_l)\, 
	\mathrm{d}W_r^K\Big)\, \mathrm{d}W^K_s \nonumber \\
	&\quad
	-\sum_{i,j\in\mathcal{J}_K} \bar{I}_{(i,j),l}^Q B'(\bar{\Y}_l) 
	(P_N B(\bar{\Y}_l)\tilde{e}_i,\tilde{e}_j)\Big)
	\Big\|_H^2\bigg] \nonumber \\
	&\quad + C \mathrm{E} \bigg[ \Big\| \sum_{l=0}^{m-1} e^{A(t_m-t_l)}
	\sum_{j\in\mathcal{J}_K} \nonumber \\
	&\quad \Big( \int_0^1\int_0^u
	B''(\bar{\xi}(\bar{\Y}_l,j,r)) 
	\Big(\sum_{i\in\mathcal{J}_K} P_NB(\bar{\Y}_l)
	\tilde{e}_iI_{(i,j),l}^Q,\sum_{i\in\mathcal{J}_K}
	P_NB(\bar{\Y}_l)\tilde{e}_iI_{(i,j),l}^Q\Big)\tilde{e}_j 
	\,\mathrm{d}r\,\mathrm{d}u \nonumber \\
	&\quad
	- \int_0^1\int_0^u B''(\bar{\bar{\xi}}(\bar{\Y}_l,j,r)) 
	\Big(\sum_{i\in\mathcal{J}_K} P_NB(\bar{\Y}_l)
	\tilde{e}_i\bar{I}_{(i,j),l}^Q,
	\sum_{i\in\mathcal{J}_K} P_NB(\bar{\Y}_l)\tilde{e}_i
	\bar{I}_{(i,j),l}^Q\Big)\tilde{e}_j\,\mathrm{d}r\,\mathrm{d}u \Big) \Big\|_H^2\bigg] \nonumber \\
	&\leq C \sum_{l=0}^{m-1}\mathrm{E}\bigg[\Big\|\int_{t_l}^{t_{l+1}}
	B'(\bar{\Y}_l)\Big(\int_{t_l}^s P_NB(\bar{\Y}_l)\,
	\mathrm{d}W_r^K\Big)\, \mathrm{d}W^K_s  -
	\sum_{i,j\in\mathcal{J}_K} \bar{I}_{(i,j),l}^Q
	B'(\bar{\Y}_l) (P_NB(\bar{\Y}_l)\tilde{e}_i,\tilde{e}_j)\Big\|_H^2\bigg] \nonumber \\
	&\quad + C \Big( \sum_{l=0}^{m-1} \Big( \mathrm{E} \bigg[ \Big\| e^{A(t_m-t_l)}
	\sum_{j\in\mathcal{J}_K} \nonumber \\
	&\quad \int_0^1 \int_0^u
	B''(\bar{\xi}(\bar{\Y}_l,j,r)) 
	\Big(\sum_{i\in\mathcal{J}_K} P_NB(\bar{\Y}_l)
	\tilde{e}_iI_{(i,j),l}^Q,\sum_{i\in\mathcal{J}_K}
	P_NB(\bar{\Y}_l)\tilde{e}_iI_{(i,j),l}^Q\Big)\tilde{e}_j
	\,\mathrm{d}r\,\mathrm{d}u \Big\|_H^2\bigg] \Big)^{\frac{1}{2}} \nonumber \\
	&\quad
	+ \sum_{l=0}^{m-1} \Big( \mathrm{E} \bigg[ \Big\| e^{A(t_m-t_l)}
	\sum_{j\in\mathcal{J}_K} \nonumber \\
	&\quad \int_0^1 \int_0^u B''(\bar{\bar{\xi}}(\bar{\Y}_l,j,r)) 
	\Big(\sum_{i\in\mathcal{J}_K} P_N B(\bar{\Y}_l)
	\tilde{e}_i\bar{I}_{(i,j),l}^Q,
	\sum_{i\in\mathcal{J}_K} P_NB(\bar{\Y}_l)\tilde{e}_i
	\bar{I}_{(i,j),l}^Q \Big) \tilde{e}_j \, \mathrm{d}r \, \mathrm{d}u \Big\|_H^2\bigg] 
	\Big)^{\frac{1}{2}} \Big)^2 .
	\end{align}
	The first term is the error that results from the approximation
	of the iterated stochastic integral. Depending on the choice of the scheme, 
	this error estimate may differ. Assumption~\eqref{Cond1:DI} states
	that
	\begin{align*}
	&\bigg( \mathrm{E}\bigg[\Big\|\int_{t_l}^{t_{l+1}}B'(\bar{\Y}_l)
	\Big(\int_{t_l}^s P_N B(\bar{\Y}_l) \, \mathrm{d}W_r^K\Big) \, \mathrm{d}W^K_s
	- \sum_{i,j\in\mathcal{J}_K} \bar{I}_{(i,j),l}^Q B'(\bar{\Y}_l)
	(P_N B(\bar{\Y}_l)\tilde{e}_i, \tilde{e}_j)\Big\|_H^2\bigg] \bigg)^{\frac{1}{2}}
	\leq \mathcal{E}(M,K)
	\end{align*}
	for all $l\in\{0,\ldots,m-1\}$, $m\in\{1,\ldots,M\}$, $h>0$
	and $M,K\in\mathbb{N}$. \\ \\
	%
	%
	Case~1: Assume that assumption (A5\ref{A5a}) is fulfilled, i.e., 
	Lemma~\ref{Proof:Lemma-Moment-nonComm} is valid for any $p \geq 2$.
	Analogously to the calculations in \eqref{Proof-Thm2.2-Eqn03}, we get 
	for \eqref{Proof-Thm2.2-Eqn02}
	\begin{align*}
	\mathrm{E}\big[\|\Y_{m,\bar{\Y}}-\bar{\Y}_m\|_H^2\big]
	&\leq C \sum_{l=0}^{m-1} \mathcal{E}(M,K)^2
	+ C M \sum_{l=0}^{m-1}\bigg(\sum_{j\in\mathcal{J}_K}
	\bigg(\mathrm{E}\bigg[\Big( \sum_{i\in\mathcal{J}_K} \big(I_{(i,j),l}^Q\big)^2\Big)^2 \big\|
	B(\bar{\Y}_l)\big\|^4_{L(U,H)}\bigg]\bigg)^{\frac{1}{2}}\bigg)^2 \\
	&\quad + C M \sum_{l=0}^{m-1}\Big(\sum_{j\in\mathcal{J}_K}
	\bigg(\mathrm{E}\bigg[ \Big( \sum_{i\in\mathcal{J}_K} \big(\bar{I}_{(i,j),l}^Q\big)^2\Big)^2
	\big\|B(\bar{\Y}_l)\big\|^4_{L(U,H)}\bigg]\bigg)^{\frac{1}{2}}\bigg)^2.
	\end{align*}
	By assumption~\eqref{Cond2:DI-A5a} and
	the properties of $I_{(i,j),l}^Q$
	as well as Remark~\ref{Rem:MomentYBar}, we obtain
	\begin{align*}
	\mathrm{E}\big[\|\Y_{m,\bar{\Y}}-\bar{\Y}_m\|_H^2\big]
	&\leq C \sum_{l=0}^{m-1} \mathcal{E}(M,K)^2 \\  
	&\quad +C M \sum_{l=0}^{m-1}\bigg(\sum_{j\in\mathcal{J}_K} 
	\bigg(\mathrm{E}\Big[ \Big( \sum_{i\in\mathcal{J}_K} \big(I_{(i,j),l}^Q\big)^2\Big)^2
	\Big]
	\mathrm{E}\Big[\big\|
	B(\bar{\Y}_l)\big\|^4_{L(U,H)}\Big]\bigg)^{\frac{1}{2}}\bigg)^2 \\
	&\quad + C M \sum_{l=0}^{m-1} \bigg( \sum_{j\in\mathcal{J}_K} 
	\bigg( \mathrm{E} \Big[ \Big( \sum_{i\in\mathcal{J}_K} \big(\bar{I}_{(i,j),l}^Q\big)^2\Big)^2
	\Big]
	\mathrm{E}\Big[\big\|B(\bar{\Y}_l)\big\|^4_{L(U,H)}\Big]\bigg)^{\frac{1}{2}}\bigg)^2\\
	&\leq C \sum_{l=0}^{m-1} \mathcal{E}(M,K)^2
	+ C_Q M \sum_{l=0}^{m-1}\Big(h^2\Big(
	\mathrm{E}\big[\|B(\bar{\Y}_l)\|^4_{L(U,H)}\big] \Big)^{\frac{1}{2}}\Big)^2 \nonumber \\
	&
	\leq C M \mathcal{E}(M,K)^2 + C_{Q,T,\delta} h^2,
	\end{align*}
	which completes the proof for case~1. \\ \\
	%
	%
	%
	%
	Case~2: If assumption (A5\ref{A5c}) is fulfilled, then 
	Lemma~\ref{Proof:Lemma-Moment-nonComm} is valid for $p = 2$. 
	Analogously to the computations in \eqref{Proof-Thm22-Case2-FinEst},
	we get with assumption~\eqref{Cond2:DI-A5c} that
	\begin{align*}
	&\mathrm{E}\big[\|\Y_{m,\bar{\Y}}-\bar{\Y}_m\|_H^2\big] \\
	&\leq C \sum_{l=0}^{m-1} \mathcal{E}(M,K)^2 
	+ C \, M \su \bigg( 
	\sum_{j \in \mathcal{J}_K} 
	\bigg( \mathrm{E}\bigg[ \Big( \sum_{i \in \mathcal{J}_K} \big( I_{(i,j),l}^Q \big)^2 \Big)^2 \bigg]
	+ \mathrm{E}\bigg[  \Big( \sum_{i \in \mathcal{J}_K} \big( I_{(i,j),l}^Q \big)^2 \Big)^3 \Big) \bigg] 
	\bigg)^{\frac{1}{2}} \nonumber \\
	&\quad + \sum_{j \in \mathcal{J}_K}
	\bigg( \mathrm{E}\bigg[ \Big( \sum_{i \in \mathcal{J}_K} \big( \bar{I}_{(i,j),l}^Q \big)^2 \Big)^2 \bigg]
	+ \mathrm{E}\bigg[  \Big( \sum_{i \in \mathcal{J}_K} \big( \bar{I}_{(i,j),l}^Q \big)^2 \Big)^3 \Big) \bigg] 
	\bigg)^{\frac{1}{2}} \bigg)^{2} \big( 1 + \mathrm{E}\big[ \big\| \bar{\Y}_l \big\|_H^2 \big] \big) \nonumber \\
	&\leq C \sum_{l=0}^{m-1} \mathcal{E}(M,K)^2
	+ C_{T} \, (\tr Q)^4 h \su \big( h^2 + \tr Q \, h^4 \big) 
	\big( 1 + \mathrm{E}\big[  \big\| X_0 \big\|_{H_{\delta}}^2 \big] \big) \nonumber \\
	&\leq C M \mathcal{E}(M,K)^2 + C_{Q,T,\delta} h^2.
	\end{align*}
	This proves the statement for case~2.
\end{proof}
%
%
%
%
%
\textbf{Acknowledgements}
Funding and support by
the Graduate School for Computing in Medicine and Life Sciences funded by
Germany's Excellence Initiative [DFG GSC 235/2] and in addition
by the Cluster of Excellence ``The Future Ocean'' is gratefully acknowledged. 
``The Future Ocean'' is funded within the framework
of the Excellence Initiative by the Deutsche Forschungsgemeinschaft (DFG) 
on behalf of the German federal and state governments.
%
%
\bibliographystyle{abbrv}
\bibliography{Diss}
\end{document}